\PassOptionsToPackage{square,comma,numbers,sort&compress}{natbib}  
\documentclass[preprint,12pt]{elsarticle}

\usepackage{natbib}
\setcitestyle{sort&compress}

\usepackage{lineno,hyperref}
\modulolinenumbers[5]

\usepackage{amsmath}
\usepackage{amssymb}
\usepackage{amsthm}

\allowdisplaybreaks

\usepackage[a4paper]{geometry}
\usepackage{graphicx}
\usepackage{float}
\usepackage{caption}
\usepackage{subcaption}

\usepackage{microtype}
\usepackage{array}

\newcommand{\pH}[1]{\frac{\partial H(x,y,z,\theta)}{\partial #1}}
\newcommand{\pE}[1]{\frac{\partial E(x,y,z,\theta)}{\partial #1}}
\newcommand{\pHn}[2]{\frac{\partial^{#1} H}{\partial {#2}^{#1}}}
\newcommand{\pEn}[2]{\frac{\partial^{#1} E}{\partial {#2}^{#1}}}
\newcommand{\pJn}[1]{\frac{\partial^{#1} J}{\partial \theta^{#1}}}

\newcommand{\curl}[1]{(\nabla\times)^{#1}}
\newcommand{\powfact}[2]{\frac{{#1}^{#2}}{(#2)!}}
\newcommand{\pt}[3]{\frac{ {\partial}^{#3} {#1}}{\partial {#2}^{#3}}}

\newcommand{\ppf}[2]{{{\partial} {#1}}/{{\partial} {#2}}}

\newtheorem{thm}{Theorem}
\newtheorem{lemma}[thm]{Lemma}
\newtheorem{corollary}[thm]{Corollary}

\begin{document}

\begin{frontmatter}

\title{Analytical Solutions of 3D Maxwell’s Equations via Infinite-Order Curl-Operator Expansions}

\author[Independent Researcher]{David Wei Ge\texorpdfstring{\corref{cor1}}{}}
\ead{gexiaobao.usa@gmail.com}
\cortext[cor1]{Corresponding author}
\address[Independent Researcher]{Retired, formerly with Microsoft, Redmond, USA}

\begin{abstract}
Solving Maxwell's equations to obtain explicit analytical representations of electromagnetic fields in open-space settings with general initial conditions and source terms remains a fundamental challenge. We address this problem through an operator-based construction, developing an infinite-order curl-operator expansion that yields analytical solutions as function-to-function mappings from initial data and source terms to electromagnetic fields. The framework is applied to representative cases, including Gaussian initial conditions, Gaussian sources, harmonic sources, and Ricker wavelet sources. In these applications, a family of deformed sine and cosine functions emerges naturally, enabling explicit analytical representations of the solutions. The analysis further shows that these deformed trigonometric functions are intrinsic to the general solution formulas, with their structure determined by the initial conditions and source terms. The results are compared with numerical simulations based on the FDTD method and further applied to a systematic study of mesh refinement effects. This study shows that, for the problems considered, the error measures approach limiting values as the mesh is refined, indicating that further step size reduction alone may produce little additional improvement once the mesh is sufficiently fine. The results demonstrate that the analytical solutions provide theoretical insight into electromagnetic field behavior while serving as practical benchmark solutions for computational electromagnetics.
\end{abstract}


\end{frontmatter}

\section{Introduction}
Electromagnetic phenomena in three dimensions are governed by Maxwell’s equations. Although computational techniques, such as the finite-difference time-domain (FDTD) method, are the dominant tools for simulating 3D problems \cite{fdtd:Yee, fdtd:meep, fdtd:Schneider, maxn:Neumann, maxn:Zhou, maxn:Fedeli, fdtd1d:Mastryukov, fdtd:SchneiderDispersion, maxn:Kohlmann}, they exhibit intrinsic errors arising from truncation of higher-order derivatives and discretization, including numerical dispersion, grid anisotropy, and spurious nonphysical artifacts such as apparent superluminal signal propagation \cite{fdtd:SchneiderDispersion, fdtd:Manry}. Despite substantial progress in both numerical and analytical approaches, constructing analytical time-domain solutions in algebraic expressions for general 3D initial and source configurations remains challenging; most existing methods rely on potential formulations with gauge freedom or on integral representations that are often analytically intractable \cite{maxs:Oleg, Lipan:24, LeBoudec2024, YANG2023126678}. In this work, we address this challenge through a systematic operator-based construction, developing an infinite-order curl-operator expansion that eliminates truncation error and yields compact analytical solutions, thereby providing a general solution mapping framework for time-domain Maxwell problems.

Central to this approach is the distinction between a mapping—the correspondence between input and output functions—and a transformation, the explicit formula that realizes this correspondence. Classical Taylor series define self-mappings on analytic functions. The theorems presented here extend this concept by establishing mappings from analytic inputs to distinct outputs, enabling exact function-to-function transformations that define a general solution mapping for Maxwell’s equations in the time domain. The one-dimensional formulation of this approach is reported in \cite{GE2026100688}; the present work extends these results to three dimensions.

From this perspective, the resulting formulations offer several advantages over conventional numerical schemes: (1) each point in space–time can be evaluated independently, without requiring a computational domain or artificial boundaries; (2) the underlying physical properties of the equations are preserved, avoiding numerical artifacts such as dispersion and grid anisotropy; and (3) the analytical structure provides insights into field behavior that are often difficult to extract from numerical data alone. These analytical solutions are intended to complement, rather than replace, numerical methods such as the FDTD method and finite element time-domain method, serving as rigorous reference fields for verification and analysis.

The remainder of this paper is organized as follows. Section \ref{sect.problem} formulates the problem and defines the solution mappings. Section \ref{sect.solution} presents the analytical solution formulas based on the Time--Space theorem (\ref{thm.ts}), and Section \ref{sect.deformedFunction} introduces the resulting deformed trigonometric functions. Section \ref{sect.relation.green} discusses the relationship with existing analytical methods. Representative case studies are presented in Section \ref{sect.caseStudies}, with the required formulas proved in \ref{appendix.cases}. Section \ref{sect.numeric} compares the analytical reference fields with FDTD results.

\section{The Problem}
\label{sect.problem}
The following notation is used throughout the paper. Capital letters denote three-dimensional vector functions, and lowercase letters denote scalars. A lowercase variable with an over-arrow denotes a three-dimensional vector variable. The symbol 0 represents either a scalar or a vector, depending on the context.

\subsection{Definition of the problem} Consider the following three-dimensional Maxwell’s equations in open space.
\begin{align}
	\label{3d.1}
	\pH{\theta} &= -\frac{1}{\eta} \nabla \times E(x,y,z,\theta)
\\
	\label{3d.2}
	\pE{\theta} &= \eta \nabla \times H(x,y,z,\theta) - \eta J(x,y,z,\theta)
\\
	\label{3d.3}
	H(x,y,z,0) &= F_h(x,y,z)
\\
	\label{3d.4}
	E(x,y,z,0) &= F_e(x,y,z)
\end{align}

where $\eta = \sqrt{{\mu}/{\epsilon}}$, $\theta = ct$, $c = {1}/{\sqrt{\epsilon \mu}}$, and
\begin{equation*}
	\begin{gathered}
		x,y,z,t \in \mathbb{R}; \epsilon, \mu \in \mathbb{R}_{>0};
		F_h, F_e \in C^{\infty}(\mathbb{R}^3,\mathbb{R}^3); J, H, E \in C^{\infty}(\mathbb{R}^4, \mathbb{R}^3) ,
	\end{gathered}
\end{equation*}
where $\epsilon$ and $\mu$ are constants, $x,y,z$ are space coordinates, $t$ is time, function $J (x,y,z,\theta)$ is the source term, functions $F_h(x,y,z)$ and $F_e(x,y,z)$ are the initial values. The equations are with dimensionless units \cite{maxn:Kohlmann, fdtd:meep}. For convenience, $\theta$ is referred to as time.
The initial values satisfy the following conditions.
\begin{equation}
	\label{3d.5}
	\nabla \cdot F_h(x,y,z) = 0
\end{equation}
\begin{equation}
	\label{3d.6}
	\nabla \cdot F_e(x,y,z) = 0
\end{equation}

\subsection{Definition of the solution}
A function pair \{ $ H(x,y,z, \theta), E(x,y,z, \theta) $ \} is a solution to the problem if and only if it satisfies (\ref{3d.1}), (\ref{3d.2}), (\ref{3d.3}), and (\ref{3d.4}). 

\subsection{Solution as a function mapping}
Taking functions $J (x,y,z,\theta)$, $F_h (x,y,z)$ and $F_e (x,y,z)$ as the input, function pair \{ $ H(x,y,z, \theta), E(x,y,z, \theta) $ \} as the output, the definitions of the problem and its solution define a function-to-function mapping:
\begin{equation}
	\label{3d.fmap}
	\{ J(x,y,z,\theta), F_h(x,y,z), F_e(x,y,z) \} \to \{ H(x,y,z,\theta), E(x,y,z,\theta) \}
\end{equation}

For practical analysis, we divide the general problem into two fundamental sub-cases.

\textbf{Initial-value problem}:
\begin{equation}
	\begin{gathered}
		\label{3d.fmapIni}
		\{J(x,y,z,\theta) \equiv 0, F_h(x,y,z), F_e(x,y,z) \} \to \{ H(x,y,z,\theta), E(x,y,z,\theta) \}
	\end{gathered}
\end{equation}

\textbf{Source-driven problem}:
\begin{equation}
	\begin{gathered}
		\label{3d.fmapSrc}
		\{ J(x,y,z,\theta), F_h(x,y,z) \equiv 0, F_e(x,y,z) \equiv 0\} \to \{ H(x,y,z,\theta), E(x,y,z,\theta) \}
	\end{gathered}
\end{equation}

\section{The solutions}
\label{sect.solution}
\textbf{Definition of higher orders of curls}:

For a vector field \(F:\mathbb{R}^3\rightarrow\mathbb{R}^3\), the repeated curl operator is defined recursively by
\begin{equation*}
	(\nabla\times)^0F = F,	
\end{equation*}
\begin{equation*}
	(\nabla\times)^nF
	=
	\nabla\times\left((\nabla\times)^{\,n-1}F\right),
	\qquad n\ge1.
\end{equation*}

\subsection{\textbf{Uniqueness of solutions}}

\begin{thm} \textbf{(Uniqueness of smooth finite-energy fields)}.
	Suppose two sufficiently smooth solutions satisfy Maxwell’s equations with the same initial conditions and source terms. Assume further that their difference fields have finite electromagnetic energy and that the Poynting flux through a sphere of radius R tends to zero as $R \to \infty$. Then the two solutions coincide.
\end{thm}
\begin{proof}
Let $(H_1,E_1)$ and $(H_2,E_2)$ be two solutions with the same initial conditions and source terms. Define the difference fields
\begin{equation*}
	E=E_1-E_2,\qquad H=H_1-H_2.	
\end{equation*}
Then $E$ and $H$ satisfy the homogeneous Maxwell equations with zero initial data.

Define the electromagnetic energy density
\begin{equation*}
	u(x,y,z,\theta)
	=
	\frac12\left(\epsilon |E(x,y,z,\theta)|^2
	+\mu |H(x,y,z,\theta)|^2\right).
\end{equation*}
By Poynting's theorem for the homogeneous system,
\begin{equation*}
\frac{\partial u}{\partial \theta}=-\nabla\cdot(E\times H).	
\end{equation*}
For a bounded volume $V$, define
\begin{equation*}
U_V(\theta)=\int_V u(x,y,z,\theta)\,dV.	
\end{equation*}
Then
\begin{align*}
	\frac{dU_V}{d\theta}
	&=
	-\int_V \nabla\cdot(E\times H)\,dV \\
	&=
	-\int_{\partial V}(E\times H)\cdot\mathbf{n}\,dS.
\end{align*}
Let $V$ be a ball of radius $R$. Under the assumed decay of the fields at spatial infinity, the boundary flux tends to zero as $R\to\infty$. Therefore, for the total energy in open space,
\begin{equation*}
U(\theta)=\int_{\mathbb R^3}u(x,y,z,\theta)\,dV,	
\end{equation*}
we have
\begin{equation*}
\frac{dU}{d\theta}=0.	
\end{equation*}
Since the two solutions have the same initial data,
\begin{equation*}
E(x,y,z,0)=0,\qquad H(x,y,z,0)=0,	
\end{equation*}
and hence
\begin{equation*}
U(0)=0.	
\end{equation*}
Therefore,
\begin{equation*}
U(\theta)=0	
\end{equation*}
for all $\theta $. Since $u(x,y,z,\theta)\ge 0$, the integral of $u$ over $\mathbb R^3$ can be zero only if $u(x,y,z,\theta )=0$ almost everywhere. By continuity of the fields, $u(x,y,z,\theta )=0$ everywhere. Hence
\begin{equation*}
E(x,y,z,\theta )=0,\qquad H(x,y,z,\theta )=0.	
\end{equation*}
Thus $E_1=E_2$ and $H_1=H_2$, proving uniqueness.
\end{proof}

\begin{thm} \textbf{(Uniqueness in the real-analytic class)}.
	Suppose two solutions of Maxwell’s equations have the same initial conditions and source terms and are analytic in $\theta$ on a connected time interval containing $\theta=0$. If the hypotheses required by the Time–Space theorem hold, then the two solutions coincide on that interval.
\end{thm}
\begin{proof}
	Let $E=E_1-E_2$ and $H=H_1-H_2$. Since the two solutions have the same source terms and initial conditions, the difference fields satisfy the homogeneous Maxwell equations and
	\begin{equation*}
	E(x,y,z,0)=0,\qquad H(x,y,z,0)=0.	
	\end{equation*}
	The above leads to
	\begin{equation}
		\label{unique.curl.h}
		\curl{n} H(x,y,z,0) = 0, \forall n \ge 0 \,,
	\end{equation}
	\begin{equation}
		\label{unique.curl.e}
		\curl{n} E(x,y,z,0) = 0, \forall n \ge 0 \,.
	\end{equation}
	
	By the Time--Space theorem, all temporal derivatives of the difference fields at $\theta=0$ vanish:
	\begin{equation*}
	\frac{\partial^n E(x,y,z,0)}{\partial\theta^n}=0,\qquad
\frac{\partial^n H(x,y,z,0)}{\partial\theta^n}=0,
\qquad n=0,1,2,\ldots.		
	\end{equation*}
	Because $E$ and $H$ are analytic in $\theta$, their Taylor expansions about $\theta=0$ vanish identically within their intervals of convergence. Hence
	\begin{equation*}
	E(x,y,z,\theta)=0,\qquad H(x,y,z,\theta)=0.	
	\end{equation*}
	By analytic continuation, this equality holds throughout the connected time interval under consideration. Therefore,
	\begin{equation*}
	E_1=E_2,\qquad H_1=H_2,	
	\end{equation*}
	which proves uniqueness.	
\end{proof}

\subsection{\textbf{Initial-value problem}}
For convenience, define
\begin{align}
	\label{ini.terms.1}
	F_1^{(n)}(x,y,z)
	&=
	(\nabla\times)^n F_h(x,y,z),
	\\
	\label{ini.terms.2}
	F_2^{(n)}(x,y,z)
	&=
	(\nabla\times)^n F_e(x,y,z).
\end{align}

\begin{thm} \label{thm.ini} \textbf{(Initial value solution)}. 
	Suppose \(F_h,F_e\in C^\infty(\mathbb{R}^3,\mathbb{R}^3)\), and suppose there exist constants \(C>0\), \(A>0\), and \(0\leq \beta<1\) such that
	\[
	\left|F_j^{(n)}(x,y,z)\right|
	\leq C A^n (n!)^\beta,
	\qquad j=1,2
	\]
	for all sufficiently large \(n\), uniformly in \((x,y,z)\). Then the series defined by \eqref{3d.solIniH} and \eqref{3d.solIniE} converge absolutely and uniformly on compact \(\theta\)-intervals. The resulting fields satisfy the initial-value problem described by \eqref{3d.fmapIni}.
	\begin{equation}
		\label{3d.solIniH}
	\begin{split}
		H(x,y,z,\theta) &= F_h(x,y,z) \\
		&+ \sum_{n=0}^{\infty}(-1)^{n+1} \powfact{\theta}{{2(n+1)}} F_1^{(2(n+1))}(x,y,z) \\
		&+ \frac{1}{\eta}\sum_{n=0}^{\infty}(-1)^{n+1} \powfact{\theta}{2n+1} F_2^{(2n+1)}(x,y,z),
	\end{split}
	\end{equation}
	\begin{equation}
	\label{3d.solIniE}
	\begin{split}
		E(x,y,z,\theta) &= F_e(x,y,z) \\
		 &+ \sum_{n=0}^{\infty}(-1)^{n+1} \powfact{\theta}{2(n+1)} F_2^{(2(n+1))}(x,y,z) \\
		  &+ \eta \sum_{n=0}^{\infty} (-1)^n \powfact{\theta}{2n+1} F_1^{(2n+1)}(x,y,z).
	\end{split}
	\end{equation}
\end{thm}

\begin{proof}
	We first justify the term-by-term operations used below. Let $T>0$ and suppose that $|\theta|\leq T$. By the hypothesis of the
	theorem, for all sufficiently large $m$ and for $j=1,2$,
	\begin{equation*}
	\left|F_j^{(m)}(x,y,z)\right|
	\leq C A^m(m!)^\beta .	
	\end{equation*}
	Hence, a typical term in either \eqref{3d.solIniH} or \eqref{3d.solIniE} is bounded by
	\begin{equation*}
	C\frac{(AT)^m}{(m!)^{1-\beta}}.	
	\end{equation*}
	Because $0\leq\beta<1$, the corresponding numerical series 	converges. The same estimate, after an index shift, applies to
	the series obtained by differentiating with respect to $\theta$ or by applying one additional curl operator. Therefore, all the
	series used below converge absolutely and uniformly on compact 	$\theta$-intervals, and the required term-by-term operations are
	valid.
	
	Differentiating \eqref{3d.solIniH} with respect to $\theta$, we obtain
	\begin{equation}
		\label{proof.ini.H.1}
		\begin{split}
			\frac{\partial H(x,y,z,\theta)}{\partial\theta}
			&=
			\sum_{n=0}^{\infty}	(-1)^{n+1}	\powfact{\theta}{2n+1}	F_1^{(2n+2)}(x,y,z)
			\\
			&\quad+	\frac{1}{\eta}\sum_{n=0}^{\infty}(-1)^{n+1}	\powfact{\theta}{2n} F_2^{(2n+1)}(x,y,z).
		\end{split}
	\end{equation}
	
	Applying the curl operator to \eqref{3d.solIniE}, we obtain
	\begin{equation*}
		\begin{split}
			\nabla\times E(x,y,z,\theta)
			&= 	F_2^{(1)}(x,y,z)
			\\
			&\quad+	\sum_{n=0}^{\infty}	(-1)^{n+1} \powfact{\theta}{2(n+1)}	F_2^{(2n+3)}(x,y,z)
			\\
			&\quad+ \eta \sum_{n=0}^{\infty}(-1)^n	\powfact{\theta}{2n+1}	F_1^{(2n+2)}(x,y,z).
		\end{split}
	\end{equation*}
	Combining the first term with the first infinite sum by shifting
	the summation index gives
	\begin{equation}
		\label{proof.ini.E.1}
		\begin{split}
			\nabla\times E(x,y,z,\theta)
			&= 	-\sum_{n=0}^{\infty}(-1)^{n+1}	\powfact{\theta}{2n} F_2^{(2n+1)}(x,y,z)
			\\
			&\quad-	\eta \sum_{n=0}^{\infty} (-1)^{n+1} \powfact{\theta}{2n+1}	F_1^{(2n+2)}(x,y,z).
		\end{split}
	\end{equation}
	Comparison of \eqref{proof.ini.H.1} and \eqref{proof.ini.E.1} yields
	\begin{equation*}
	\frac{\partial H}{\partial\theta}
	=
	-\frac{1}{\eta}\nabla\times E,	
	\end{equation*}
	which is \eqref{3d.1}.
	
	Next, differentiating \eqref{3d.solIniE} with respect to $\theta$, we obtain
	\begin{equation}
		\label{proof.ini.E.2}
		\begin{split}
			\frac{\partial E(x,y,z,\theta)}{\partial\theta}
			&= 	\sum_{n=0}^{\infty} (-1)^{n+1} \powfact{\theta}{2n+1} F_2^{(2n+2)}(x,y,z)
			\\
			&\quad+	\eta \sum_{n=0}^{\infty} (-1)^n \powfact{\theta}{2n} F_1^{(2n+1)}(x,y,z).
		\end{split}
	\end{equation}
	
	Applying the curl operator to \eqref{3d.solIniH}, we obtain
	\begin{equation*}
		\begin{split}
			\nabla\times H(x,y,z,\theta)
			&=	F_1^{(1)}(x,y,z)
			\\
			&\quad+ \sum_{n=0}^{\infty} (-1)^{n+1} \powfact{\theta}{2(n+1)} F_1^{(2n+3)}(x,y,z)
			\\
			&\quad+ \frac{1}{\eta} \sum_{n=0}^{\infty} (-1)^{n+1} \powfact{\theta}{2n+1} F_2^{(2n+2)}(x,y,z).
		\end{split}
	\end{equation*}
	Combining the first term with the first infinite sum gives
	\begin{equation}
		\label{proof.ini.H.2}
		\begin{split}
			\nabla\times H(x,y,z,\theta)
			&= \sum_{n=0}^{\infty} (-1)^n \powfact{\theta}{2n} F_1^{(2n+1)}(x,y,z)
			\\
			&\quad+ \frac{1}{\eta} \sum_{n=0}^{\infty} (-1)^{n+1} \powfact{\theta}{2n+1} F_2^{(2n+2)}(x,y,z).
		\end{split}
	\end{equation}
	Comparison of \eqref{proof.ini.E.2} and \eqref{proof.ini.H.2} yields
	\begin{equation*}
	\frac{\partial E}{\partial\theta}
	=
	\eta\nabla\times H,	
	\end{equation*}
	which is \eqref{3d.2}.
	
	Finally, setting $\theta=0$ in \eqref{3d.solIniH}, all terms in
	the infinite sums vanish, and hence
	\begin{equation*}
	H(x,y,z,0)=F_h(x,y,z),	
	\end{equation*}
	which is \eqref{3d.3}. Similarly, setting $\theta=0$ in \eqref{3d.solIniE} gives
	\begin{equation*}
	E(x,y,z,0)=F_e(x,y,z),	
	\end{equation*}
	which is \eqref{3d.4}.
	
	Therefore, the fields $H$ and $E$ defined by \eqref{3d.solIniH} and \eqref{3d.solIniE}, respectively, satisfy the initial-value problem \eqref{3d.fmapIni}.
\end{proof}

\begin{corollary} \textbf{(Time analyticity of the initial-value solutions).}
	For every fixed $(x,y,z)\in\mathbb{R}^3$, the solution fields 	$H(x,y,z,\theta)$ and $E(x,y,z,\theta)$ given by 	Theorem~\ref{thm.ini} are real analytic functions of $\theta$.
\end{corollary}

\begin{proof}
	Substituting \eqref{ini.terms.1} and \eqref{ini.terms.2} into \eqref{3d.solIniH} and \eqref{3d.solIniE}, respectively, gives
	\begin{equation*}
		\begin{split}
			H(x,y,z,\theta)
			&=F_h(x,y,z)
			\\
			&\quad+ \sum_{n=0}^{\infty} (-1)^{n+1} \powfact{\theta}{2(n+1)} \curl{2(n+1)}F_h(x,y,z)
			\\
			&\quad+ \frac{1}{\eta} \sum_{n=0}^{\infty} (-1)^{n+1} \powfact{\theta}{2n+1} \curl{2n+1}F_e(x,y,z),
		\end{split}
	\end{equation*}
	and
	\begin{equation*}
		\begin{split}
			E(x,y,z,\theta)
			&=F_e(x,y,z)
			\\
			&\quad+ \sum_{n=0}^{\infty} (-1)^{n+1} \powfact{\theta}{2(n+1)} \curl{2(n+1)}F_e(x,y,z)
			\\
			&\quad+ \eta \sum_{n=0}^{\infty} (-1)^n \powfact{\theta}{2n+1} \curl{2n+1}F_h(x,y,z).
		\end{split}
	\end{equation*}
	
	By Theorem~\ref{thm.ini}, these series define fields satisfying Maxwell's equations and the prescribed initial conditions.
	Therefore, the Time--Space identities \eqref{ts.h1}--\eqref{ts.e2} apply to the resulting solution.
	Using those identities to replace the repeated spatial curls by the corresponding temporal derivatives at $\theta=0$, we obtain
	\begin{equation}
		\label{H.taylor}
		\begin{split}
			H(x,y,z,\theta)
			&= \sum_{n=0}^{\infty}
			\frac{\partial^{2n}H(x,y,z,0)}
			{\partial\theta^{2n}}
			\frac{\theta^{2n}}{(2n)!}
			\\
			&\quad+
			\sum_{n=0}^{\infty}
			\frac{\partial^{2n+1}H(x,y,z,0)}
			{\partial\theta^{2n+1}}
			\frac{\theta^{2n+1}}{(2n+1)!},
		\end{split}
	\end{equation}
	and
	\begin{equation}
		\label{E.taylor}
		\begin{split}
			E(x,y,z,\theta)
			&=
			\sum_{n=0}^{\infty}
			\frac{\partial^{2n}E(x,y,z,0)}
			{\partial\theta^{2n}}
			\frac{\theta^{2n}}{(2n)!}
			\\
			&\quad+
			\sum_{n=0}^{\infty}
			\frac{\partial^{2n+1}E(x,y,z,0)}
			{\partial\theta^{2n+1}}
			\frac{\theta^{2n+1}}{(2n+1)!}.
		\end{split}
	\end{equation}
	
	Combining the even- and odd-order terms in \eqref{H.taylor} and \eqref{E.taylor} yields
	\begin{equation*}
	H(x,y,z,\theta)
	=
	\sum_{m=0}^{\infty}
	\frac{\partial^m H(x,y,z,0)}
	{\partial\theta^m}
	\frac{\theta^m}{m!},	
	\end{equation*}
	and
	\begin{equation*}
	E(x,y,z,\theta)
	=
	\sum_{m=0}^{\infty}
	\frac{\partial^m E(x,y,z,0)}
	{\partial\theta^m}
	\frac{\theta^m}{m!}.	
	\end{equation*}
	The convergence established in Theorem~\ref{thm.ini} holds on every compact $\theta$-interval. Hence, for every fixed
	$(x,y,z)$, each field agrees with its Taylor series about $\theta=0$ for every real $\theta$. Therefore, $H$ and $E$ are
	real analytic with respect to time, componentwise.
\end{proof}

\subsection{\textbf{Source-driven problem}}
For convenience, define
\begin{align}
	\label{src.terms.1}
	S_1^{(n)}(x,y,z)
	&=
	\sum_{k=0}^{n}(-1)^k
	(\nabla\times)^{2k+1}
	\frac{\partial^{2(n-k)}J(x,y,z,0)}
	{\partial\theta^{2(n-k)}},
	\\
	\label{src.terms.2}
	S_2^{(n)}(x,y,z)
	&=
	\sum_{k=1}^{n}(-1)^{k+1}
	(\nabla\times)^{2k-1}
	\frac{\partial^{2(n-k)+1}J(x,y,z,0)}
	{\partial\theta^{2(n-k)+1}},
	\\
	\label{src.terms.3}
	S_3^{(n)}(x,y,z)
	&=
	\sum_{k=0}^{n}(-1)^{k+1}
	(\nabla\times)^{2k}
	\frac{\partial^{2(n-k)+1}J(x,y,z,0)}
	{\partial\theta^{2(n-k)+1}},
	\\
	\label{src.terms.4}
	S_4^{(n)}(x,y,z)
	&=
	\sum_{k=0}^{n}(-1)^{k+1}
	(\nabla\times)^{2k}
	\frac{\partial^{2(n-k)}J(x,y,z,0)}
	{\partial\theta^{2(n-k)}}.
\end{align}

\begin{thm}\label{thm.src}
	\textbf{(Source-driven solution).}
	Suppose that $J\in C^\infty(\mathbb{R}^4,\mathbb{R}^3)$ and that, for every fixed $(x,y,z)$, the function
	$\theta\mapsto J(x,y,z,\theta)$ is represented by its Taylor series about $\theta=0$ for every $\theta\in\mathbb{R}$.
	Suppose further that there exist constants $C>0$, $A>0$, and $0\leq\beta<1$ such that
	\[
	\left|S_j^{(n)}(x,y,z)\right|
	\leq C A^n(n!)^\beta,
	\qquad j=1,2,3,4,
	\]
	for all sufficiently large $n$, uniformly in $(x,y,z)$. Then the series defined by \eqref{3d.solSrcH} and
	\eqref{3d.solSrcE} converge absolutely and uniformly on compact $\theta$-intervals. The resulting fields satisfy the
	source-driven problem described by \eqref{3d.fmapSrc}.
	\begin{equation}
		\label{3d.solSrcH}
		\begin{split}
			H(x,y,z,\theta)
			&=
			\sum_{n=0}^{\infty}
			\frac{\theta^{2(n+1)}}{(2(n+1))!}
			S_1^{(n)}(x,y,z)
			\\
			&\quad+
			\sum_{n=1}^{\infty}
			\frac{\theta^{2n+1}}{(2n+1)!}
			S_2^{(n)}(x,y,z),
		\end{split}
	\end{equation}
	and
	\begin{equation}
		\label{3d.solSrcE}
		\begin{split}
			E(x,y,z,\theta)
			&=
			\eta
			\sum_{n=0}^{\infty}
			\frac{\theta^{2(n+1)}}{(2(n+1))!}
			S_3^{(n)}(x,y,z)
			\\
			&\quad+
			\eta
			\sum_{n=0}^{\infty}
			\frac{\theta^{2n+1}}{(2n+1)!}
			S_4^{(n)}(x,y,z).
		\end{split}
	\end{equation}
\end{thm}

\begin{proof}
	We first justify the term-by-term operations used below. Let $T>0$ and suppose that $|\theta|\leq T$. By the assumed bounds
	on $S_j^{(n)}$, the terms in \eqref{3d.solSrcH} and \eqref{3d.solSrcE} are bounded, up to fixed factors and harmless
	index shifts, by
	\[
	C\frac{(AT^2)^n}{(n!)^{2-\beta}}.
	\]
	Because $0\leq\beta<1$, the resulting numerical series converges. The corresponding series obtained by differentiating
	with respect to $\theta$, applying the curl operator, shifting indices, or separating finitely many terms satisfy analogous
	bounds. Hence, the operations below are valid on compact $\theta$-intervals.
	
	Differentiating \eqref{3d.solSrcH} with respect to $\theta$, we	obtain
	\begin{equation}
		\label{proof.src.H.1}
		\begin{split}
			\frac{\partial H(x,y,z,\theta)}{\partial\theta}
			&=
			\sum_{n=0}^{\infty}
			\powfact{\theta}{2n+1}
			\sum_{k=0}^{n}
			(-1)^k
			\curl{2k+1}
			\pt{J(x,y,z,0)}{\theta}{2(n-k)}
			\\
			&\quad+
			\sum_{n=1}^{\infty}
			\powfact{\theta}{2n}
			\sum_{k=1}^{n}
			(-1)^{k+1}
			\curl{2k-1}
			\pt{J(x,y,z,0)}{\theta}{2(n-k)+1}.
		\end{split}
	\end{equation}
	
	Applying the curl operator to \eqref{3d.solSrcE}, we obtain
	\begin{equation*}
		\begin{split}
			\nabla\times E(x,y,z,\theta)
			&=
			\eta
			\sum_{n=0}^{\infty}
			\powfact{\theta}{2(n+1)}
			\sum_{k=0}^{n}
			(-1)^{k+1}
			\curl{2k+1}
			\pt{J(x,y,z,0)}{\theta}{2(n-k)+1}
			\\
			&\quad+
			\eta
			\sum_{n=0}^{\infty}
			\powfact{\theta}{2n+1}
			\sum_{k=0}^{n}
			(-1)^{k+1}
			\curl{2k+1}
			\pt{J(x,y,z,0)}{\theta}{2(n-k)}.
		\end{split}
	\end{equation*}
	Reindexing the first series gives
	\begin{equation}
		\label{proof.src.E.1}
		\begin{split}
			&\nabla\times E(x,y,z,\theta)
			=
			\eta
			\sum_{n=1}^{\infty}
			\powfact{\theta}{2n}
			\sum_{k=1}^{n}
			(-1)^k
			\curl{2k-1}
			\pt{J(x,y,z,0)}{\theta}{2(n-k)+1}
			\\
			&\quad+
			\eta
			\sum_{n=0}^{\infty}
			\powfact{\theta}{2n+1}
			\sum_{k=0}^{n}
			(-1)^{k+1}
			\curl{2k+1}
			\pt{J(x,y,z,0)}{\theta}{2(n-k)}.
		\end{split}
	\end{equation}
	Comparing \eqref{proof.src.H.1} and	\eqref{proof.src.E.1}, we find
	\[
	\frac{\partial H}{\partial\theta}
	=
	-\frac{1}{\eta}\nabla\times E,
	\]
	which is \eqref{3d.1}.
	
	Next, differentiating \eqref{3d.solSrcE} with respect to $\theta$ gives
	\begin{equation*}
		\begin{split}
			\frac{\partial E(x,y,z,\theta)}{\partial\theta}
			&=
			\eta
			\sum_{n=0}^{\infty}
			\powfact{\theta}{2n+1}
			\sum_{k=0}^{n}
			(-1)^{k+1}
			\curl{2k}
			\pt{J(x,y,z,0)}{\theta}{2(n-k)+1}
			\\
			&\quad+
			\eta
			\sum_{n=0}^{\infty}
			\powfact{\theta}{2n}
			\sum_{k=0}^{n}
			(-1)^{k+1}
			\curl{2k}
			\pt{J(x,y,z,0)}{\theta}{2(n-k)}.
		\end{split}
	\end{equation*}
	Separating the $k=0$ terms yields
	\begin{equation*}
		\begin{split}
			\frac{\partial E(x,y,z,\theta)}{\partial\theta}
			&=
			-\eta
			\sum_{n=0}^{\infty}
			\powfact{\theta}{2n+1}
			\pt{J(x,y,z,0)}{\theta}{2n+1}
			\\
			&\quad-
			\eta
			\sum_{n=0}^{\infty}
			\powfact{\theta}{2n}
			\pt{J(x,y,z,0)}{\theta}{2n}
			\\
			&\quad+
			\eta
			\sum_{n=1}^{\infty}
			\powfact{\theta}{2n+1}
			\sum_{k=1}^{n}
			(-1)^{k+1}
			\curl{2k}
			\pt{J(x,y,z,0)}{\theta}{2(n-k)+1}
			\\
			&\quad+
			\eta
			\sum_{n=1}^{\infty}
			\powfact{\theta}{2n}
			\sum_{k=1}^{n}
			(-1)^{k+1}
			\curl{2k}
			\pt{J(x,y,z,0)}{\theta}{2(n-k)}.
		\end{split}
	\end{equation*}
	The first two series combine to give
	\[
	-\eta
	\sum_{m=0}^{\infty}
	\frac{\theta^m}{m!}
	\pt{J(x,y,z,0)}{\theta}{m}
	=
	-\eta J(x,y,z,\theta),
	\]
	where the final equality follows from the assumed Taylor representation of $J$ about $\theta=0$.
	
	Reindexing the last even-order series, we therefore obtain
	\begin{equation}
		\label{proof.src.E.2}
		\begin{split}
			\frac{\partial E(x,y,z,\theta)}{\partial\theta}
			&=
			-\eta J(x,y,z,\theta)
			\\
			&\quad+
			\eta
			\sum_{n=1}^{\infty}
			\powfact{\theta}{2n+1}
			\sum_{k=1}^{n}
			(-1)^{k+1}
			\curl{2k}
			\pt{J(x,y,z,0)}{\theta}{2(n-k)+1}
			\\
			&\quad+
			\eta
			\sum_{n=0}^{\infty}
			\powfact{\theta}{2(n+1)}
			\sum_{k=0}^{n}
			(-1)^k
			\curl{2k+2}
			\pt{J(x,y,z,0)}{\theta}{2(n-k)}.
		\end{split}
	\end{equation}
	
	Applying the curl operator to \eqref{3d.solSrcH} gives
	\begin{equation}
		\label{proof.src.H.2}
		\begin{split}
			\nabla\times H(x,y,z,\theta)
			&=
			\sum_{n=0}^{\infty}
			\powfact{\theta}{2(n+1)}
			\sum_{k=0}^{n}
			(-1)^k
			\curl{2k+2}
			\pt{J(x,y,z,0)}{\theta}{2(n-k)}
			\\
			&\quad+
			\sum_{n=1}^{\infty}
			\powfact{\theta}{2n+1}
			\sum_{k=1}^{n}
			(-1)^{k+1}
			\curl{2k}
			\pt{J(x,y,z,0)}{\theta}{2(n-k)+1}.
		\end{split}
	\end{equation}
	Comparison of \eqref{proof.src.E.2} and	\eqref{proof.src.H.2} yields
	\[
	\frac{\partial E}{\partial\theta}
	=
	\eta\nabla\times H-\eta J,
	\]
	which is \eqref{3d.2}.
	
	Finally, setting $\theta=0$ in \eqref{3d.solSrcH} and \eqref{3d.solSrcE}, every summation term vanishes. Hence,
	\[
	H(x,y,z,0)=0,
	\qquad
	E(x,y,z,0)=0.
	\]
	These are the prescribed zero initial conditions.
	
	Therefore, the fields defined by \eqref{3d.solSrcH} and	\eqref{3d.solSrcE} satisfy the source-driven problem \eqref{3d.fmapSrc}.
\end{proof}

\begin{corollary}
	\textbf{(Time analyticity of the source-driven solutions).}
	For every fixed $(x,y,z)\in\mathbb{R}^3$, the solution fields $H(x,y,z,\theta)$ and $E(x,y,z,\theta)$ given by
	Theorem~\ref{thm.src} are real analytic functions of $\theta$.
\end{corollary}

\begin{proof}
	Because the source-driven problem has zero initial fields, the Time--Space identities \eqref{ts.h1}, \eqref{ts.h2},
	\eqref{ts.e1}, and \eqref{ts.e2} reduce at $\theta=0$ to
	\begin{align}
		\label{ts.h1a}
		\pt{H(x,y,z,0)}{\theta}{2n+1}
		&=
		J_{h,2n+1}(x,y,z,0),
		\\
		\label{ts.h2a}
		\pt{H(x,y,z,0)}{\theta}{2n+2}
		&=
		J_{h,2n+2}(x,y,z,0),
		\\
		\label{ts.e1a}
		\pt{E(x,y,z,0)}{\theta}{2n+1}
		&=
		\eta J_{e,2n+1}(x,y,z,0),
		\\
		\label{ts.e2a}
		\pt{E(x,y,z,0)}{\theta}{2n+2}
		&=
		\eta J_{e,2n+2}(x,y,z,0).
	\end{align}
	
	Substituting \eqref{src.terms.1}--\eqref{src.terms.4} into \eqref{3d.solSrcH} and \eqref{3d.solSrcE}, respectively, gives
	\begin{equation}
		\label{src.proof.H.2b}
		\begin{split}
			H(x,y,z,\theta)
			&=
			\sum_{n=0}^{\infty}
			\sum_{k=0}^{n}
			(-1)^k
			(\nabla\times)^{2k+1}
			\frac{\partial^{2(n-k)}J(x,y,z,0)}
			{\partial\theta^{2(n-k)}}
			\frac{\theta^{2(n+1)}}{(2(n+1))!}
			\\
			&\quad+
			\sum_{n=1}^{\infty}
			\sum_{k=1}^{n}
			(-1)^{k+1}
			(\nabla\times)^{2k-1}
			\frac{\partial^{2(n-k)+1}J(x,y,z,0)}
			{\partial\theta^{2(n-k)+1}}
			\frac{\theta^{2n+1}}{(2n+1)!},
		\end{split}
	\end{equation}
	and
	\begin{equation}
		\label{src.proof.E.2b}
		\begin{split}
			E(x,y,z,\theta)
			&=
			\eta
			\sum_{n=0}^{\infty}
			\sum_{k=0}^{n}
			(-1)^{k+1}
			(\nabla\times)^{2k}
			\frac{\partial^{2(n-k)+1}J(x,y,z,0)}
			{\partial\theta^{2(n-k)+1}}
			\frac{\theta^{2(n+1)}}{(2(n+1))!}
			\\
			&\quad+
			\eta
			\sum_{n=0}^{\infty}
			\sum_{k=0}^{n}
			(-1)^{k+1}
			(\nabla\times)^{2k}
			\frac{\partial^{2(n-k)}J(x,y,z,0)}
			{\partial\theta^{2(n-k)}}
			\frac{\theta^{2n+1}}{(2n+1)!}.
		\end{split}
	\end{equation}
	
	Using \eqref{ts.jh1} and \eqref{ts.jh2} in \eqref{src.proof.H.2b}, and using \eqref{ts.je1} and \eqref{ts.je2} in \eqref{src.proof.E.2b}, we obtain
	\begin{equation}
		\label{src.proof.H.2}
		\begin{split}
			H(x,y,z,\theta)
			&=
			\sum_{n=0}^{\infty}
			J_{h,2n+2}(x,y,z,0)
			\frac{\theta^{2(n+1)}}{(2(n+1))!}
			\\
			&\quad+
			\sum_{n=0}^{\infty}
			J_{h,2n+1}(x,y,z,0)
			\frac{\theta^{2n+1}}{(2n+1)!},
		\end{split}
	\end{equation}
	and
	\begin{equation}
		\label{src.proof.E.2}
		\begin{split}
			E(x,y,z,\theta)
			&=
			\sum_{n=0}^{\infty}
			\eta J_{e,2n+2}(x,y,z,0)
			\frac{\theta^{2(n+1)}}{(2(n+1))!}
			\\
			&\quad+
			\sum_{n=0}^{\infty}
			\eta J_{e,2n+1}(x,y,z,0)
			\frac{\theta^{2n+1}}{(2n+1)!}.
		\end{split}
	\end{equation}
	In \eqref{src.proof.H.2}, the $n=0$ term of the second sum is zero, since
	\[
	\frac{\partial H(x,y,z,0)}{\partial\theta}
	=
	-\frac{1}{\eta}
	\nabla\times E(x,y,z,0)
	=
	0.
	\]
	
	Substituting \eqref{ts.h1a} and \eqref{ts.h2a} into \eqref{src.proof.H.2}, and substituting \eqref{ts.e1a} and
	\eqref{ts.e2a} into \eqref{src.proof.E.2}, gives
	\begin{equation}
		\label{src.proof.H.1}
		\begin{split}
			H(x,y,z,\theta)
			&=
			\sum_{n=1}^{\infty}
			\frac{\partial^{2n}H(x,y,z,0)}
			{\partial\theta^{2n}}
			\frac{\theta^{2n}}{(2n)!}
			\\
			&\quad+
			\sum_{n=0}^{\infty}
			\frac{\partial^{2n+1}H(x,y,z,0)}
			{\partial\theta^{2n+1}}
			\frac{\theta^{2n+1}}{(2n+1)!},
		\end{split}
	\end{equation}
	and
	\begin{equation}
		\label{src.proof.E.1}
		\begin{split}
			E(x,y,z,\theta)
			&=
			\sum_{n=1}^{\infty}
			\frac{\partial^{2n}E(x,y,z,0)}
			{\partial\theta^{2n}}
			\frac{\theta^{2n}}{(2n)!}
			\\
			&\quad+
			\sum_{n=0}^{\infty}
			\frac{\partial^{2n+1}E(x,y,z,0)}
			{\partial\theta^{2n+1}}
			\frac{\theta^{2n+1}}{(2n+1)!}.
		\end{split}
	\end{equation}
	
	Since $H(x,y,z,0)=E(x,y,z,0)=0$, the missing zeroth-order terms may be included in the even-order sums. Consequently,
	\begin{align*}
		H(x,y,z,\theta)
		&=
		\sum_{m=0}^{\infty}
		\frac{\partial^mH(x,y,z,0)}
		{\partial\theta^m}
		\frac{\theta^m}{m!},
		\\
		E(x,y,z,\theta)
		&=
		\sum_{m=0}^{\infty}
		\frac{\partial^mE(x,y,z,0)}
		{\partial\theta^m}
		\frac{\theta^m}{m!}.
	\end{align*}
	
	By Theorem~\ref{thm.src}, these series converge on every compact $\theta$-interval. Thus, for every fixed $(x,y,z)$, the fields
	agree with their Taylor series about $\theta=0$. Therefore, $H$ and $E$ are real analytic with respect to time, componentwise.
\end{proof}

Since Maxwell's equations are linear, the solution operator defined by \eqref{3d.solSrcH} and \eqref{3d.solSrcE} preserves linear combinations of sources.
\begin{thm}[Superposition principle]
	Let $\{H_1,E_1\}$ and $\{H_2,E_2\}$ be the solutions
	corresponding to the sources $J_1$ and $J_2$, respectively.
	Then, for any $a,b\in\mathbb{R}$,
	\[
	\{aH_1+bH_2,\;
	aE_1+bE_2\}
	\]
	is the solution corresponding to the source
	\[
	aJ_1+bJ_2.
	\]
\end{thm}

Although formulas \eqref{3d.solSrcH} and \eqref{3d.solSrcE} are written with the expansion center at $\theta=0 $, the Time–Space theorem is valid at an arbitrary time $\theta_0 $. Consequently, the same construction provides a local-in-time analytical representation from $\theta_0$ to $\theta_0+\bigtriangleup_{\theta} $, with the source derivatives evaluated at $\theta_0$. Under the stated convergence assumptions, these derivatives encode the local behavior of the source over the interval of convergence. Thus, the source-driven solution should be interpreted as a local evolution formula rather than a representation depending exclusively on source information at the initial time.

\section{\textbf{Deformed trigonometric functions}}
\label{sect.deformedFunction}
\subsection{Standard and deformed trigonometric functions}
Consider the classical traveling-wave sine and cosine functions, which admit the series representation for $\theta \ne 0$
\begin{align}
		\label{form.sin}
	\sin (\omega(x-\theta)) 
	&= \sum_{n=0}^{\infty}\frac{(-1)^n\theta^{2n+1}}{(2n+1)!} w^{(n)}_1(x,\theta) \,, \\
	\label{form.cos}
	\cos (\omega(x-\theta)) 
	&= \sum_{n=0}^{\infty} \frac{(-1)^n\theta^{2n}}{(2n)!} w^{(n)}_2(x,\theta) \,.
\end{align}
where
\begin{align}
	\label{form.w1}
	w^{(n)}_1(x,\theta) &= \sum_{k=0}^{2n+1}\binom{2n+1}{k}(-1)^{k+1}(x/\theta)^k \,, \\
	\label{form.w2}
	w^{(n)}_2(x,\theta) &= \sum_{k=0}^{2n}\binom{2n}{k}(-1)^k(x/\theta)^k \,.
\end{align}
Comparing \eqref{form.sin} and \eqref{form.cos} with the solution formulas \eqref{3d.solIniH}, \eqref{3d.solIniE}, \eqref{3d.solSrcH} and \eqref{3d.solSrcE}, we see the following common structure:
\begin{align}
	\label{sind}
	\mathrm{sind}(x,y,z,\theta) &= \sum_{n=0}^{\infty}\frac{(-1)^n\theta^{2n+1}}{(2n+1)!} \mathrm{d}^{(n)}_1(x,y,z,\theta) \,, \\
	\label{cosd}
	\mathrm{cosd}(x,y,z,\theta) &= \sum_{n=0}^{\infty}\frac{(-1)^n\theta^{2n}}{(2n)!} \mathrm{d}^{(n)}_2(x,y,z,\theta) \,,
\end{align}
This common series structure motivates the introduction of the following definition. We refer to functions of the form \eqref{sind} and \eqref{cosd} as deformed trigonometric functions, where $\mathrm{d}^{(n)}_1(x,y,z,\theta) $ and $\mathrm{d}^{(n)}_2(x,y,z,\theta) $ 
are called the formation functions, which characterize the particular deformation.

The distinction between the classical and deformed cases lies entirely in the formation functions $w^{(n)}_i$ and $d^{(n)}_i$. The factorial denominators and parity structure remain unchanged, while the formation functions encode the geometry and physics of the underlying problem.

These structures are not postulated a priori. Rather, they arise naturally from the derivation of the solution formulas, indicating that they are intrinsic to the analytical representation of electromagnetic wave propagation.

Since these structures arise from the solution process rather than being postulated, their mathematical properties remain largely unexplored. The formation functions $\mathrm{d}^{(n)}_1(x,y,z,\theta)$ and $\mathrm{d}^{(n)}_2(x,y,z,\theta)$ are determined by the underlying differential equations. For problems outside electromagnetics, analogous formation functions may arise with different mathematical structures.

The quantities introduced previously for notational convenience, namely \eqref{ini.terms.1}, \eqref{ini.terms.2}, \eqref{src.terms.1}, \eqref{src.terms.2}, \eqref{src.terms.3} and \eqref{src.terms.4}, are precisely the formation functions associated with the analytical solutions of Maxwell's equations. They are determined entirely by the initial fields and the source terms. For several representative initial-value and source-driven problems, the following sections present the corresponding deformed trigonometric functions forming the algebraic expressions of the electromagnetic fields.

\subsection{Double-binomial coefficients}
\textbf{Double-binomial coefficients}. For three integers, $h, n$ and $k$, $h \ge 0, n \ge 0$, we introduce the following two families of double-binomial coefficients. 
\begin{equation}
	\label{phnk}
	\begin{split}
		p_{h,n,k} &:= \frac{n!(k+h)!(2(n+h)+1)!}{(n-k)!k!(n+h)!(2(k+h)+1)!}
		\\
		p_{h,n,k} &:= 0, \quad \text{if } k<0 \lor k>n
	\end{split}
\end{equation}

\begin{equation}
	\label{qhnk}
	\begin{split}
		q_{h,n,k} &:= \frac{(n+1)!(k+1+h)!(2(n+h))!}{(n-k)!(k+1)!(n+1+h)!(2(k+h))!}
		\\
		&h \ge 0, n \ge 0, n \ge k \ge 0
		\\
		q_{h,n,k} &:= 0, \quad \text{if } k<-h \lor k>n
	\end{split}
\end{equation}
The ordinary binomial coefficient and the double-binomial coefficients have parallel factorial structures:
\[
\binom{n}{k} = \frac{n!}{(n-k)!k!}, q_{0,n,k} = \frac{(2n)!}{(n-k)!(2k)!}, p_{0,n,k} = \frac{(2n+1)!}{(n-k)!(2k+1)!} .
\]

\subsection{Deformed trigonometric functions in case studies}
For all $\xi,\sigma, \phi, \omega \in  \mathbb{R} $ the following deformed trigonometric functions appearing in case studies are formed by the double-binomial coefficients given by \eqref{phnk} and \eqref{qhnk}.

The double-binomial coefficients reveal that the formation functions of the deformed trigonometric functions remain closely connected to those of the classical sine and cosine functions. Thus, the deformation preserves the underlying Taylor-series structure while enriching the coefficient functions that encode the initial conditions and source terms.

These functions constitute successive extensions of the same deformed trigonometric framework. The additional variables $\omega$ and $\phi$ reflect the increasing complexity of the corresponding initial-value and source-driven problems.

\label{def.cos.sin}
\begin{align}
	\label{cosp}
	\mathrm {cosp} (\xi,\sigma)_{h,s,d} &= \sum_{n=0}^{\infty}\frac{(-1)^n\xi^{2(n+s)}}{(2(n+s))!}\sum_{k=0}^{n+d}(-1)^k p_{h+d,n,k}\sigma^{2k}
	\\
	\label{sinp}
	\mathrm {sinp} (\xi,\sigma)_{h,s,d} &= \sum_{n=0}^{\infty}\frac{(-1)^n\xi^{2(n+s)+1}}{(2(n+s)+1)!}\sum_{k=0}^{n+d}(-1)^k p_{h+d,n,k}\sigma^{2k}
	\\
	\label{cosq}
	\mathrm {cosq} (\xi,\sigma)_{h,s,d} &= \sum_{n=0}^{\infty}\frac{(-1)^n\xi^{2(n+s)}}{(2(n+s))!}\sum_{k=0}^{n+d}(-1)^k q_{h+d,n,k}\sigma^{2k}
	\\
	\label{sinq}
	\mathrm {sinq} (\xi,\sigma)_{h,s,d} &= \sum_{n=0}^{\infty}\frac{(-1)^n\xi^{2(n+s)+1}}{(2(n+s)+1)!}\sum_{k=0}^{n+d}(-1)^k q_{h+d,n,k}\sigma^{2k}
\end{align}

\begin{align}
	\label{cosp.3}
	\mathrm {cosp} (\xi,\sigma, \omega)_{h,s} &= \sum_{n=0}^{\infty}\frac{(-1)^n\xi^{2(n+s)}}{(2(n+s))!} \sum_{m=0}^{n}\omega^{2(n-m)}\sum_{k=0}^{m}(-1)^k p_{h,m,k}\sigma^{2k}
	\\
	\label{sinp.3}
	\mathrm {sinp} (\xi,\sigma, \omega)_{h,s} &= \sum_{n=0}^{\infty}\frac{(-1)^n\xi^{2(n+s)+1}}{(2(n+s)+1)!}\sum_{m=0}^{n}  \omega^{2(n-m)}\sum_{k=0}^{m}(-1)^k p_{h,m,k}\sigma^{2k}
\end{align}

\begin{align}
	\label{cospo.4}
	\mathrm {cospo} (\xi,\sigma, \phi, \omega)_{h,s} &= \sum_{n=0}^{\infty}\frac{(-1)^n\xi^{2(n+s)}}{(2(n+s))!} \sum_{m=0}^{n} f_{o,n-m}(\phi) \omega^m\sum_{k=0}^{m}(-1)^k p_{h,m,k}\sigma^{2k}
	\\
	\label{cospe.4}
	\mathrm {cospe} (\xi,\sigma, \phi, \omega)_{h,s} &= \sum_{n=0}^{\infty}\frac{(-1)^n\xi^{2(n+s)}}{(2(n+s))!} \sum_{m=0}^{n} f_{e,n-m}(\phi) \omega^m\sum_{k=0}^{m}(-1)^k p_{h,m,k}\sigma^{2k}
	\\
	\label{sinpo.4}
	\mathrm {sinpo} (\xi,\sigma, \phi, \omega)_{h,s} &= \sum_{n=0}^{\infty}\frac{(-1)^n\xi^{2(n+s)+1}}{(2(n+s)+1)!} \sum_{m=0}^{n} f_{o,n-m}(\phi) \omega^m\sum_{k=0}^{m}(-1)^k p_{h,m,k}\sigma^{2k}
	\\
	\label{sinpe.4}
	\mathrm {sinpe} (\xi,\sigma, \phi, \omega)_{h,s} &= \sum_{n=0}^{\infty}\frac{(-1)^n\xi^{2(n+s)+1}}{(2(n+s)+1)!} \sum_{m=0}^{n} f_{e,n-m}(\phi) \omega^m\sum_{k=0}^{m}(-1)^k p_{h,m,k}\sigma^{2k}
\end{align}

where functions $f_{o,n}(\phi) $ and $f_{e,n}(\phi) $ are defined by
\begin{align}
	\label{ricker.fo}
	f_{o,n}(\phi) &= \sum_{k=0}^{n+1}(-1)^kp_{0,n+1,k}(2\phi)^{2k+1}
	\\
	\label{ricker.fe}
	f_{e,n}(\phi) &= \sum_{k=0}^{n+1}(-1)^kq_{0,n+1,k}(2\phi)^{2k}
\end{align}

\subsection{Formation and deformation}

Table \ref{tbl.deforms} summarizes several connections between classical and deformed trigonometric functions. Its purpose is
to illustrate their common structure and the mechanism of deformation rather than to provide a comprehensive comparison.

\begin{table}[H]
	\centering
	\caption{Comparison of classical and deformed trigonometric structures}
	\label{tbl.deforms}
	\begin{tabular}{
			>{\raggedright\arraybackslash}p{0.20\textwidth}
			>{\raggedright\arraybackslash}p{0.35\textwidth}
			>{\raggedright\arraybackslash}p{0.35\textwidth}
			}
		\hline
		Comparison
		& Classical
		& Deformed
		\\
		\hline
		Representative functions
		& $\sin(x-\theta)$, $\cos(x-\theta)$
		& Solutions of Maxwell's equations
		\\
		\hline
		
		Series structure
		&
		$\displaystyle
		\sum_{n=0}^{\infty}
		\frac{(-1)^n\theta^{2n+1}}{(2n+1)!}
		w_1^{(n)}(x,\theta)$
		&
		$\displaystyle
		\sum_{n=0}^{\infty}
		\frac{(-1)^n\theta^{2n+1}}{(2n+1)!}
		d_1^{(n)}(x,y,z,\theta)$
		\\
		\hline
		
		Formation functions
		&
		$w_1^{(n)}$ and $w_2^{(n)}$, formed using ordinary binomial coefficients $C^{2n+1}_{k}$ and $C^{2n}_k$
		&
		$d_1^{(n)}$ and $d_2^{(n)}$, formed using the double-binomial coefficients
		$p_{h,n,k}$ and $q_{h,n,k}$
		\\
		\hline
		
		Representative coefficients
		&
		$\displaystyle 
		C^{n}_k = \frac{n!}{(n-k)!k!}$
		&
		$\displaystyle
		p_{0,n,k}
		=
		\frac{(2n+1)!}{(n-k)!(2k+1)!}$
		\\
		\hline
	\end{tabular}
	
\end{table}

\section{\textbf{Relation to existing analytical methods}}
\label{sect.relation.green}
One of the classical analytical approaches to solving linear differential equations is to represent the solution through an integral operator. Green's function theory provides a general framework for constructing such representations. For a linear differential operator $L$, 
\[
L u = f, 
\]
Green's function theory realizes the inverse operator $u = L^{-1}f $ through an integral kernel
\[
u(t) = \int G(t,\tau) f(\tau) d\tau .
\]

To solve Maxwell's equations in this manner, a popular approach is to introduce a scalar potential $\phi$ and a vector potential $A$ to represent electromagnetic fields:
\[
B=\nabla\times A, E=-\nabla \phi - \frac{\partial A}{\partial t}.
\]
Under the Lorenz gauge, the scalar and vector potentials satisfy decoupled wave equations,
\[
\nabla^2\phi - \frac{1}{c^2} \frac{\partial^2\phi}{\partial t^2} = -\frac{\rho}{\epsilon_0}   ,
\]
\[
\nabla^2 A - \frac{1}{c^2} \frac{\partial^2 A}{\partial t^2} = -\mu_0 J ,
\]
 to which Green's functions theory may apply. In this approach, Jefimenko equations express the electromagnetic fields through retarded integral kernels,
\[
E = \frac{1}{4\pi \epsilon_0} \iiint \left( \frac{(\rho-r/c)\vec{n}}{r^2}+\frac{\dot{\rho}(t-r/c)\vec{n}}{cr}-\frac{\dot{J(t-r/c)}}{c^2 r} \right)dV' .
\]

If one seeks an explicit analytical expression that no longer contains integral operators, obtaining the resulting integrals in closed analytical form is often difficult.
On the other hand, given a function, it is relatively easy to find its derivatives of arbitrary orders. That is how the present work tries to exploit.

Rather than constructing the solution through an integral representation, the present work constructs it through successive differentiation. Arbitrary-order derivatives are generated, related through the Time–Space theorem, and assembled into convergent infinite-order expansions.

Moreover, this process naturally introduces new functions whenever they are required by the analytical solution. In the present work, these appear as the deformed sine and cosine functions arising from the extended Taylor expansions.

While causality is made explicit through the retarded integral kernels in the Jefimenko equations, in the present work it is encoded implicitly in the time-analytic expansion of the solution. The Taylor series determines the fields at later times from the initial values and source terms, with the Time–Space theorem relating temporal derivatives to spatial curl operators.

In the Jefimenko equations, the mapping from the source and initial data to the electromagnetic fields is realized through retarded integral operators. In the present work, the same mapping is realized through infinite-order differential operators. The infinite-order curl expansions provide the explicit analytical realization of this operator representation.

The general solution formulas are expressed in terms of infinite-order differential operators acting on the prescribed initial fields or source functions. Once these operators are applied to a specific analytical function, they are evaluated explicitly, yielding closed-form algebraic expressions for the electromagnetic fields. The representative case studies in Section \ref{sect.caseStudies} illustrate this process. By contrast, Green-function and Jefimenko formulations remain expressed in terms of retarded integral operators unless the corresponding integrals can also be evaluated analytically.

\begin{table}[H]
	\centering
	\caption{Comparison of two analytical solutions of Maxwell's equations}
	\label{tbl.green}
	\begin{tabular}{
			>{\raggedright\arraybackslash}p{0.20\textwidth}
			>{\raggedright\arraybackslash}p{0.35\textwidth}
			>{\raggedright\arraybackslash}p{0.35\textwidth}
			}
		\hline
		Aspect
		& Green/Jefimenko approach
		& Present work
		\\
		\hline
		
		Unknowns
		&
		Potentials or fields
		&
		Fields directly
		\\
		\hline

		Solution construction
		&
		Integral operators
		&
		Infinite-order differential operators
		\\
		\hline
		
		Causality
		&
		Retarded integration
		&
		Time-analytic expansion
		\\
		\hline

		Explicit field representation
		&
		Integral representation
		&
		Explicit algebraic expressions after operator evaluation
		\\
		\hline
	\end{tabular}
\end{table}

\begin{table}[H]
	\centering
	\caption{Symbols used in the case studies}
	\label{tbl.symbols}
	\begin{tabular}{l l l}
		\hline
		Symbol & Definition & Description \\ %
		\hline
		$a $ & $a \in \mathbb{R}_{>0}$ & a constant in function $exp(-ar^2)$ \\ %
		\hline
		$g$ & $g=2a$ & a factor in $exp(-ar^2) $ related derivatives \\ %
		\hline
		$r^2 $ & $r^2=x^2+y^2+z^2 $ & space radius \\ %
		\hline
		$w $ & $w=gr^2 $ & a factor appears in $exp(-ar^2) $ related derivatives \\ %
		\hline
		$\vec{z} $ & $\vec{z}=\begin{bmatrix}
			0 \\
			0 \\
			1
		\end{bmatrix} $ & a vector constant \\ %
		\hline
		$\vec{s} $ & $\vec{s}=\begin{bmatrix}
			x \\
			y \\
			z
		\end{bmatrix} $ & a vector variable \\ %
		\hline 
		$\vec{c} $ & $\vec{c}=\begin{bmatrix}
			x \\
			0 \\
			-z
		\end{bmatrix} $ & a vector variable \\ %
		\hline 
		$\vec{f} $ & $\vec{f}=\begin{bmatrix}
			y \\
			-x \\
			0
		\end{bmatrix} $ & a vector variable \\ %
		\hline 
		$\vec{b} $ & $\vec{b}=\begin{bmatrix}
			yz \\
			-2zx \\
			xy
		\end{bmatrix} $ & direction of an initial value \\ %
		\hline 
		$\vec{d} $ & $\vec{d}=g(x^2-z^2)\vec{s} $ & a vector forming the initial value solution \\ %
		\hline
		$\vec{h} $ & $\vec{h}=g\vec{f} $ & a vector forming the source-driven solutions \\ %
		\hline
		$\vec{e} $ & $\vec{e}=gz\vec{s} $ & a vector forming the source-driven solutions \\ %
		\hline
	\end{tabular}
	
\end{table}

The two approaches are complementary rather than competing. Under the assumptions for which uniqueness holds, both approaches represent the same physical solution for the same initial conditions and source terms, although they construct that solution through fundamentally different mathematical representations.

Establishing an explicit correspondence between the integral-operator formulation and the differential-operator formulation developed in this work would be an interesting direction for future research. Such a correspondence could provide additional insight into the relationship between these two representations and, for example, help clarify the role of the Lorenz gauge, which does not appear explicitly in the present formulation.

Table \ref{tbl.green} lists some aspects of these approaches.

\section{\textbf{Representative case studies}}
\label{sect.caseStudies}

The general solution formulas developed in the preceding sections yield explicit analytical field expressions by evaluating the infinite-order differential operators for prescribed initial fields and source functions. The following representative initial-value and source-driven problems illustrate this process and demonstrate the resulting explicit algebraic field expressions. 

The representative case studies also demonstrate that explicit analytical field expressions facilitate the identification of structural properties of the electromagnetic fields, such as geometric relationships and propagation characteristics, that are considerably more difficult to infer from discrete simulation data alone.

The symbols used in these case studies are defined in Table \ref{tbl.symbols}.

Numerical results and comparisons with the FDTD method are presented separately in Section \ref{sect.numeric}.

\subsection{\textbf{Case-study 1:Gaussian initial-value}} 
The following initial values are used.
\begin{align}
	\label{3d.exa.ini.H}
	F_h(x,y,z) &= 0
\\
	\label{3d.exa.ini.E}
	F_e(x,y,z) &= \exp(-ar^2) \vec{b}
\end{align}
By applying (\ref{curl.exp.a.b.2n}) and (\ref{curl.exp.a.b.2n1}) to (\ref{3d.exa.ini.E}) and substituting the results into (\ref{3d.solIniH}) and (\ref{3d.solIniE}), we obtain the following closed-form analytical solution:
\begin{equation*}
	\label{3d.exa.sol.ini.HE}
	\begin{split}
		H(x,y,z,\theta) &= \\
		&-\frac{1}{\eta \sqrt{a}}\exp({-ar^2})  \mathrm {sinp} (\xi,\sigma)_{3,0,0} \vec{d} \\
		&+ \frac{1}{\eta \sqrt{a}}\exp({-ar^2})\left(\frac{1}{2} \mathrm {sinp} (\xi,\sigma)_{1,0,1}+2\mathrm {sinp} (\xi,\sigma)_{2,0,0}  \right)\vec{c} \\
		E(x,y,z,\theta) &= \exp({-ar^2}) \mathrm {cosp} (\xi,\sigma)_{2,0,0} \vec{b}
	\end{split}
\end{equation*}
where $\xi=\sqrt{a}\theta, \sigma=2\sqrt{a}r $.

From the analytical solution, the following characteristics of the electromagnetic fields can be observed:
\begin{itemize}
	\item The direction of the electric field is $\vec{b} $, which remains constant over time.
	\item The magnetic field is confined to a two-dimensional plane spanned by the vectors $\vec{c} $ and $\vec{s} $, as $\vec{d} = g(x^2-z^2)\vec{s} $.
	\item Since $\vec{c} \cdot \vec{b} =0 $ and $\vec{s} \cdot \vec{b} = 0 $, it follows that $H \cdot E = 0 $. Thus, the electric and magnetic fields are everywhere orthogonal.
	\item For energy transfer, using
	\begin{align*}
		\vec{b} \times \vec{c} &= (x^2+y^2)\vec{s} - (x^2-z^2)\vec{c},
		\\
		\vec{b} \times \vec{s} &= (x^2-z^2)\vec{s} - r^2 \vec{c} ,
	\end{align*}
	it follows that the Poynting vector is also confined to the plane spanned by $\vec{c} $ and $\vec{s} $.
	\item Since an analytical expression for the Poynting vector is available, the characteristics of energy transfer can be examined analytically in a systematic manner.
\end{itemize}

\subsection{\textbf{Case-study 2: Gaussian source}}
The following source is used.
\begin{equation}
	\label{3d.exa.src}
	J(x,y,z,\theta) = \exp({-ar^2})\vec{z}
\end{equation}
(\ref{3d.exa.src}) gives
\begin{equation}
	\label{3d.exa.src.time}
	\frac{\partial^n J(x,y,z,\theta)}{\partial \theta^n} = 0; \forall n>0 .
\end{equation}
By applying (\ref{curl.exp.a.z.2n1}) and (\ref{curl.exp.a.z.2n2}) to (\ref{3d.exa.src}) and substituting the results and (\ref{3d.exa.src.time}) into (\ref{3d.solSrcH}) and (\ref{3d.solSrcE}), we obtain the following closed-form analytical solution.
\begin{align}
	\label{3d.exa.sol.src.H}
	H(x,y,z,\theta) &= -\frac{1}{a} \exp({-ar^2}) \mathrm{cosp}(\xi,\sigma)_{1,1,0} \vec{h}
\\
	\label{3d.exa.sol.src.E}
	\begin{split}
	E(x,y,z,\theta) &= \frac{2\eta}{\sqrt{a}}\exp(-ar^2)\mathrm{sinp}(\xi,\sigma)_{2,1,0}\vec{e} \\
	&- \frac{\eta}{\sqrt{a}}\exp(-ar^2) \left( \xi - 4\mathrm{sinp}(\xi,\sigma)_{1,1,0} + \sigma^2\mathrm{sinp}(\xi,\sigma)_{2,1,0} \right) \vec{z}
	\end{split}
\end{align}
where $\xi=\sqrt{a}\theta, \sigma=2\sqrt{a}r $.

From the analytical solution, the following characteristics of the electromagnetic fields can be observed:
\begin{itemize}
	\item The direction of the magnetic field is $\vec{f} $, as $\vec{h} = g\vec{f} $, which remains constant over time.
	\item The electric field is confined to a two-dimensional plane spanned by the vectors $\vec{z} $ and $\vec{s} $, as $\vec{e} = gz\vec{s} $.
	\item Since $\vec{z} \cdot \vec{f} =0 $ and $\vec{s} \cdot \vec{f} = 0 $, it follows that $H \cdot E = 0 $. Thus, the electric and magnetic fields are everywhere orthogonal.
	\item For energy transfer, using
	\begin{align*}
		\vec{f} \times \vec{z} &= -\vec{s} + z\vec{z} ,
		\\
		\vec{f} \times \vec{s} &= -z\vec{s} + r^2 \vec{z} ,
	\end{align*}
	it follows that the Poynting vector is also confined to the plane spanned by $\vec{s} $ and $\vec{z} $.
	\item Since an analytical expression for the Poynting vector is available, the characteristics of energy transfer can be examined analytically in a systematic manner.
\end{itemize}

\subsection{\textbf{Case-study 3: harmonic source}}
The following harmonic source \cite{fdtd:Schneider} is used.
\begin{equation}
	\label{3d.exa.src.cos}
	J(x,y,z,\theta) = \exp({-ar^2})\cos(\omega \theta) \vec{z}
\end{equation}
where $a, \omega \in \mathbb{R}_{>0}$ are constant.

The temporary derivatives of the source are 
\begin{equation}
	\label{case3.time.even}
	\frac{\partial ^{2n} \cos(\omega \theta)}{\partial \theta^{2n}} = (-1)^n \omega^{2n} \cos(\omega \theta)
\end{equation}
\begin{equation}
	\label{case3.time.odd}
	\frac{\partial ^{2n+1} \cos(\omega \theta)}{\partial \theta^{2n+1}} = (-1)^{n+1} \omega^{2n+1} \sin(\omega \theta)
\end{equation}
By applying (\ref{case3.time.even}), (\ref{case3.time.odd}), (\ref{curl.exp.a.z.2n1}) and (\ref{curl.exp.a.z.2n2}) to (\ref{3d.exa.src.cos}) and substituting the results into (\ref{3d.solSrcH}) and (\ref{3d.solSrcE}), we obtain the following closed-form analytical solution:
\begin{align}
	\label{3d.exa.sol.src.cos.H}
	H(x,y,z,\theta) &= -\frac{1}{a} \exp({-ar^2}) \mathrm{cosp}(\xi,\sigma, \omega_a)_{1,1} \vec{h}
\\
	\label{3d.exa.sol.src.cos.E}
	\begin{split}
		E(x,y,z,\theta) &= \frac{2\eta}{\sqrt{a}}\exp(-ar^2)\mathrm{sinp}(\xi,\sigma,\omega_a)_{2,1}\vec{e} \\
		&- \frac{\eta}{\sqrt{a}}\exp(-ar^2) f_z(\xi,\sigma,\omega_a) \vec{z} \\
		f_z(\xi,\sigma,\omega_a) &= \frac{\sin(\omega_a\xi)}{\omega_a} - 4\mathrm{sinp}(\xi,\sigma,\omega_a)_{1,1} + \sigma^2\mathrm{sinp}(\xi,\sigma,\omega_a)_{2,1}
	\end{split}
\end{align}
where $\xi=\sqrt{a}\theta, \sigma=2\sqrt{a}r, \omega_a=\omega/\sqrt{a} $.

Similar analytical insights can be obtained from the above solution formulas as in Case Study 2.

\subsection{\textbf{Case-study 4: Ricker wavelet source}}
The Ricker wavelet source \cite{fdtd:Schneider} is given by 
\begin{equation*}
	f_r(t) = (1-2(\pi f_p (t-d_r)^2))\exp(-(\pi f_p(t-d_r))^2)
\end{equation*}
where $f_p $ is the "peak frequency" and $d_r $ is the temporal delay.

Let
\begin{equation*}
	\begin{split}
		\theta &= ct \\
		\theta_d &= cd_r \\
		b &= \left(\frac{\pi f_p}{c} \right)^2
	\end{split}
\end{equation*}
The source term becomes
\begin{equation}
	\label{3d.exa.src.ricker}
	J(x,y,z,\theta) = \exp({-ax^2}) (1-2b(\theta-\theta_d)^2 )\exp(-b(\theta-\theta_d)^2) \vec{z}
\end{equation}
where $a, b, \theta_d \in \mathbb{R}_{>0}$ are three constants.

Notice that (\ref{3d.exa.src.ricker}) can be written as 
\begin{equation}
	\label{3d.exa.src.ricker.2}
	J(x,y,z,\theta) =-\frac{1}{2b} \exp({-ax^2}) \frac{\partial^2 \exp(-b(\theta-\theta_d)^2)}{\partial \theta^2}  \vec{z}
\end{equation}

Apply formulas (A.9) and (A.10) in \cite{GE2026100688} to (\ref{3d.exa.src.ricker.2}), we have the temporal derivatives of the source:
\begin{align}
	\label{ricker.temp.2n}
	\begin{split}
	&\frac{\partial ^{2n} J(x,y,z,0)}{\partial \theta^{2n}} \\
	&= \frac{1}{2} \exp(-b{\theta_d}^2)(-1)^nb^n\sum_{k=0}^{n+1}(-1)^kq_{0,n+1,k}(2\sqrt{b}\theta_d)^{2k}\exp(-ar^2) \vec{z}  
	\end{split}
	\\
	\label{ricker.temp.2n1}
	\begin{split}
	&\frac{\partial ^{2n+1} J(x,y,z,0)}{\partial \theta^{2n+1}} \\
	&= \frac{1}{2} \exp(-b{\theta_d}^2)(-1)^{n+1}{\sqrt{b}}^{2n+1}\sum_{k=0}^{n+1}(-1)^kp_{0,n+1,k}(2\sqrt{b}\theta_d)^{2k+1}\exp(-ar^2) \vec{z}  
	\end{split}
\end{align}

By applying (\ref{ricker.temp.2n}), (\ref{ricker.temp.2n1}), (\ref{curl.exp.a.z.2n1}) and (\ref{curl.exp.a.z.2n2}) to (\ref{3d.exa.src.ricker.2}) and substituting the results into (\ref{3d.solSrcH}) and (\ref{3d.solSrcE}), we obtain the following closed-form analytical solution:
\begin{align}
	\label{3d.exa.sol.src.ricker.H}
	H(x,y,z,\theta) &= \exp({-ar^2-b{\theta_d}^2}) \mathrm{f_h}(\xi, \sigma, \xi_d, a_b) \vec{h}
\\
	\label{3d.exa.sol.src.ricker.E}
	E(x,y,z,\theta) &= \frac{\eta}{\sqrt{b}} \exp({-ar^2-b{\theta_d}^2}) (\mathrm{f_z}(\xi,\sigma, \xi_d, a_b)\vec{z}+\mathrm{f_e}(\xi,\sigma, \xi_d, a_b)\vec{e}) 
\end{align}
where $\xi=\sqrt{b}\theta, \xi_d = \sqrt{b}\theta_d, \sigma=2\sqrt{a}r, a_b=a/b $. Functions $\mathrm{f_h}, \mathrm{f_z} $ and $\mathrm{f_e} $ are given by
\begin{align}
	\label{ricker.src.fh}
	\mathrm{f_h}(\xi,\sigma, \phi, \omega) &= \frac{1}{2b} \left( \mathrm{sinpo}(\xi,\sigma,\phi, \omega)_{1,1} - \mathrm{cospe}(\xi,\sigma,\phi, \omega)_{1,1} \right)
\\
	\label{ricker.src.fe}
	\mathrm{f_e}(\xi,\sigma, \phi, \omega) &= \frac{a}{b} \left(\mathrm{sinpe}(\xi,\sigma,\phi, \omega)_{2,1} - \mathrm{cospo}(\xi,\sigma,\phi, \omega)_{2,2} \right)
\\
	\label{ricker.src.fz}
	\begin{split}
		\mathrm{f_z}(\xi,\sigma, \phi, \omega) &= \frac{1}{2} \mathrm{f_\xi}(\xi, \phi)  \\
		&+ \frac{a}{2b}\sigma^2 \mathrm{cospo}(\xi,\sigma,\phi,\omega)_{2,2} - \frac{2a}{b} \mathrm{cospo}(\xi,\sigma,\phi,\omega)_{1,2} \\
		&+ \frac{2a}{b} \mathrm{sinpe}(\xi,\sigma,\phi,\omega)_{1,1} - \frac{a}{2b} \sigma^2 \mathrm{sinpe}(\xi,\sigma,\phi,\omega)_{2,1}
	\end{split}
\end{align}
where
\begin{equation}
	\label{ricker.src.fxi}
	\begin{split}
		\mathrm{f_\xi}(\xi, \phi) &= 2 \phi (3-2\phi^2)\xi^2 - 2(1-2\phi^2)\xi \\
		&-2\phi \mathrm{cosp}(\xi,2\phi)_{0,2,2} + \mathrm{sinq}(\xi,2\phi)_{0,1,2}
	\end{split}
\end{equation}

Similar analytical insights can be obtained from the above solution formulas, as demonstrated in Case Study 2.

\section{\textbf{Numerical results}}
\label{sect.numeric}
\subsection{Numerical evaluation of analytical solutions}

Unlike the standard trigonometric functions, the deformed sine and cosine functions are non-periodic, which makes their numerical evaluation increasingly demanding for large spatial and temporal values. 
For the parameter ranges considered in this paper, neither double-precision arithmetic nor 500-digit precision produced numerically stable results because intermediate terms became extremely large during the evaluation of the truncated series. Calculations were therefore performed using arbitrary-precision arithmetic with 5000 decimal digits. To verify that this precision was sufficient, the computations were repeated using 5500 decimal digits, and no observable differences were found within the adopted working precision.

The infinite series were truncated after 600 terms. Under these settings, the analytical solutions exhibited stable convergence throughout the numerical examples presented in this paper. These parameter choices serve as a reliable reference implementation rather than an optimized computational strategy. The efficient numerical evaluation of the non-periodic functions introduced in this work remains an interesting topic for future investigation.

Electromagnetic field snapshots are presented within an 8-meter cube:
\begin{equation}
	\label{num.domain}
	\begin{split}
	&H(i\bigtriangleup_s,j\bigtriangleup_s,k\bigtriangleup_s,q\bigtriangleup_{\theta}) \\
	&E(i\bigtriangleup_s,j\bigtriangleup_s,k\bigtriangleup_s,q\bigtriangleup_{\theta}) \\
	&i,j,k = 0,\pm 1,\pm 2,...,\pm 10 \\
	&q = 0,1,2,...,50	
	\end{split}
\end{equation}

3D drawing of the fields are presented in section \ref{sect.3d}.

\subsection{Comparisons with FDTD data}
\label{sect.fdtd.verify}
For the calculations presented in this section, the FDTD computational domains are chosen sufficiently large that the field data within the central 8-meter cube defined by \eqref{num.domain} remain unaffected by boundary effects. Data outside this region are discarded.

Because of the staggered-grid structure of the FDTD algorithm, the six electromagnetic field components are not defined at identical spatial locations. Direct comparisons therefore require separate analytical evaluations at the corresponding Yee-grid locations of each field component.

The FDTD results are numerical approximations whose accuracy depends on parameters such as the spatial step size, time step, Courant number, computational domain, and boundary treatment. Consequently, pointwise differences between the analytical and FDTD results cannot, by themselves, establish or invalidate the mathematical correctness of the analytical solutions. The comparisons below therefore focus primarily on whether the FDTD results exhibit field patterns that are visually consistent with those obtained from the corresponding analytical expressions.

The comparisons are presented from two complementary perspectives. We first zoom in on the $E_x$ component along a prescribed line defined by \eqref{fdtd.Ex} to examine detailed data patterns. We then zoom out to the three-dimensional electromagnetic fields to compare their overall spatial structures. This broader view complements the preceding line-based comparison by illustrating the overall spatial structure of the analytical and FDTD fields.

\begin{equation}
	\label{fdtd.Ex}
	\begin{split}
		&E_x(s \bigtriangleup_s + \bigtriangleup_s/2,5\bigtriangleup_s,5\bigtriangleup_s, q\bigtriangleup_\theta) \\	
		&s=0,1,2,...,10 \\
		&q=0,1,2,...,50
	\end{split}
\end{equation}

\subsection{Zoom-in comparison: Ex patterns}

\begin{figure}[H]
	\begin{subfigure}{.5\textwidth}
		\includegraphics[width=.9\linewidth]{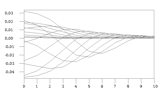}
		\caption{$Ex $ by analytical solution}
		\label{fig:exa.1.Ex.analytic}
	\end{subfigure}%
	\begin{subfigure}{.5\textwidth}
		\includegraphics[width=.9\linewidth]{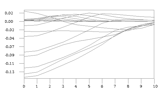}
		\caption{$Ex $ by FDTD}
		\label{fig:exa.1.Ex.fdtd}
	\end{subfigure}
	\caption{$E_x $ by the analytical solution and by FDTD simulation (case study 1).}
	\label{fig:exa.1.Ex}
\end{figure}
\textbf{Case 1}. Figure \ref{fig:exa.1.Ex} compares the $E_x$ component obtained from the analytical solution with that obtained from the FDTD simulation. Figure \ref{fig:exa.1.Ex.analytic} presents snapshots of $E_x(x+\bigtriangleup_s/2,y,z,\theta)$ evaluated using the analytical formulas, whereas Figure \ref{fig:exa.1.Ex.fdtd} shows the corresponding FDTD results. The two sets of snapshots exhibit similar overall field structures, although noticeable quantitative differences remain.

\begin{figure}[H]
	\begin{subfigure}{.5\textwidth}
		\includegraphics[width=.9\linewidth]{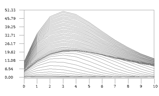}
		\caption{$Ex $ by analytical solution}
		\label{fig:exa.2.Ex.analytic}
	\end{subfigure}%
	\begin{subfigure}{.5\textwidth}
		\includegraphics[width=.9\linewidth]{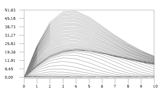}
		\caption{$Ex $ by FDTD}
		\label{fig:exa.2.Ex.fdtd}
	\end{subfigure}
	\caption{$E_x $ by the analytical solution and by FDTD simulation (case study 2).}
	\label{fig:exa.2.Ex}
\end{figure}
\textbf{Case 2}. Figure \ref{fig:exa.2.Ex.analytic} presents snapshots of $E_x(x+\bigtriangleup_s/2,y,z,\theta)$ evaluated using the analytical formulas, whereas Figure \ref{fig:exa.2.Ex.fdtd} shows the corresponding FDTD results. The two sets of snapshots exhibit closely similar spatial patterns, with only relatively small visible differences.

\begin{figure}[H]
	\begin{subfigure}{.5\textwidth}
		\includegraphics[width=.9\linewidth]{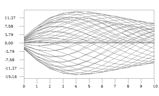}
		\caption{$Ex $ by analytical solution}
		\label{fig:exa.3.Ex.analytic}
	\end{subfigure}%
	\begin{subfigure}{.5\textwidth}
		\includegraphics[width=.9\linewidth]{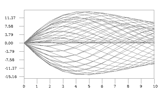}
		\caption{$Ex $ by FDTD}
		\label{fig:exa.3.Ex.fdtd}
	\end{subfigure}
	\caption{$E_x $ by the analytical solution and by FDTD simulation (case study 3).}
	\label{fig:exa.3.Ex}
\end{figure}
\textbf{Case 3}. Figure \ref{fig:exa.3.Ex.analytic} presents snapshots of $E_x(x+\bigtriangleup_s/2,y,z,\theta)$ evaluated using the analytical formulas, whereas Figure \ref{fig:exa.3.Ex.fdtd} shows the corresponding FDTD results. The analytical and FDTD snapshots are visually nearly indistinguishable at the scale shown.

\begin{figure}[H]
	\begin{subfigure}{.5\textwidth}
		\includegraphics[width=.9\linewidth]{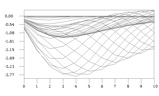}
		\caption{$Ex $ by analytical solution}
		\label{fig:exa.4.Ex.analytic}
	\end{subfigure}%
	\begin{subfigure}{.5\textwidth}
		\includegraphics[width=.9\linewidth]{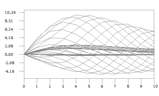}
		\caption{$Ex $ by FDTD}
		\label{fig:exa.4.Ex.fdtd}
	\end{subfigure}
	\caption{$E_x $ by the analytical solution and by FDTD simulation (case study 4).}
	\label{fig:exa.4.Ex}
\end{figure}
\textbf{Case 4}. Figure \ref{fig:exa.4.Ex.analytic} presents snapshots of $E_x(x+\bigtriangleup_s/2,y,z,\theta)$ evaluated using the analytical formulas, whereas Figure \ref{fig:exa.4.Ex.fdtd} shows the corresponding FDTD results. Substantial quantitative differences are visible between the two sets of results. Nevertheless, the snapshots retain similarities in their overall spatial patterns. These similarities are also apparent when the Case 4 results are viewed together with the patterns obtained in the preceding three cases.

\subsection{Zoom-out comparison: three-dimensional field patterns}
\label{sect.3d}

In the figures, colored three-dimensional grids indicate coordinate planes: (x,y), (y,z), and (z,x) are shown in red, blue, and yellow, respectively. Vectors are depicted as straight line segments in three dimensions, with lengths proportional to their magnitudes. A consistent length scale is maintained within each figure group.

\subsubsection{3D data patterns: Case 1}

\begin{figure}[H]
	\begin{subfigure}{.33\textwidth}
		\includegraphics[width=.82\linewidth]{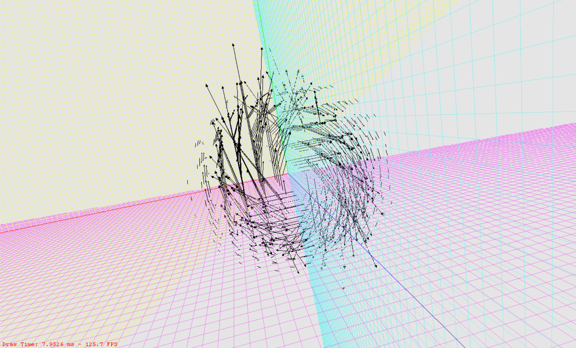}
		\caption{$E$ at $\theta=0.23$}
	\end{subfigure}%
	\begin{subfigure}{.33\textwidth}
		\includegraphics[width=.82\linewidth]{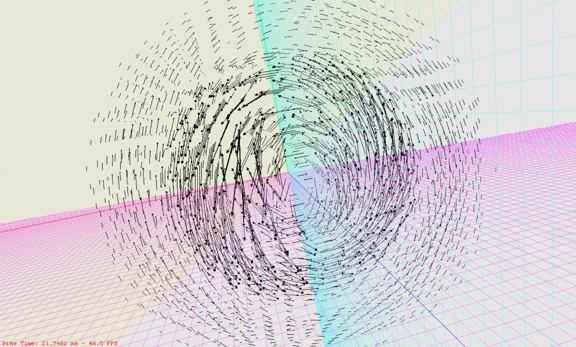}
		\caption{$E$ at $\theta=2.3$}
	\end{subfigure}%
	\begin{subfigure}{.33\textwidth}
		\includegraphics[width=.82\linewidth]{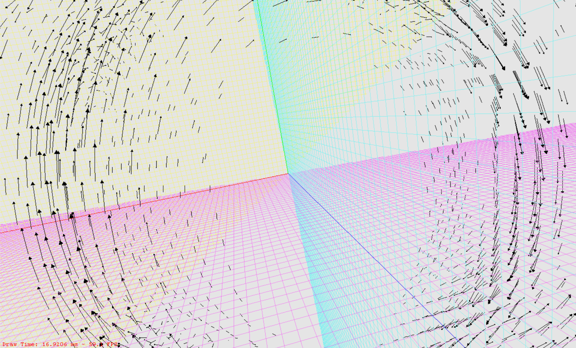}
		\caption{$E$ at $\theta=5.75$}
	\end{subfigure}
	\caption{Electric field by analytical solution. $q=1,10,25$ (Case-study 1).}
	\label{fig:exa.1a.E}
\end{figure}
\begin{figure}[H]
	\begin{subfigure}{.33\textwidth}
		\includegraphics[width=.82\linewidth]{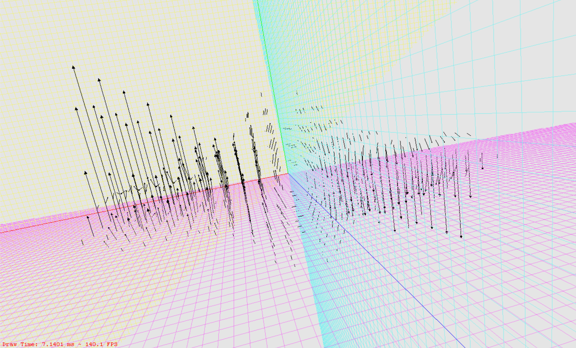}
		\caption{$E$ at $\theta=0.23$}
	\end{subfigure}%
	\begin{subfigure}{.33\textwidth}
		\includegraphics[width=.82\linewidth]{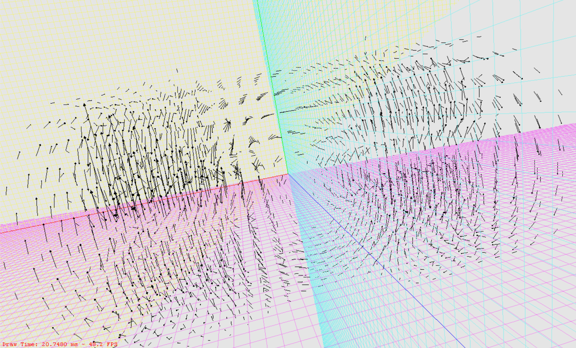}
		\caption{$E$ at $\theta=2.3$}
	\end{subfigure}%
	\begin{subfigure}{.33\textwidth}
		\includegraphics[width=.82\linewidth]{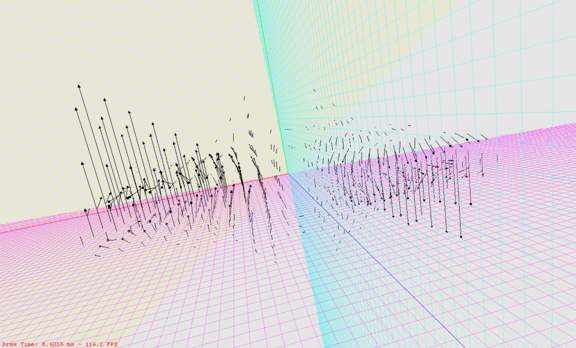}
		\caption{$E$ at $\theta=5.75$}
	\end{subfigure}
	\caption{Electric field by FDTD. $q=1,10,25$ (Case-study 1).}
	\label{fig:exa.1a.fdtd.E}
\end{figure}	
Figure \ref{fig:exa.1a.E} shows 3 snapshots of the electric field obtained from the analytical solution, while Figure \ref{fig:exa.1a.H}  presents the corresponding magnetic field snapshots. Figures \ref{fig:exa.1a.fdtd.E} and \ref{fig:exa.1a.fdtd.H} show the FDTD simulation results, with vector components visualized at coincident spatial locations for comparison.
\begin{figure}[H]
	\begin{subfigure}{.33\textwidth}
		\includegraphics[width=.82\linewidth]{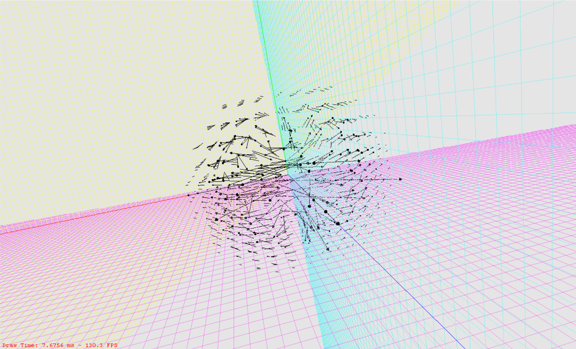}
		\caption{$H$ at $\theta=0.23$}
	\end{subfigure}%
	\begin{subfigure}{.33\textwidth}
		\includegraphics[width=.82\linewidth]{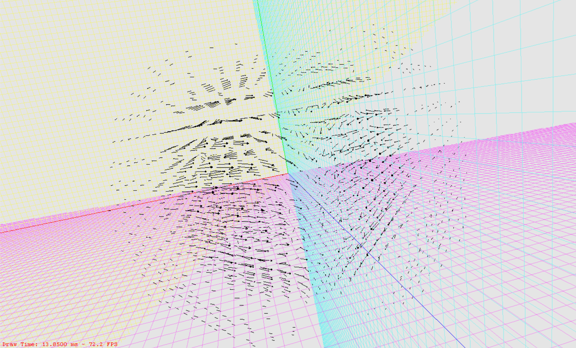}
		\caption{$H$ at $\theta=2.3$}
	\end{subfigure}%
	\begin{subfigure}{.33\textwidth}
		\includegraphics[width=.82\linewidth]{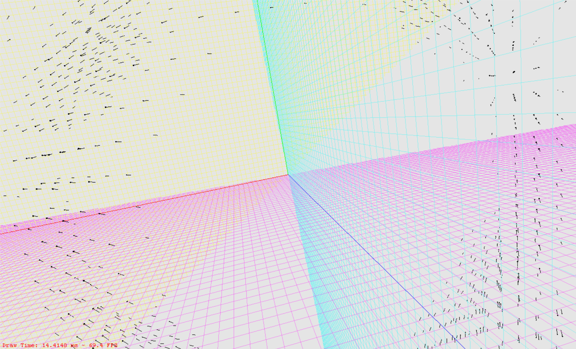}
		\caption{$H$ at $\theta=5.75$}
	\end{subfigure}
	\caption{Magnetic field by analytical solution. $q=1,10,25$ (Case-study 1).}
	\label{fig:exa.1a.H}
\end{figure}
\begin{figure}[H]
	\begin{subfigure}{.33\textwidth}
		\includegraphics[width=.82\linewidth]{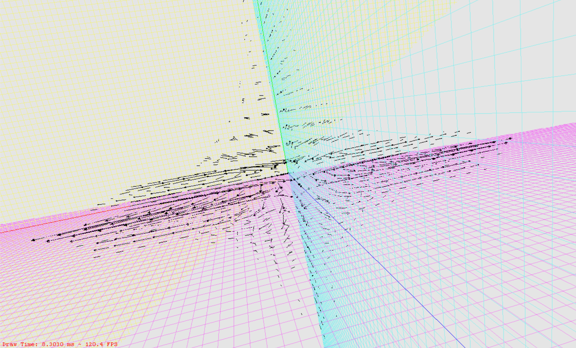}
		\caption{$H$ at $\theta=0.23$}
	\end{subfigure}%
	\begin{subfigure}{.33\textwidth}
		\includegraphics[width=.82\linewidth]{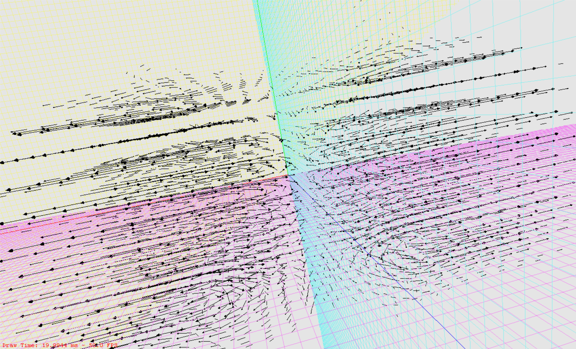}
		\caption{$H$ at $\theta=2.3$}
	\end{subfigure}%
	\begin{subfigure}{.33\textwidth}
		\includegraphics[width=.82\linewidth]{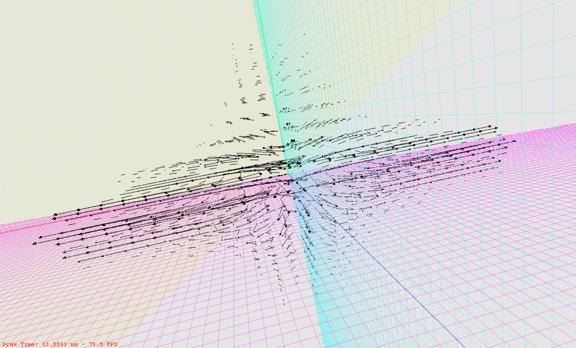}
		\caption{$H$ at $\theta=5.75$}
	\end{subfigure}
	\caption{Magnetic field by FDTD. $q=1,10,25$ (Case-study 1).}
	\label{fig:exa.1a.fdtd.H}	
\end{figure}
These three-dimensional visualizations show significant discrepancies between the analytical and FDTD results for this case. 
The analytical field exhibits the expected spatial symmetry, whereas the FDTD field shows a noticeable asymmetry. 
To investigate the origin of this discrepancy, we examined the initialization of the electric field on the staggered Yee grid, as shown by \eqref{gsini.stag}. The analysis indicates that the observed asymmetry originates from the staggered placement of the field components in the prescribed initial data.
\begin{equation}
	\label{gsini.stag}
	F_e(x,y,z) = \begin{bmatrix}
		yz \cdot exp(-a ((x+\bigtriangleup_s/2)^2+y^2 + z^2))\\
		-2zx \cdot exp(-a (x^2 + (y+\bigtriangleup_s/2)^2 + z^2)) \\
		xy \cdot exp(-a (x^2+y^2+(z+\bigtriangleup_s/2)^2))
	\end{bmatrix} .
\end{equation}
To isolate the effect of the staggered initialization, we also consider the corresponding non-staggered initial field:
\begin{equation}
	\label{gsini.nonstag}
	F_e(x,y,z) = exp(-a (x^2+y^2+z^2))\begin{bmatrix}
		yz \\
		-2zx  \\
		xy 
	\end{bmatrix} .
\end{equation}

Figures \ref{fig:exa.1.ini.stag} and \ref{fig:exa.1.ini.nonstag} show the initial values on staggered and non-staggered grids, respectively. 

\begin{figure}[H]
	\begin{subfigure}{.5\textwidth}
		\includegraphics[width=.9\linewidth]{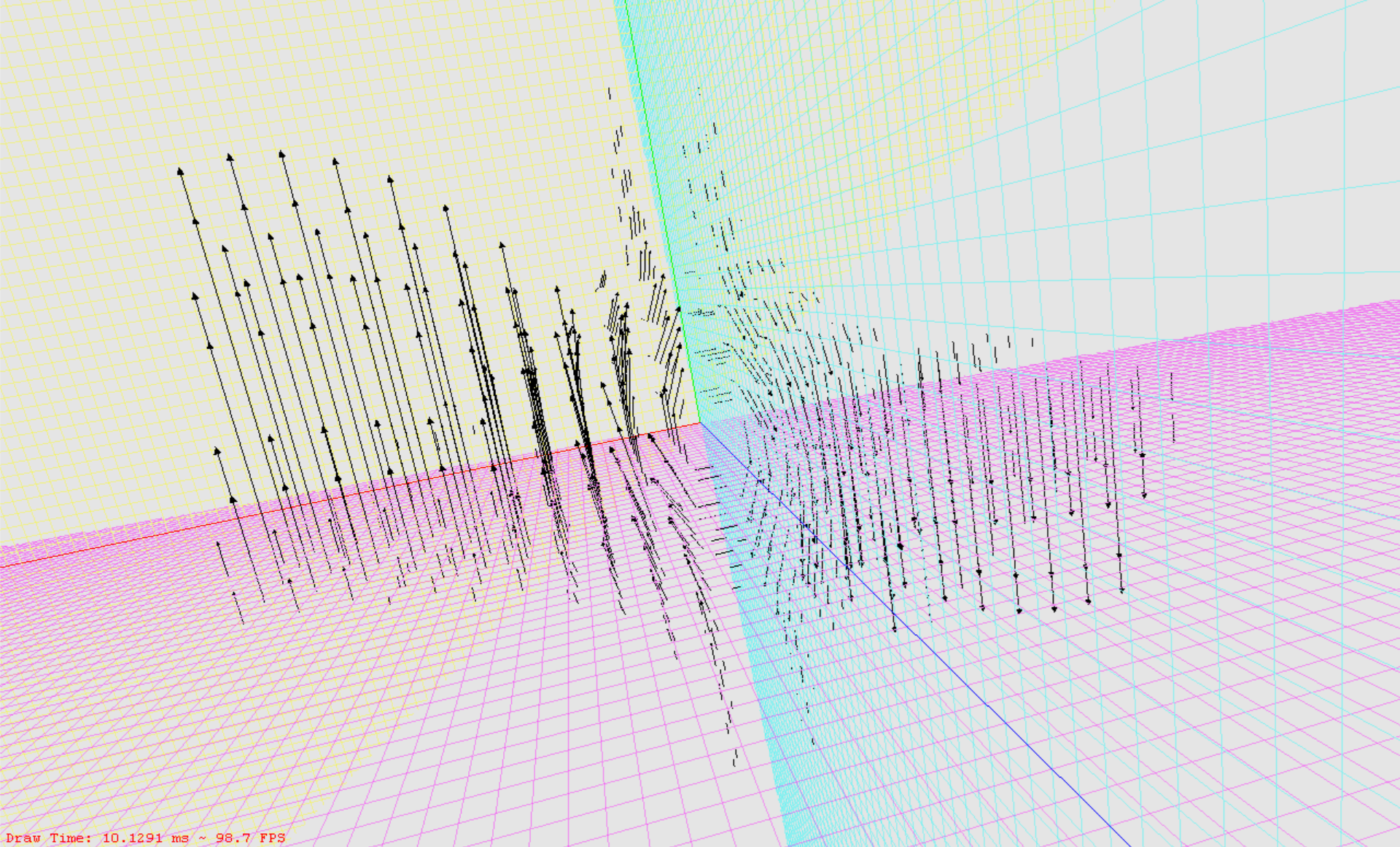}
		\caption{Initial values on Yee grid}
		\label{fig:exa.1.ini.stag}
	\end{subfigure}%
	\begin{subfigure}{.5\textwidth}
		\includegraphics[width=.9\linewidth]{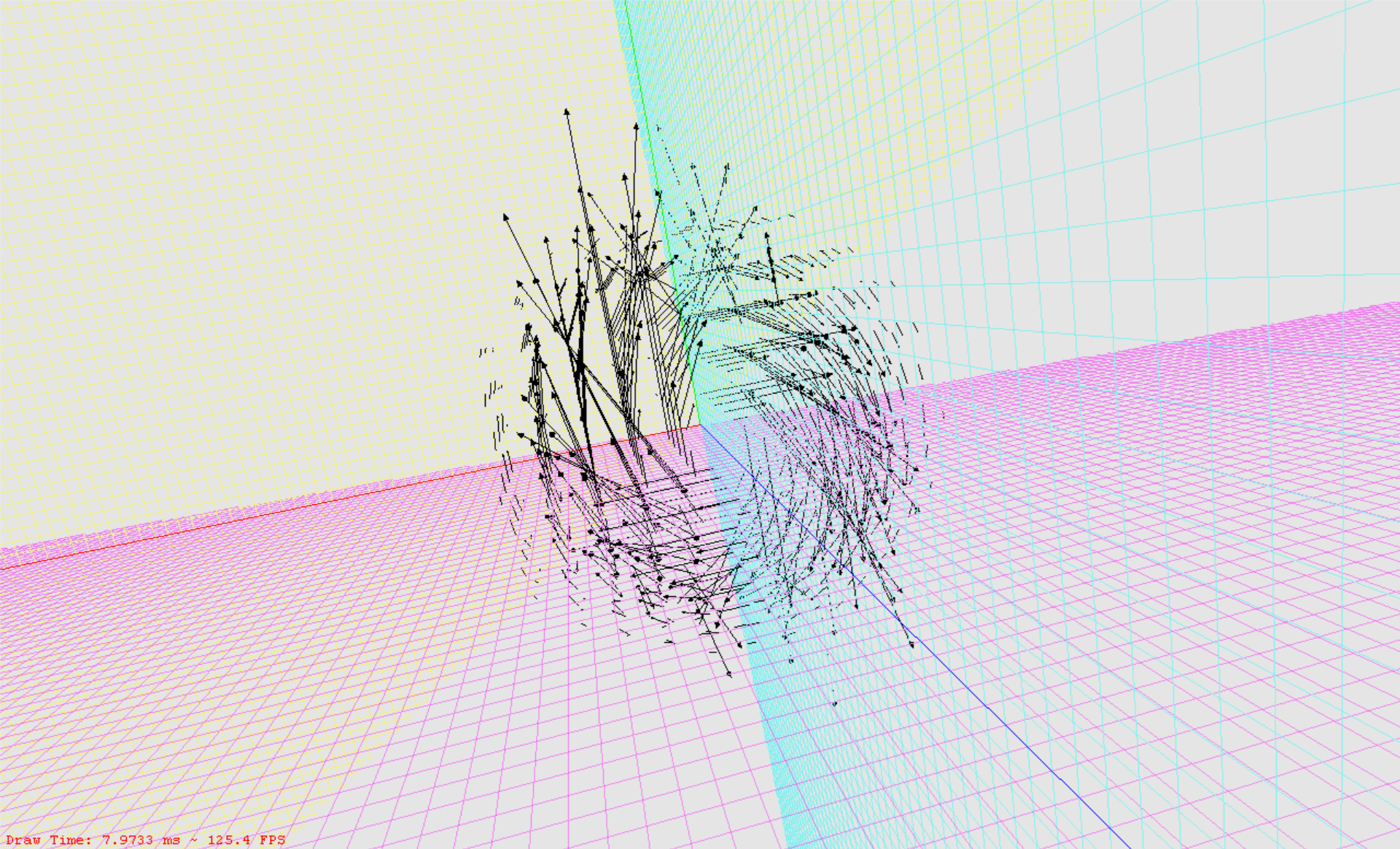}
		\caption{Initial values on non-staggered grid}
		\label{fig:exa.1.ini.nonstag}
	\end{subfigure}
	\caption{Initial values on staggered and non-staggered grid (case study 1).}
	\label{fig:exa.1.ini}
\end{figure}

To isolate the influence of the staggered initialization, we also performed an FDTD simulation using the non-staggered initial values \eqref{gsini.nonstag}. The resulting three-dimensional field distributions are presented in Figures \ref{fig:exa.2a.fdtd.E.nonstag} and \ref{fig:exa.2a.fdtd.H.nonstag}, showing good agreement with the analytical results.

\begin{figure}[H]
	\begin{subfigure}{.33\textwidth}
		\includegraphics[width=.82\linewidth]{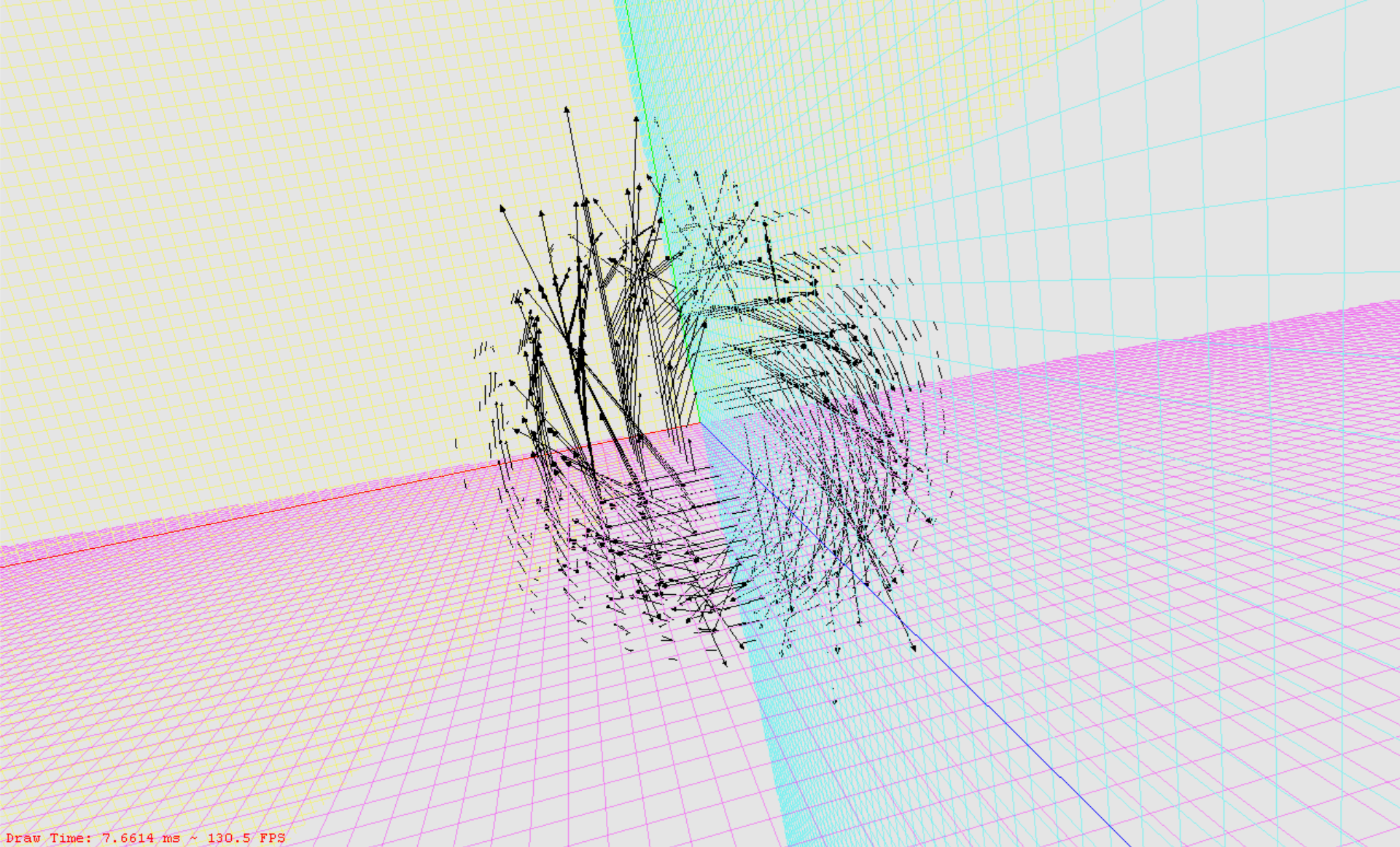}
		\caption{$E$ at $\theta=0.23$}
	\end{subfigure}%
	\begin{subfigure}{.33\textwidth}
		\includegraphics[width=.82\linewidth]{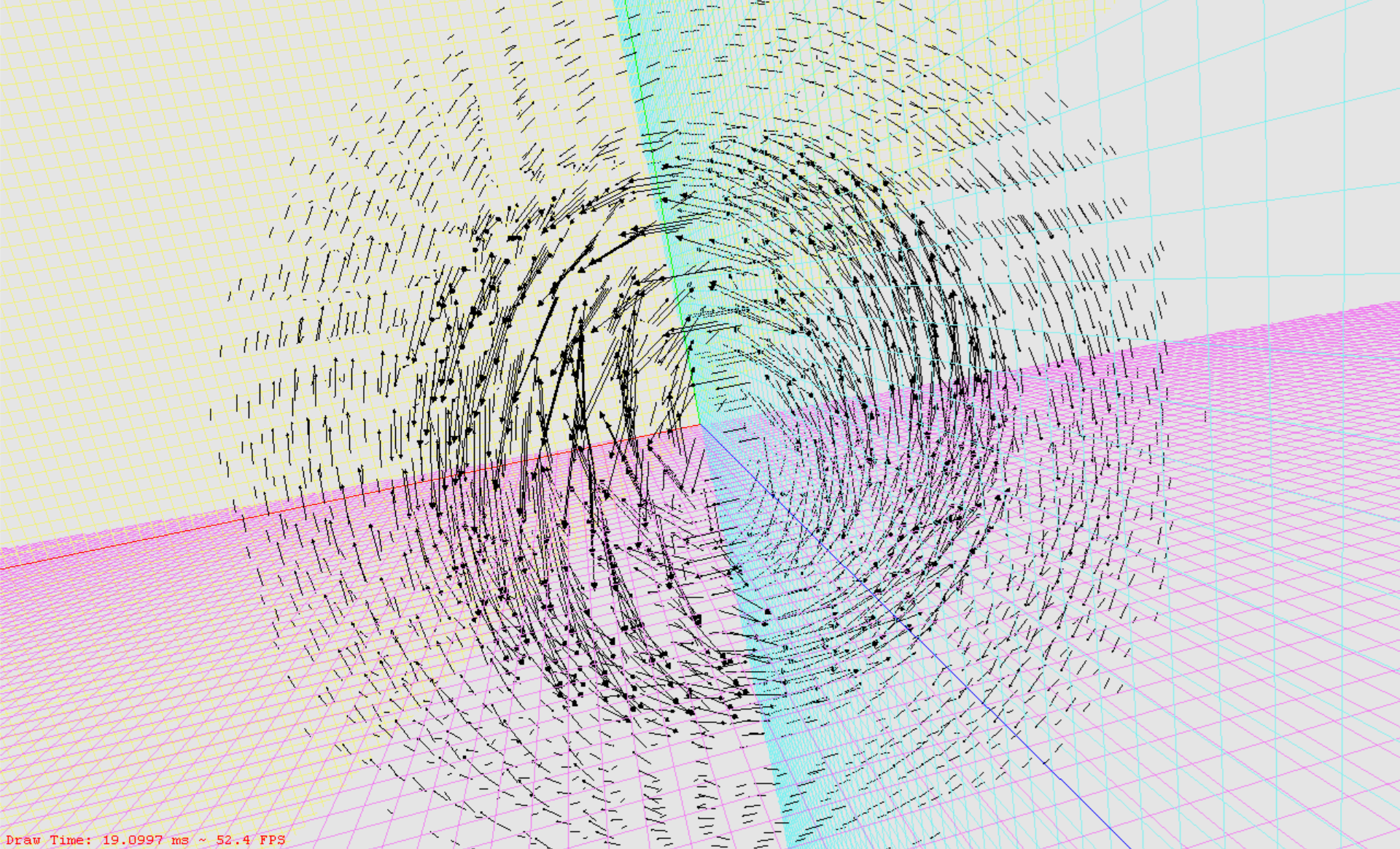}
		\caption{$E$ at $\theta=2.3$}
	\end{subfigure}%
	\begin{subfigure}{.33\textwidth}
		\includegraphics[width=.82\linewidth]{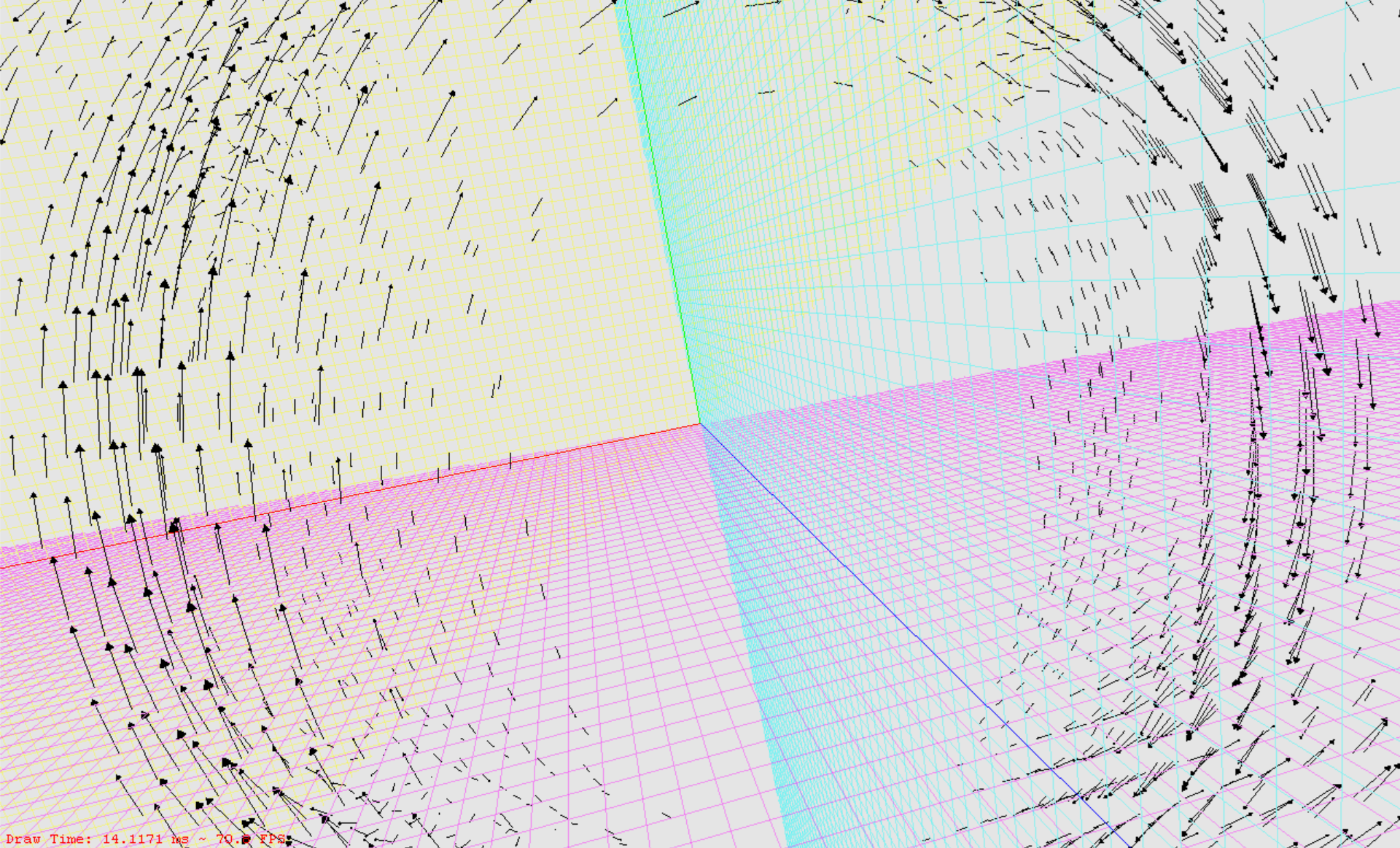}
		\caption{$E$ at $\theta=5.75$}
	\end{subfigure}
	\caption{FDTD:Electric field of non-staggered grid. $q=1,10,25$ (Case-study 1).}
	\label{fig:exa.2a.fdtd.E.nonstag}
\end{figure}	
\begin{figure}[H]
	\begin{subfigure}{.33\textwidth}
		\includegraphics[width=.82\linewidth]{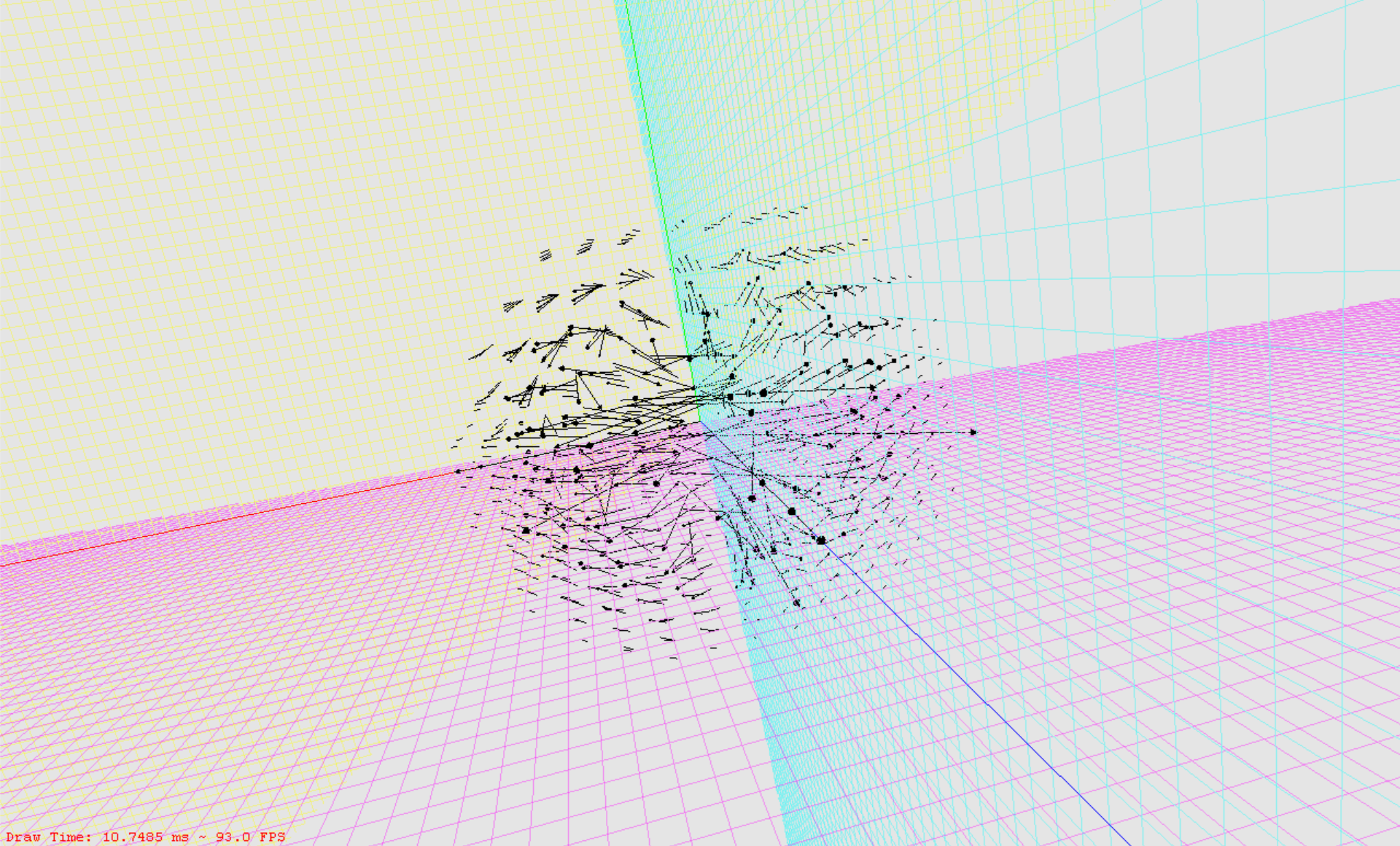}
		\caption{$H$ at $\theta=0.23$}
	\end{subfigure}%
	\begin{subfigure}{.33\textwidth}
		\includegraphics[width=.82\linewidth]{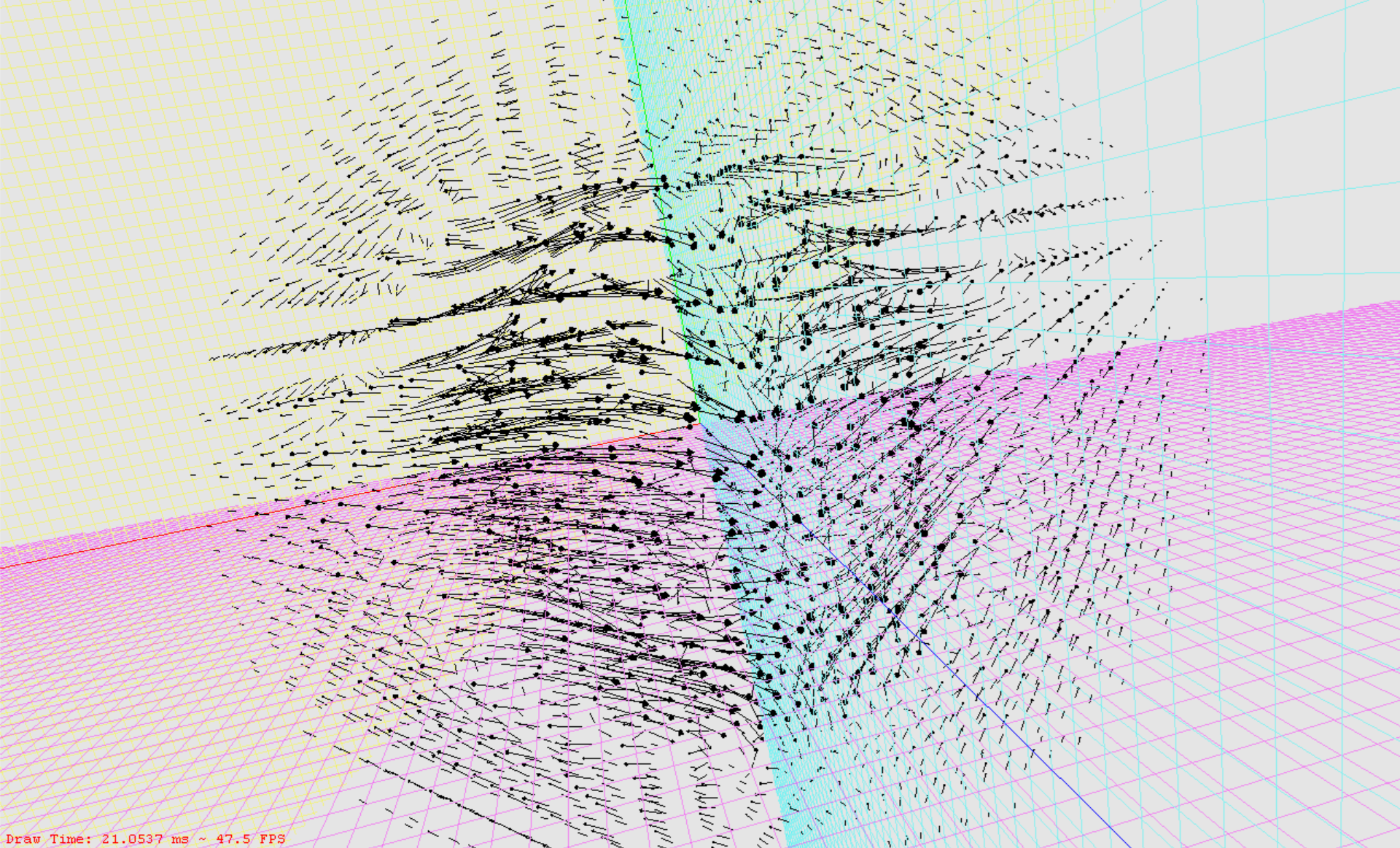}
		\caption{$H$ at $\theta=2.3$}
	\end{subfigure}%
	\begin{subfigure}{.33\textwidth}
		\includegraphics[width=.82\linewidth]{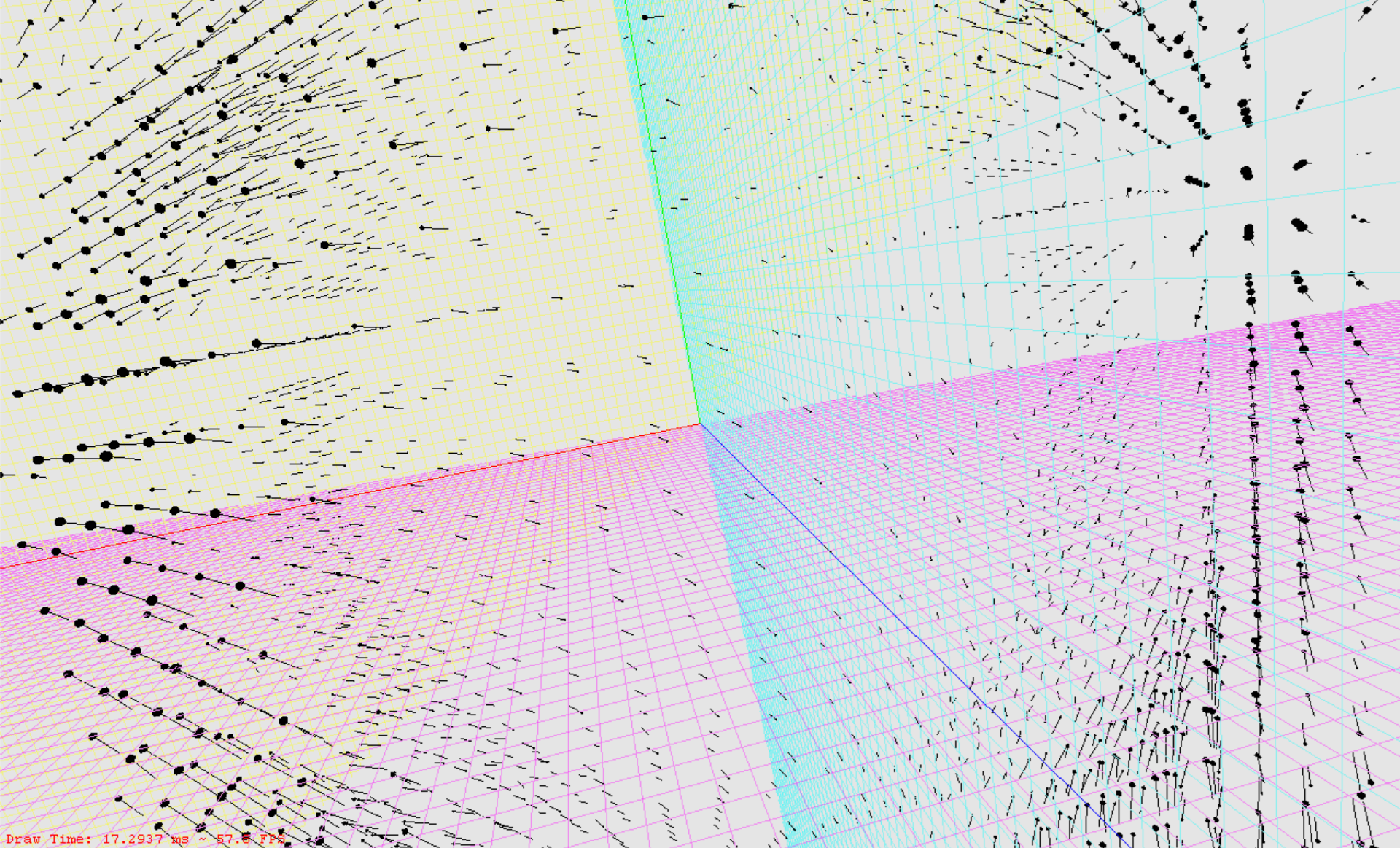}
		\caption{$H$ at $\theta=5.75$}
	\end{subfigure}
	\caption{FDTD:Magnetic field of non-staggered grid. $q=1,10,25$ (Case-study 1).}
	\label{fig:exa.2a.fdtd.H.nonstag}
\end{figure}

\subsubsection{3D data patterns: Case 2}
\begin{figure}[H]
	\begin{subfigure}{.33\textwidth}
		\includegraphics[width=.82\linewidth]{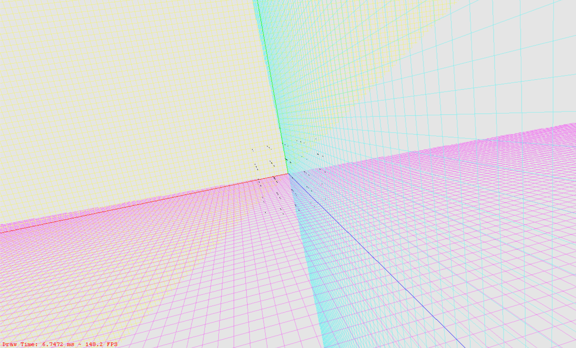}
		\caption{$E$ at $\theta=0.23$}
	\end{subfigure}%
	\begin{subfigure}{.33\textwidth}
		\includegraphics[width=.82\linewidth]{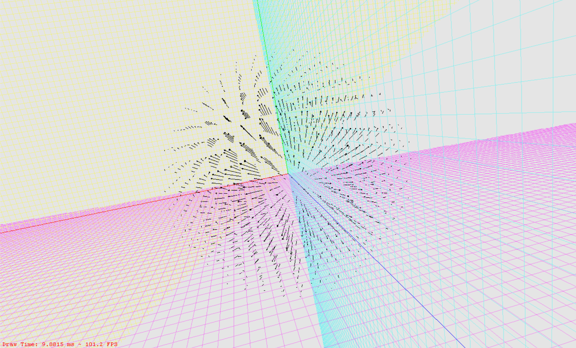}
		\caption{$E$ at $\theta=2.3$}
	\end{subfigure}%
	\begin{subfigure}{.33\textwidth}
		\includegraphics[width=.82\linewidth]{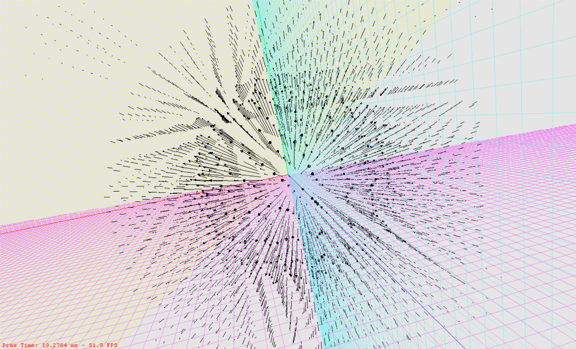}
		\caption{$E$ at $\theta=5.75$}
	\end{subfigure}
	\caption{Electric field by analytical solution. $q=1,10,25$ (Case-study 2).}
	\label{fig:exa.2a.E}
\end{figure}	
\begin{figure}[H]
	\begin{subfigure}{.33\textwidth}
		\includegraphics[width=.82\linewidth]{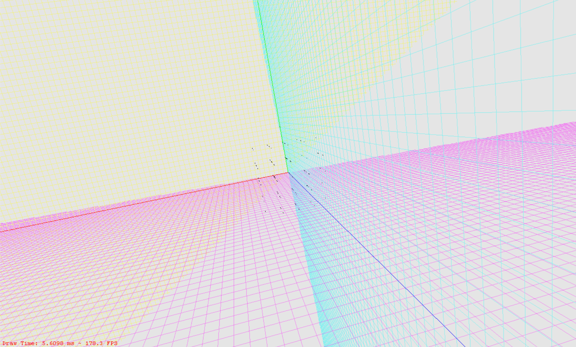}
		\caption{$E$ at $\theta=0.23$}
	\end{subfigure}%
	\begin{subfigure}{.33\textwidth}
		\includegraphics[width=.82\linewidth]{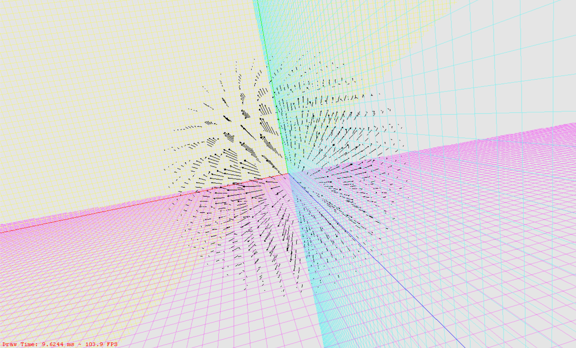}
		\caption{$E$ at $\theta=2.3$}
	\end{subfigure}%
	\begin{subfigure}{.33\textwidth}
		\includegraphics[width=.82\linewidth]{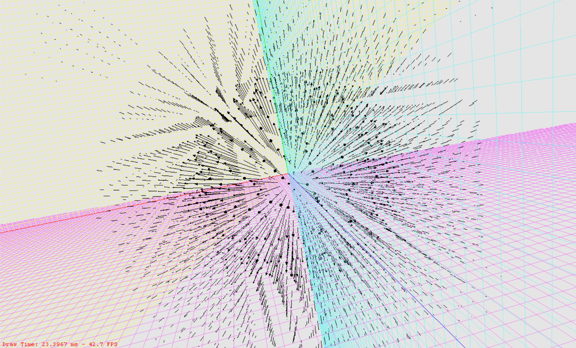}
		\caption{$E$ at $\theta=5.75$}
	\end{subfigure}
	\caption{Electric field by FDTD. $q=1,10,25$ (Case-study 2).}
	\label{fig:exa.2a.fdtd.E}
\end{figure}	

\begin{figure}[H]
	\begin{subfigure}{.33\textwidth}
		\includegraphics[width=.82\linewidth]{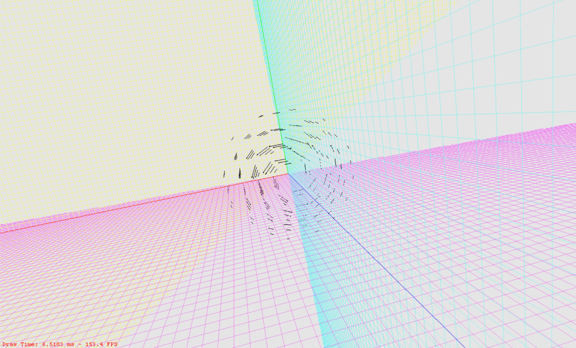}
		\caption{$H$ at $\theta=0.23$}
	\end{subfigure}%
	\begin{subfigure}{.33\textwidth}
		\includegraphics[width=.82\linewidth]{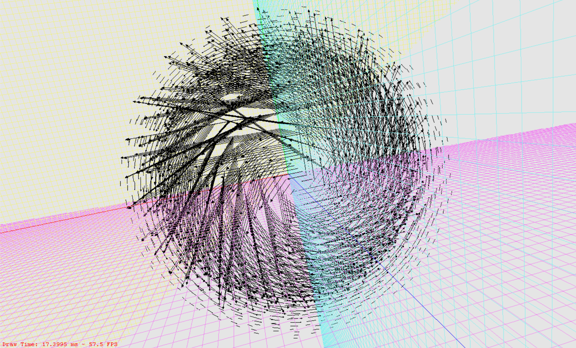}
		\caption{$H$ at $\theta=2.3$}
	\end{subfigure}%
	\begin{subfigure}{.33\textwidth}
		\includegraphics[width=.82\linewidth]{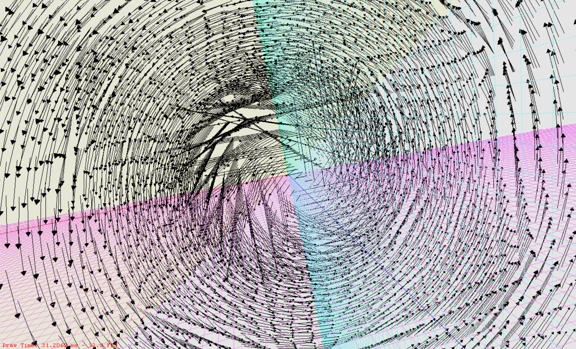}
		\caption{$H$ at $\theta=5.75$}
	\end{subfigure}
	\caption{Magnetic field by analytical solution. $q=1,10,25$ (Case-study 2).}
	\label{fig:exa.2a.H}
\end{figure}	
\begin{figure}[H]
	\begin{subfigure}{.33\textwidth}
		\includegraphics[width=.82\linewidth]{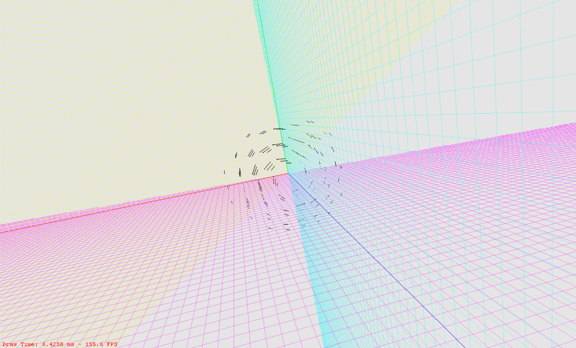}
		\caption{$H$ at $\theta=0.23$}
	\end{subfigure}%
	\begin{subfigure}{.33\textwidth}
		\includegraphics[width=.82\linewidth]{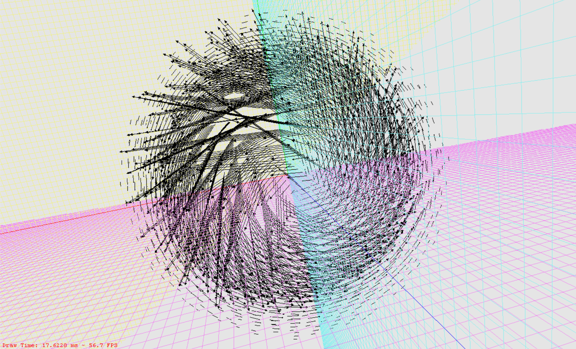}
		\caption{$H$ at $\theta=2.3$}
	\end{subfigure}%
	\begin{subfigure}{.33\textwidth}
		\includegraphics[width=.82\linewidth]{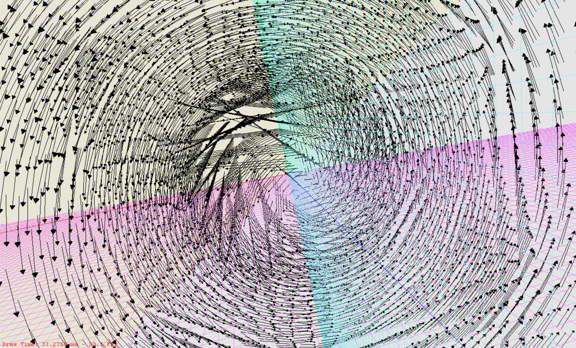}
		\caption{$H$ at $\theta=5.75$}
	\end{subfigure}
	\caption{Magnetic field by FDTD. $q=1,10,25$ (Case-study 2).}
	\label{fig:exa.2a.fdtd.H}
\end{figure}	
Figure \ref{fig:exa.2a.E} shows three snapshots of the electric field obtained from the analytical solution, while Figure \ref{fig:exa.2a.H} presents the corresponding magnetic field snapshots. Figures \ref{fig:exa.2a.fdtd.E} and \ref{fig:exa.2a.fdtd.H} show the corresponding FDTD simulation results.

These three-dimensional visualizations exhibit close agreement between the analytical and FDTD fields, with the principal spatial features reproduced consistently by both approaches.

\subsubsection{3D data patterns: Case 3}
\begin{figure}[H]
	\begin{subfigure}{.33\textwidth}
		\includegraphics[width=.82\linewidth]{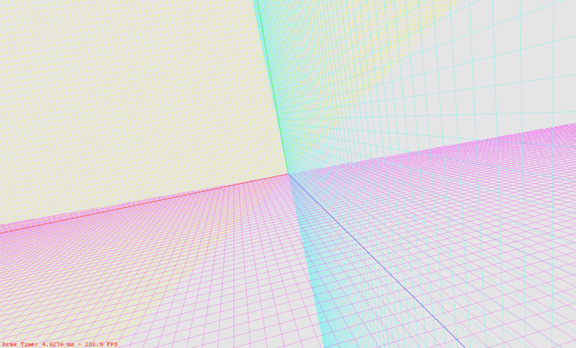}
		\caption{$E$ at $\theta=0.23$}
	\end{subfigure}%
	\begin{subfigure}{.33\textwidth}
		\includegraphics[width=.82\linewidth]{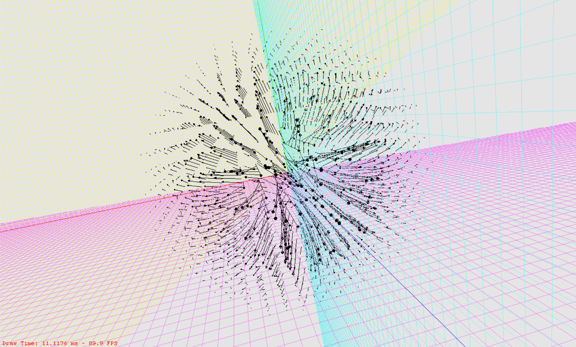}
		\caption{$E$ at $\theta=2.3$}
	\end{subfigure}%
	\begin{subfigure}{.33\textwidth}
		\includegraphics[width=.82\linewidth]{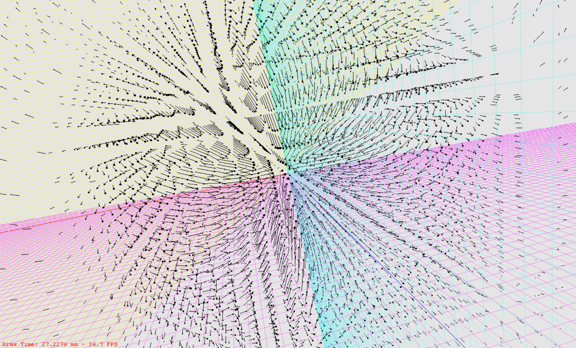}
		\caption{$E$ at $\theta=5.75$}
	\end{subfigure}
	\caption{Electric field by analytical solution. $q=1,10,25$ (Case-study 3).}
	\label{fig:exa.3a.E}
\end{figure}
\begin{figure}[H]
	\begin{subfigure}{.33\textwidth}
		\includegraphics[width=.82\linewidth]{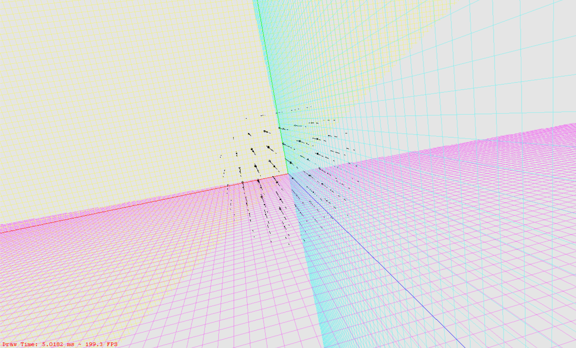}
		\caption{$E$ at $\theta=0.23$}
	\end{subfigure}%
	\begin{subfigure}{.33\textwidth}
		\includegraphics[width=.82\linewidth]{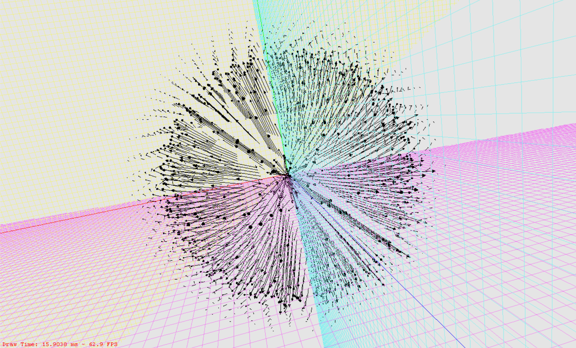}
		\caption{$E$ at $\theta=2.3$}
	\end{subfigure}%
	\begin{subfigure}{.33\textwidth}
		\includegraphics[width=.82\linewidth]{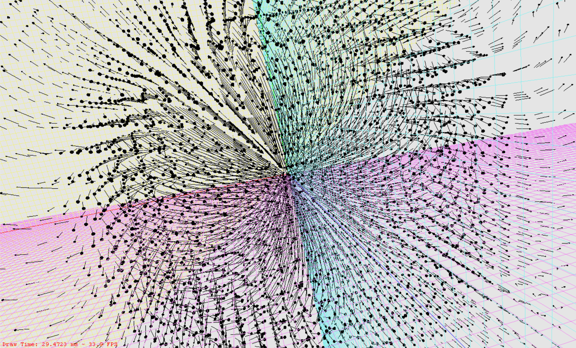}
		\caption{$E$ at $\theta=5.75$}
	\end{subfigure}
	\caption{Electric field by FDTD. $q=1,10,25$ (Case-study 3).}
	\label{fig:exa.3a.fdtd.E}
\end{figure}	

\begin{figure}[H]
	\begin{subfigure}{.33\textwidth}
		\includegraphics[width=.82\linewidth]{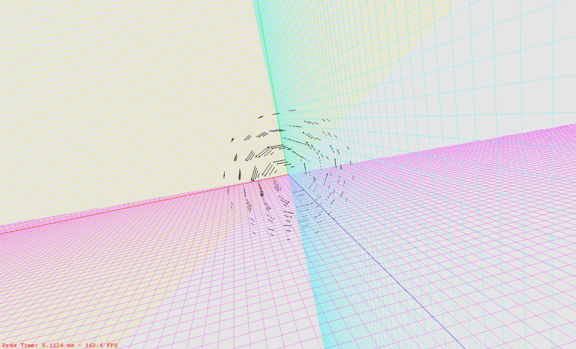}
		\caption{$H$ at $\theta=0.23$}
	\end{subfigure}%
	\begin{subfigure}{.33\textwidth}
		\includegraphics[width=.82\linewidth]{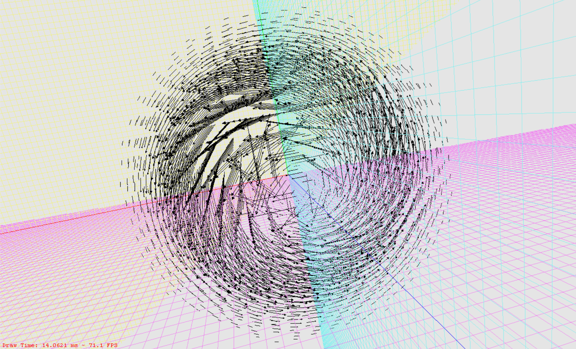}
		\caption{$H$ at $\theta=2.3$}
	\end{subfigure}%
	\begin{subfigure}{.33\textwidth}
		\includegraphics[width=.82\linewidth]{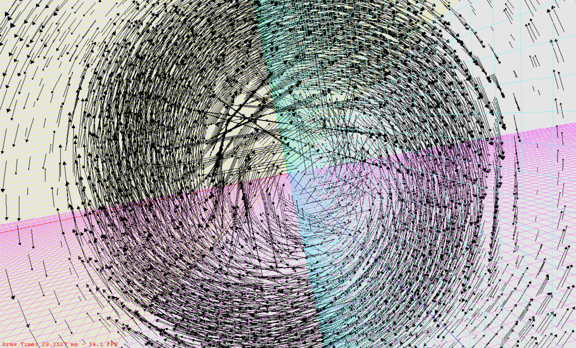}
		\caption{$H$ at $\theta=5.75$}
	\end{subfigure}
	\caption{Magnetic field by analytical solution. $q=1,10,25$ (Case-study 3).}
	\label{fig:exa.3a.H}
\end{figure}	
\begin{figure}[H]
	\begin{subfigure}{.33\textwidth}
		\includegraphics[width=.82\linewidth]{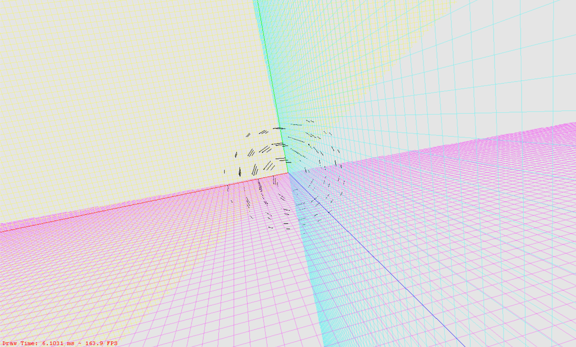}
		\caption{$H$ at $\theta=0.23$}
	\end{subfigure}%
	\begin{subfigure}{.33\textwidth}
		\includegraphics[width=.82\linewidth]{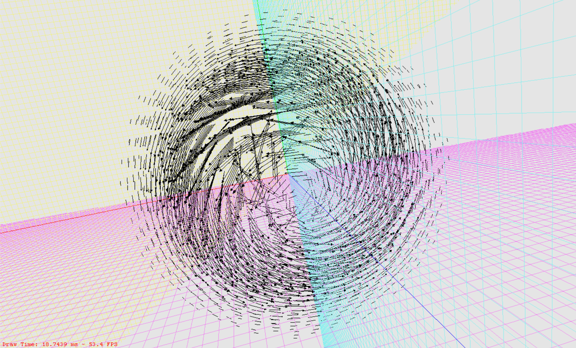}
		\caption{$H$ at $\theta=2.3$}
	\end{subfigure}%
	\begin{subfigure}{.33\textwidth}
		\includegraphics[width=.82\linewidth]{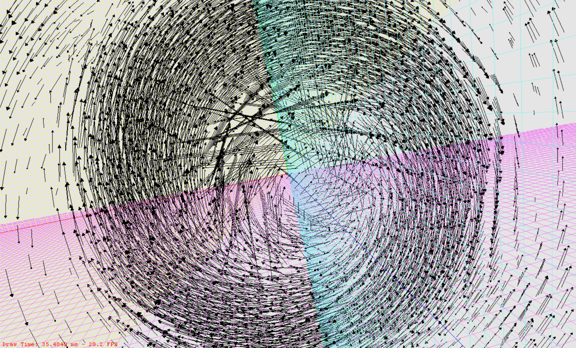}
		\caption{$H$ at $\theta=5.75$}
	\end{subfigure}
	\caption{Magnetic field by FDTD. $q=1,10,25$ (Case-study 3).}
	\label{fig:exa.3a.fdtd.H}	
\end{figure}	

Figure \ref{fig:exa.3a.E} shows three snapshots of the electric field obtained from the analytical solution, while Figure \ref{fig:exa.3a.H} presents the corresponding magnetic field snapshots. Figures \ref{fig:exa.3a.fdtd.E} and \ref{fig:exa.3a.fdtd.H} show the corresponding FDTD simulation results.

These three-dimensional visualizations again exhibit close agreement between the analytical and FDTD fields, with the overall field structures appearing nearly identical.

\subsubsection{3D data patterns: Case 4}
\begin{figure}[H]
	\begin{subfigure}{.33\textwidth}
		\includegraphics[width=.82\linewidth]{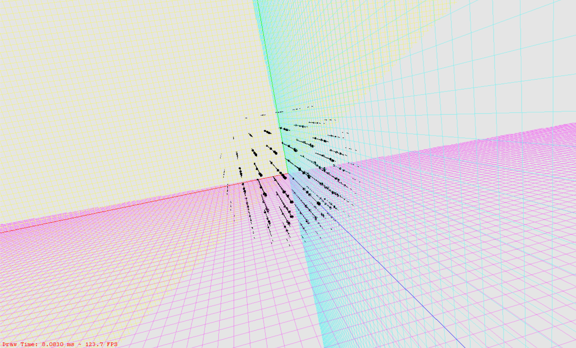}
		\caption{$E$ at $\theta=0.23$}
	\end{subfigure}%
	\begin{subfigure}{.33\textwidth}
		\includegraphics[width=.82\linewidth]{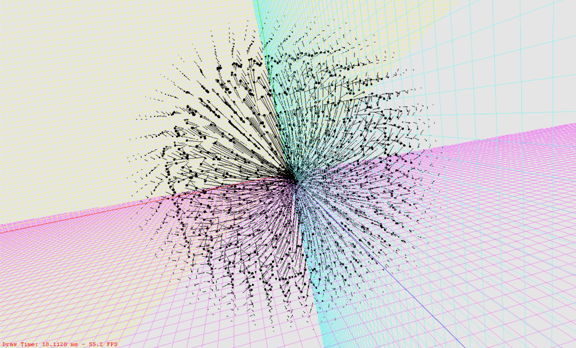}
		\caption{$E$ at $\theta=2.3$}
	\end{subfigure}%
	\begin{subfigure}{.33\textwidth}
		\includegraphics[width=.82\linewidth]{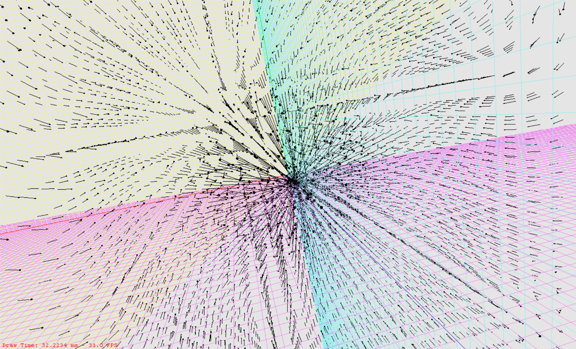}
		\caption{$E$ at $\theta=5.75$}
	\end{subfigure}
	\caption{Electric field by analytical solution. $q=1,10,25$ (Case-study 4).}
	\label{fig:exa.4a.E}
\end{figure}	
\begin{figure}[H]
	\begin{subfigure}{.33\textwidth}
		\includegraphics[width=.82\linewidth]{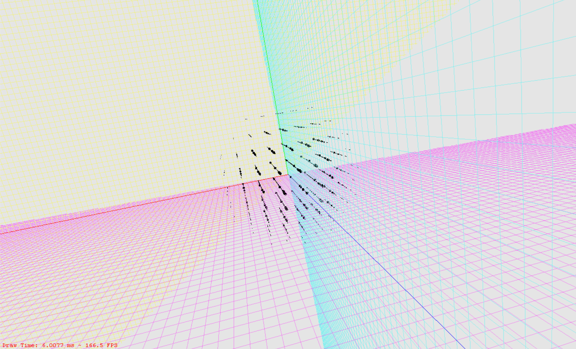}
		\caption{$E$ at $\theta=0.23$}
	\end{subfigure}%
	\begin{subfigure}{.33\textwidth}
		\includegraphics[width=.82\linewidth]{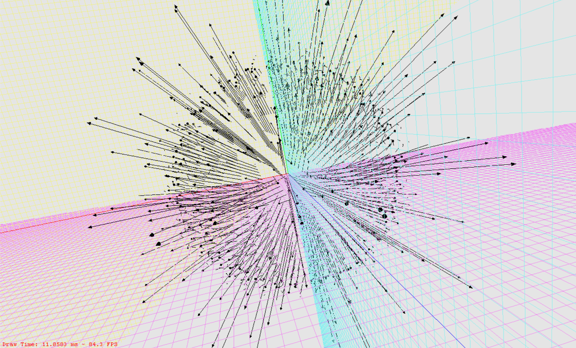}
		\caption{$E$ at $\theta=2.3$}
	\end{subfigure}%
	\begin{subfigure}{.33\textwidth}
		\includegraphics[width=.82\linewidth]{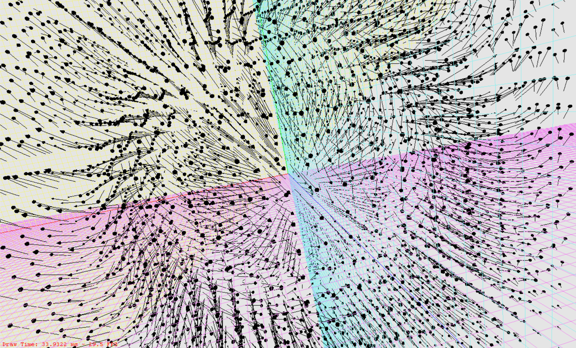}
		\caption{$E$ at $\theta=5.75$}
	\end{subfigure}
	\caption{Electric field by FDTD. $q=1,10,25$ (Case-study 4).}
	\label{fig:exa.4a.fdtd.E}
\end{figure}

\begin{figure}[H]
	\begin{subfigure}{.33\textwidth}
		\includegraphics[width=.82\linewidth]{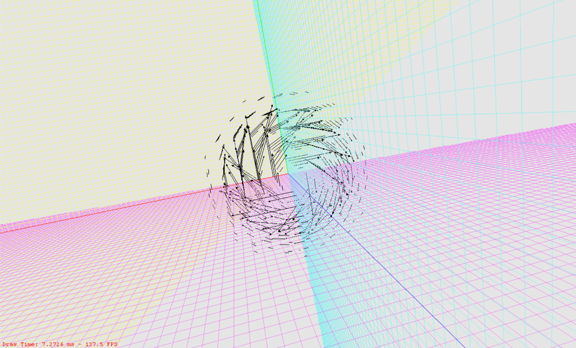}
		\caption{$H$ at $\theta=0.23$}
	\end{subfigure}%
	\begin{subfigure}{.33\textwidth}
		\includegraphics[width=.82\linewidth]{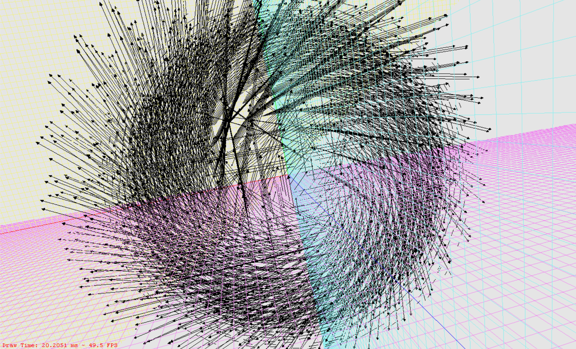}
		\caption{$H$ at $\theta=2.3$}
	\end{subfigure}%
	\begin{subfigure}{.33\textwidth}
		\includegraphics[width=.82\linewidth]{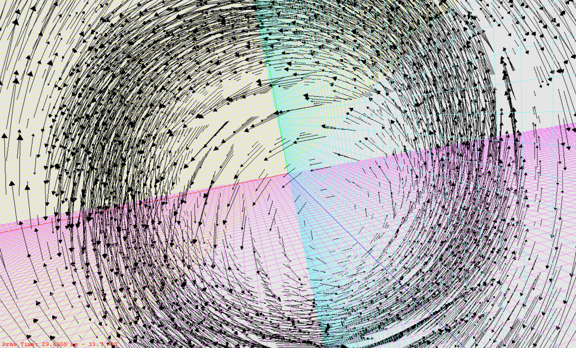}
		\caption{$H$ at $\theta=5.75$}
	\end{subfigure}
	\caption{Magnetic field by analytical solution. $q=1,10,25$ (Case-study 4).}
	\label{fig:exa.4a.H}
\end{figure}
\begin{figure}[H]
	\begin{subfigure}{.33\textwidth}
		\includegraphics[width=.82\linewidth]{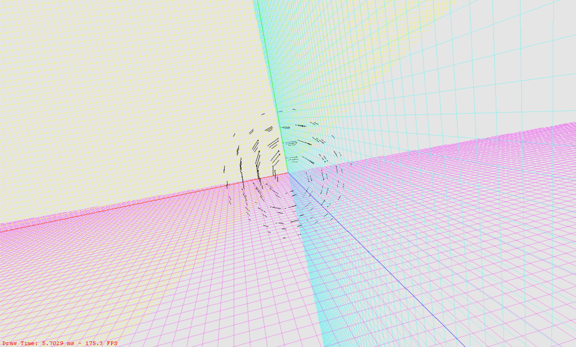}
		\caption{$H$ at $\theta=0.23$}
	\end{subfigure}%
	\begin{subfigure}{.33\textwidth}
		\includegraphics[width=.82\linewidth]{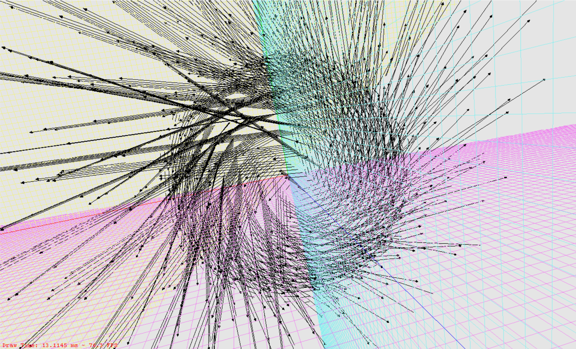}
		\caption{$H$ at $\theta=2.3$}
	\end{subfigure}%
	\begin{subfigure}{.33\textwidth}
		\includegraphics[width=.82\linewidth]{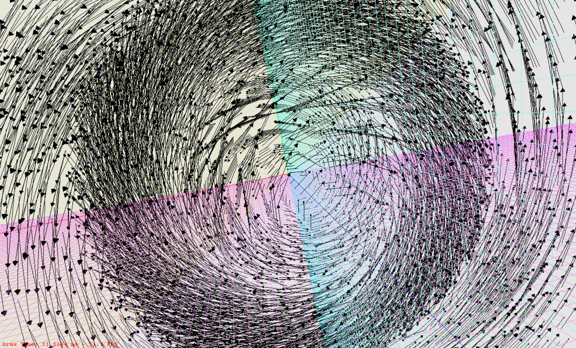}
		\caption{$H$ at $\theta=5.75$}
	\end{subfigure}
	\caption{Magnetic field by FDTD. $q=1,10,25$ (Case-study 4).}
	\label{fig:exa.4a.fdtd.H}
\end{figure}
Figure \ref{fig:exa.4a.E} shows three snapshots of the electric field obtained from the analytical solution, while Figure \ref{fig:exa.4a.H} presents the corresponding magnetic field snapshots. Figures \ref{fig:exa.4a.fdtd.E} and \ref{fig:exa.4a.fdtd.H} show the corresponding FDTD simulation results.

These three-dimensional visualizations also exhibit close agreement between the analytical and FDTD fields, showing that the principal three-dimensional field patterns are reproduced consistently by the FDTD simulation.

\subsection{Effects of FDTD mesh refinements}
\label{sect.fdtd.mesh}
Mesh size is a key factor affecting the accuracy of FDTD simulations, whereas it does not influence the analytical solutions derived in this work. Consequently, the analytical solutions can serve as benchmark reference fields for evaluating the effects of mesh refinement in FDTD. The purpose of this section is to illustrate one possible application of the proposed analytical solutions rather than to provide a comprehensive investigation of FDTD algorithms, which is an important subject but beyond the scope of the present paper.

The effects of FDTD mesh refinement are examined by progressively reducing the spatial step size $\bigtriangleup_s$. The initial spatial and temporal step sizes are $\bigtriangleup_s = 0.4$ and $\bigtriangleup_\theta = 0.23 $, respectively. The simulations are then repeated using the reduced spatial step sizes defined by (\ref{fdtd.dsn}), while the corresponding temporal step sizes are adjusted to preserve the Courant number.

\begin{equation}
	\label{fdtd.dsn}
	\begin{split}
		\bigtriangleup_{s} &= 0.4 ; n=0 \\
		\bigtriangleup_{sn} &= \bigtriangleup_s /(2n) ; n=1,2,...	
	\end{split}
\end{equation}

For each mesh refinement, two statistics are calculated as error indicators: (1) the maximum difference, defined by (\ref{Ex_err_max}); and (2) the cumulative difference, defined by (\ref{Ex_err_sum}).
\begin{align}
	\label{Ex_err_max}
	err_{max}(n) &= \max\limits_q \frac{\max\limits_s|E_{xa}-E_{xn}|}{\max\limits_s|E_{xa}|} \\	
	\label{Ex_err_sum}
	err_{sum}(n) &=\underset{q} \Sigma {\frac{\underset{s}\Sigma|E_{xa}-E_{xn}|}{\max\limits_s|E_{xa}|}}
\end{align}
where $E_{xa} $ is the $E_x$ component calculated from the analytical solution, and $E_{xn} $ is the $E_x $ component obtained from the FDTD simulation using the spatial step size defined by (\ref{fdtd.dsn}). 

For each refinement level, the error metrics are evaluated using the spatial locations and time instants common to all simulations, defined by (\ref{fdtd.Ex}), so that the comparisons are performed on the same set of sampling points.

\begin{figure}[H]
	\begin{subfigure}{.5\textwidth}
		\includegraphics[width=.9\linewidth]{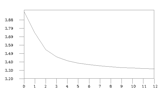}
		\caption{Max difference vs step sizes}
		\label{fig:exa.1.Ex.errMax}
	\end{subfigure}%
	\begin{subfigure}{.5\textwidth}
		\includegraphics[width=.9\linewidth]{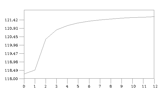}
		\caption{Sum of difference vs step sizes}
		\label{fig:exa.1.Ex.errSum}
	\end{subfigure}
	\caption{Effects of step sizes on FDTD simulations (Case-study 1).}
	\label{fig:exa.1.Ex.err}
\end{figure}
\textbf{Case 1}. Figure \ref{fig:exa.1.Ex.errMax} shows the maximum difference between $E_x$ values computed from the analytical solution and those obtained from FDTD simulations, while Figure \ref{fig:exa.1.Ex.errSum} shows the cumulative absolute difference. The variable n on the horizontal axis denotes the spatial step size defined in (\ref{fdtd.dsn}). These figures lead to the following observations:
\begin{itemize}
	\item The maximum difference decreases as the step size is reduced.
	\item The cumulative difference increases as the step size is reduced.
	\item Further reduction of the step size yields diminishing changes once the step size is sufficiently small.
\end{itemize}

\begin{figure}[H]
	\begin{subfigure}{.5\textwidth}
		\includegraphics[width=.9\linewidth]{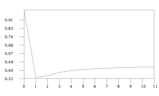}
		\caption{Max difference vs step sizes}
		\label{fig:exa.2.Ex.errMax}
	\end{subfigure}%
	\begin{subfigure}{.5\textwidth}
		\includegraphics[width=.9\linewidth]{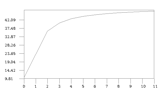}
		\caption{Sum of difference vs step sizes}
		\label{fig:exa.2.Ex.errSum}
	\end{subfigure}
	\caption{Effects of step sizes on FDTD simulations (Case-study 2).}
	\label{fig:exa.2.Ex.err}
\end{figure}
\textbf{Case 2}. Figure \ref{fig:exa.2.Ex.errMax} shows the maximum difference between $E_x$ values computed from the analytical solution and those obtained from FDTD simulations, while Figure \ref{fig:exa.2.Ex.errSum} shows the cumulative absolute difference. These figures lead to the following observations:
\begin{itemize}
	\item Reducing the step size to $\bigtriangleup_s / 2 $ leads to a decrease in the maximum difference. Further reduction of the step size results in a slight increase in the maximum difference, indicating a weak non-monotonic trend.
	\item Reducing the step size to $\bigtriangleup_s / 2 $ increases the cumulative difference. Further reduction of the step size produces only minor changes.
\end{itemize}
A similar weak non-monotonic behavior in the maximum difference is observed in a subset of the case studies, while the overall trend remains relatively flat.

\begin{figure}[H]
	\begin{subfigure}{.5\textwidth}
		\includegraphics[width=.9\linewidth]{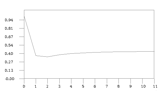}
		\caption{Max difference vs step sizes}
		\label{fig:exa.3.Ex.errMax}
	\end{subfigure}%
	\begin{subfigure}{.5\textwidth}
		\includegraphics[width=.9\linewidth]{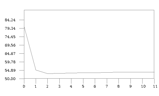}
		\caption{Sum of difference vs step sizes}
		\label{fig:exa.3.Ex.errSum}
	\end{subfigure}
	\caption{Effects of step sizes on FDTD simulations (Case-study 3).}
	\label{fig:exa.3.Ex.err}
\end{figure}
\textbf{Case 3}. Figure \ref{fig:exa.3.Ex.errMax} shows the maximum difference between $E_x$ values computed from the analytical solution and those obtained from FDTD simulations, while Figure \ref{fig:exa.3.Ex.errSum} shows the cumulative difference. These figures lead to the following observations:
\begin{itemize}
	\item Reducing the step size to $\bigtriangleup_s /2$ leads to a decrease in the maximum difference. Further reduction of the step size results in negligible changes, with only a slight increase in the value.
	\item Reducing the step size to $\bigtriangleup_s /2$ also decreases the cumulative difference. Further reduction of the step size produces negligible changes, and the values remain essentially constant.
\end{itemize}
These trends indicate that further reduction of the step size yields limited improvement in FDTD accuracy once the step size becomes sufficiently small.

\begin{figure}[H]
	\begin{subfigure}{.5\textwidth}
		\includegraphics[width=.9\linewidth]{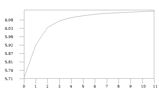}
		\caption{Max difference vs step sizes}
		\label{fig:exa.4.Ex.errMax}
	\end{subfigure}%
	\begin{subfigure}{.5\textwidth}
		\includegraphics[width=.9\linewidth]{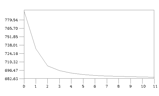}
		\caption{Sum of difference vs step sizes}
		\label{fig:exa.4.Ex.errSum}
	\end{subfigure}
	\caption{Effects of step sizes on FDTD simulations (Case-study 4).}
	\label{fig:exa.4.Ex.err}
\end{figure}

\textbf{Case 4}. Figure \ref{fig:exa.4.Ex.errMax} shows the maximum difference between $E_x$ values computed from the analytical solution and those obtained from FDTD simulations, while Figure \ref{fig:exa.4.Ex.errSum} shows the cumulative difference. These figures exhibit the same trend observed in the other cases: the FDTD results become largely insensitive to further step size reduction once the step size becomes sufficiently small.

\textbf{Summary}
For all cases considered, the error measures approach limiting values as the mesh is refined. However, the convergence behavior depends on the definition of the error indicator. Depending on the selected measure and the sampling procedure, the error may decrease monotonically, increase monotonically, or vary non-monotonically before approaching its limiting value.

The analytical reference fields therefore reveal that, for the problems considered here, reducing the mesh size alone does not necessarily produce a monotonic reduction in the observed errors. The observed convergence behavior depends on factors such as the definition of the error indicator, the spatial sampling locations, and the number of sampled points.

Since error measures converge to non-zero limiting values, further improvements in numerical accuracy may require modifications to the numerical approximation itself rather than continued mesh refinement alone. The analytical reference fields developed in this work provide a natural benchmark for such investigations. Our preliminary numerical experiments suggest that, when these reference fields are used to evaluate higher-order FDTD schemes, the numerical errors may decrease monotonically with increasing approximation order. A detailed investigation of higher-order FDTD methods is beyond the scope of the present paper.

\section{Discussion}

The present work constructs a generic solution to Maxwell’s equations using an infinite-order curl-operator expansion. This approach yields what we term an analytical solution, in contrast to conventional numerical schemes. In the case studies, the differential operators appearing in the general expressions are explicitly evaluated, leading to formulas that we refer to as closed-form analytical solutions, which are formed by Taylor style series of algebraic expressions.

These closed-form solutions reveal that the electromagnetic fields travel in deformed sine and cosine functions; the structures of deformations depend on the initial conditions and source terms. To illustrate the nature of this deformation, consider the classical traveling-wave sine function, which admits the series representation
\begin{equation*}
	\sin(x-ct) = \sum_{n=0}^{\infty} \frac{(-1)^n (x-ct)^{2n+1}}{(2n+1)!}
\end{equation*}
Within the present framework, a representative deformed sine function takes the form
\begin{equation*}
	\sum_{n=0}^{\infty}\frac{(-1)^n (ct)^{2(n+s)+1}}{(2(n+s)+1)!}\sum_{k=0}^{n}p_{2,n,k}(2x)^{2k}
\end{equation*}
where $p_{2,n,k}$ are coefficients determined by the operator expansion and $s$ is a parameter arising from the solution construction. In contrast to the classical case, the spatial and temporal variables are no longer coupled through a simple phase term $(x-ct)$; instead, the solution exhibits a structured mixing of spatial polynomials and temporally weighted series terms. Many other deformations arise naturally within the same analytical framework. This structural shift suggests that the classical notion of phase velocity may require reinterpretation in this mathematical context. However, this is only a mathematical observation, not a physical interpretation or prediction.

The deformed trigonometric functions describe the instantaneous behavior of the electromagnetic fields. Understanding their implications for steady-state or long-time behavior requires analysis of their asymptotic properties as spatial and temporal variables approach infinity. In particular, it remains an open question whether the deformation diminishes in these limits, potentially recovering standard trigonometric behavior.

Notably, the deformed sine and cosine functions are not limited to specific case studies but are intrinsically embedded in the general solution formulas (\ref{3d.solIniH}), (\ref{3d.solIniE}), (\ref{3d.solSrcH}) and (\ref{3d.solSrcE}). This suggests that, within the present analytical framework, the evolution of electromagnetic fields can be interpreted as being governed by these deformed trigonometric structures, whose forms are determined by the initial conditions and source terms. This observation reveals a unifying structural feature of the analytical solutions and provides a new perspective on how electromagnetic fields propagate under general conditions.

Unlike classical periodic functions, whose values can be readily evaluated for arbitrarily large arguments, the deformed sine and cosine functions are not strictly periodic. This lack of exact periodicity introduces additional challenges in evaluating the solutions for large spatial or temporal domains. Developing efficient analytical or numerical methods for their evaluation in such regimes represents a significant direction for future research in applied mathematics.

The Time-Space theorem presented in \ref{thm.ts} can also be applied to higher-order FDTD simulations. The author has implemented such schemes using both the Yee grid and a non-staggered grid, and in both cases the expected improvements in accuracy were observed.

\section{Conclusion}

This work presents an operator-based framework for constructing analytical solutions to Maxwell’s equations in open-space settings with general initial conditions and source terms. The resulting infinite-order curl-operator expansion provides explicit function-to-function mappings and leads to closed-form representations in terms of deformed sine and cosine functions. Comparisons with FDTD simulations demonstrate both agreement in field structure and systematic differences attributable to discretization effects and estimation errors. The observed insensitivity of FDTD results to further step size reduction highlights inherent limitations of numerical refinement. The analytical solutions offer new theoretical insight into electromagnetic field behavior and serve as practical benchmarks for computational electromagnetics. Future work will further investigate the properties of the deformed functions and extend the framework to broader classes of problems.
 
\appendix

\section{Time-Space theorem}
\label{thm.ts}
Maxwell’s equations establish a correspondence between first-order temporal derivatives and first-order spatial curls. The following theorem generalizes this relationship to derivatives of arbitrary order, providing a systematic extension of the time–space structure.
\begin{thm} \label{timespacetheorem} If a function-pair $H(x,y,z,\theta), E(x,y,z,\theta) \in C^{\infty}(\mathbb{R}^4, \mathbb{R}^3) $ satisfies Maxwell's equations (\ref{3d.1}) and (\ref{3d.2}) then for any $n \ge 0 $
	\begin{equation}
		\label{ts.h1}
		\pHn{2n+1} {\theta}= \frac{1}{\eta} (-1)^{n+1} \nabla^{2n+1} \times E + J_{h,2n+1}
	\end{equation} 
	\begin{equation}
		\label{ts.h2}
		\pHn{2n+2} {\theta}= (-1)^{n+1} \nabla^{2n+2} \times H + J_{h,2n+2}
	\end{equation}
	\begin{equation}
	\label{ts.e1}
	\pEn{2n+1} {\theta}= \eta (-1)^{n} \nabla^{2n+1} \times H + \eta J_{e,2n+1}
	\end{equation} 
	\begin{equation}
	\label{ts.e2}
	\pEn{2n+2} {\theta}= (-1)^{n+1} \nabla^{2n+2} \times E + \eta J_{e,2n+2}
	\end{equation} 
where the source terms are given by
\begin{equation}
	\label{ts.jh1}
	\begin{gathered}
	J_{h,1} = 0 \\
	J_{h,2n+1} = \sum_{m=1}^{n} (-1)^{m+1} \nabla^{2m-1} \times \pJn{2(n-m)+1}, n > 0
	\end{gathered}
\end{equation}
\begin{equation}
	\label{ts.jh2}
	J_{h,2n+2} = \sum_{m=0}^{n} (-1)^{m} \nabla^{2m+1} \times \pJn{2(n-m)}
\end{equation}
\begin{equation}
	\label{ts.je1}
	J_{e,2n+1} = \sum_{m=0}^{n} (-1)^{m+1} \nabla^{2m} \times \pJn{2(n-m)}
\end{equation}
\begin{equation}
	\label{ts.je2}
	J_{e,2n+2} = \sum_{m=0}^{n} (-1)^{m+1} \nabla^{2m} \times \pJn{2(n-m)+1}
\end{equation}
\end{thm}
\begin{proof}
	We will induct on n.
	
	\textbf{Base case ($n=0, n=1$)}: It is easy to validate that for $n=0$ and $n=1$ the theorem holds.
	
	\textbf{Inductive Hypothesis ($n=m$)}: Assume for some $m \in \mathbb{Z}_{\ge 0}$ the theorem holds.
	
	\textbf{Inductive Step}: 
	
	By applying the differential operation $ \partial / {\partial \theta} $ on both sides of (\ref{ts.h2}) with $n=m$, we have
\begin{equation}
	\label{ts.p1}
	\pHn{2(m+1)+1}{\theta} = \frac{1}{\eta} (-1)^{(m+1)+1} \curl{2(m+1)+1} E + \sum_{k=1}^{m+1}(-1)^{k+1} \curl{2k-1} \pJn{2(m+1-k)+1} 
\end{equation}
The above shows that (\ref{ts.h1}) holds for $n=m+1$.	
	
By applying the differential operation ${\partial} / {\partial \theta} $ on both sides of (\ref{ts.p1}) and substitute (\ref{3d.2}) into it, we have
\begin{equation*}
	\pHn{2(m+2)}{\theta} = (-1)^{m+2} \curl{2(m+2)} H + \sum_{k=0}^{m+1} (-1)^k \curl{2k+1} \pJn{2(m+1-k)}
\end{equation*}
The above shows that (\ref{ts.h2}) holds for $n=m+1 $.

By applying the differential operation ${\partial} / {\partial \theta} $ on both sides of (\ref{ts.e2}) with $n=m $ and substitute (\ref{3d.2}) into it, we have
\begin{equation}
	\label{ts.p2}
	\pEn{2(m+1)+1}{\theta} = \eta (-1)^{m+1} \curl{2(m+1)+1} H + \eta \sum_{k=0}^{m+1} (-1)^{k+1} \curl{2k} \pJn{2(m+1-k)}
\end{equation}
The above shows that (\ref{ts.e1}) holds for $n=m+1$.

By applying the differential operation ${\partial} / {\partial \theta} $ on both sides of (\ref{ts.p2}) and substitute (\ref{3d.1}) into it, we have
\begin{equation*}
	\pEn{2(m+1+1)}{\theta} = (-1)^{m+1+1} \curl{2(m+1+1)} E + \eta \sum_{k=0}^{m+1} (-1)^{k+1} \curl{2k} \pJn{2(m+1-k)+1}
\end{equation*}
The above shows that (\ref{ts.e2}) holds for $n=m+1$.

$\therefore$ By the principle of induction, (\ref{ts.h1}), (\ref{ts.h2}), (\ref{ts.e1}) and (\ref{ts.e2}) hold for all $n\in \mathbb{Z}_{\ge 0}$.
\end{proof}

\section{Formulas used by the case-studies}
\label{appendix.cases}
To apply the extended Taylor series presented in this paper, one must obtain analytical expressions for the derivatives and curls of the initial values and the source term. For the case studies presented in this paper, the required formulas are provided and proved in this appendix.

\subsection{\textbf{Identities}}
The following identities are needed to deduce solutions in the case studies.
\begin{lemma} For $n \ge 0, n \ge k \ge 0 $ we have the following identities
	\begin{align}
		\label{ph.n.n}
		p_{h,n,n} &= 1
\\
		\label{lemma.p2nk}
		p_{2,n,k} &= p_{1,n,k} + 4(k+1) p_{1,n,k+1}
\\
		\label{lemma.p3nk}
		p_{3,n,k} &= 4(k+1)p_{2,n,k+1} + p_{2,n,k}
\\
		\label{lemma.p1np1k}
		p_{1,n+1,k} &= 2(2k+5)p_{2,n,k} + p_{2,n,k-1}
	\\
		\label{lemma.unk.2}
		p_{1,n+1,k} - 4p_{2,n,k} &= 2(2k+3)p_{2,n,k} + p_{2,n,k-1}
\end{align}

\end{lemma}

\begin{proof}(\textbf{Eq.\eqref{ph.n.n}}).
	Substituting $k=n $ into \eqref{phnk} gives \eqref{ph.n.n}.
\end{proof}

\begin{proof}(\textbf{Eq.\eqref{lemma.p2nk}}).
	We start from the right hand side of \eqref{lemma.p2nk} and reach the left hand side.
	\begin{align*}
		\mathrm{RHS} &=  \frac{n!(k+1)!(2(n+1)+1)!}{(n-k)!k!(n+1)!(2(k+1)+1)!} \\
		&+ 4(k+1)\frac{n!(k+2)!(2(n+1)+1)!}{(n-k-1)!(k+1)!(n+1)!(2(k+2)+1)!} \\
		&= \frac{n!(k+1)!(2n+3)!}{(n-k)!k!(n+1)!(2k+3)!}\left(1+\frac{2(n-k)}{(2k+5)}\right) \\
		&=
		\frac{n!(k+1)!(2n+3)!}{(n-k)!k!(n+1)!(2k+3)!}\left(\frac{(2n+5)}{(2k+5)}\right) \\
		&=
		\frac{n!(k+1)!(2n+3)!}{(n-k)!k!(n+1)!(2k+3)!}\left(\frac{(2n+5)}{(2k+5)}\right)\frac{(2k+4)}{(2k+4)}\frac{(2n+4)}{(2n+4)} \\
		&=
		\frac{n!(k+2)!(2n+5)!}{(n-k)!k!(n+2)!(2k+5)!} = p_{2,n,k} = \mathrm{LHS}.
	\end{align*}
\end{proof}
\begin{proof}(\textbf{Eq.\eqref{lemma.p3nk}}).
	We start from the right hand side of \eqref{lemma.p3nk} and reach the left hand side.
	\begin{align*}
		\mathrm{RHS} &= \frac{4(k+1)n!(k+3)!(2n+5)!}{(n-k-1)!(n+2)!(k+1)!(2k+7)!} \\
		&+ \frac{n!(k+2)!(2n+5)!}{(n-k)!(n+2)!k!(2k+5)!} \\
		&= \frac{n!(k+2)!(2n+5)!}{(n-k)!(n+2)!k!(2k+5)!}\left(\frac{2(n-k)}{(2k+7)}+1\right) \\
		&= \frac{n!(k+2)!(2n+5)!}{(n-k)!(n+2)!k!(2k+5)!}\left(\frac{2n+7}{2k+7}\right) \\
		&=  \frac{n!(k+2)!(2n+5)!}{(n-k)!(n+2)!k!(2k+5)!}\left(\frac{2n+7}{2k+7}\right)\frac{(2n+6)}{(2n+6)}\frac{(2k+6)}{(2k+6)} \\		
		&= \frac{n!(k+2)!(2n+7)!}{(n-k)!(n+2)!k!(2k+7)!}\left(\frac{2k+6}{2n+6}\right) \\
		&= \frac{n!(k+3)!(2n+7)!}{(n-k)!(n+3)!k!(2k+7)!} = p_{3,n,k} = \mathrm{LHS}.
	\end{align*}
\end{proof}
\begin{proof}(\textbf{Eq.\eqref{lemma.p1np1k}}).
	We start from the right hand side of \eqref{lemma.p1np1k} and reach the left hand side.
	\begin{align*}
		\mathrm{RHS} &= \frac{2(2k+5)n!(k+2)!(2n+5)!}{(n-k)!(n+2)!k!(2k+5)!} \\
		&+ \frac{n!(k+1)!(2n+5)!}{(n+1-k)!(n+2)!(k-1)!(2k+3)!} \\
		&= \frac{n!(k+1)!(2n+5)!}{(n+2)!(2k+3)!}\left(\frac{1}{(n-k)!k!}+\frac{1}{(n+1-k)!(k-1)!}\right) \\
		&= \frac{n!(k+1)!(2n+5)!}{(n+1-k)!(n+2)!k!(2k+3)!}(n+1) \\
		&= \frac{(n+1)!(k+1)!(2n+5)!}{(n+1-k)!(n+2)!k!(2k+3)!} = p_{1,n+1,k} = \mathrm{LHS}.
	\end{align*}
\end{proof}
\begin{proof}(\textbf{Eq.\eqref{lemma.unk.2}}).
	Subtracting $4p_{2,n,k} $ from both sides of \eqref{lemma.p1np1k} gives \eqref{lemma.unk.2}.
\end{proof}

\begin{lemma}
	Define both sides of \eqref{lemma.unk.2} as $u_{n,k} $,
	\begin{align}
		\label{lemma.unk}
		u_{n,k} &= p_{1,n+1,k} - 4p_{2,n,k} \\
		\label{lemma.unk.0}
		u_{n,k} &= 2(2k+3)p_{2,n,k} + p_{2,n,k-1},
	\end{align}
	then
	\begin{equation}
\label{lemma.p2np1k}
p_{2,n+1,k} = 4(k+1)u_{n,k+1} + 4p_{3,n,k} + u_{n,k}.
	\end{equation}
\end{lemma}
\begin{proof}(\textbf{Eq.\eqref{lemma.p2np1k}}).
	Because
	\begin{align*}
		4(k+1)u_{n,k+1} &= p_{2,n+1,k}\frac{2(n+1-k)((2k+5)n+4k+7)}{(n+1)(2n+7)(2k+7)}, \\
		4p_{3,n,k} &= p_{2,n+1,k}\frac{2(n+1-k)}{(n+1)(2k+7)}, \\ 
		u_{n,k} &= p_{2,n+1,k}\frac{(2k+3)(n+1)+2k}{(2n+7)(n+1)},
	\end{align*}
	we have
	\begin{equation*}
		4(k+1)u_{n,k+1} + 4p_{3,n,k} + u_{n,k} = p_{2,n+1,k} f ,
	\end{equation*}
	where
	\begin{equation*}
		\begin{split}
			&f = \\
			&
			\frac{2(n+1-k)((2k+5)n+4k+7)}{(n+1)(2n+7)(2k+7)}+\frac{2(n+1-k)}{(n+1)(2k+7)}+\frac{(2k+3)(n+1)+2k}{(2n+7)(n+1)}
		\end{split}.
	\end{equation*}
	Because
	\begin{equation*}
		\begin{split}
		&\frac{2(n+1-k)((2k+5)n+4k+7)}{(n+1)(2n+7)(2k+7)}+\frac{2(n+1-k)}{(n+1)(2k+7)}+\frac{(2k+3)(n+1)+2k}{(2n+7)(n+1)} \\
		 &= 1,	
		\end{split}
	\end{equation*}
	we have
	\begin{equation*}
		4(k+1)u_{n,k+1} + 4p_{3,n,k} + u_{n,k} = p_{2,n+1,k}.
	\end{equation*}
\end{proof}

\begin{lemma} 
\begin{equation}
	\label{lemma.gsrc.sum.1}
	\sum_{k=0}^{n+1}\frac{(-1)^k}{2^{n+1-k}}\left( p_{2,n,k-1}+4p_{1,n,k} \right)s^k
	= \sum_{k=0}^{n}\frac{(-1)^k}{2^{n-k}}\left( 2p_{1,n,k} - s \, p_{2,n,k}\right)s^k ,
\end{equation}
where $n \ge 0, s \in \mathbb{R} $.	
\end{lemma}
\begin{proof}
	We start from the left hand side of \eqref{lemma.gsrc.sum.1} and reach its right hand side.
	Because
	\begin{equation*}
		p_{2,n,k-1}|_{k=0} = 0, p_{1,n,k}|_{k=n+1}=0,
	\end{equation*}
	we have
	\begin{align*}
		\mathrm{LHS} &= \sum_{k=1}^{n+1}\frac{(-1)^k}{2^{n+1-k}}\left( p_{2,n,k-1} \right)s^k 
		+ \sum_{k=0}^{n}\frac{(-1)^k}{2^{n+1-k}}\left( 4p_{1,n,k} \right)s^k \\
	&= \sum_{k=0}^{n}\frac{(-1)^{k+1}}{2^{n-k}}\left( p_{2,n,k} \right)s^{k+1} 
	+ \sum_{k=0}^{n}\frac{(-1)^k}{2^{n-k}}\left( 2p_{1,n,k} \right)s^k \\
	&= \sum_{k=0}^{n}\frac{(-1)^k}{2^{n-k}}\left( 2p_{1,n,k} - s \, p_{2,n,k}\right)s^k = \mathrm{RHS}.
	\end{align*}
\end{proof}

\begin{lemma} Using symbols defined in Table.\ref{tbl.symbols}, we can calculate the following cross-products: 
	\begin{align}
	\label{gs.times.b}
g \vec{s} \times \vec{b} &= w \vec{c} - \vec{d} \, ,
\\
	\label{s.times.d}
\vec{s} \times \vec{d} &= 0 \, ,
\\
\label{s.times.c}
\vec{s} \times \vec{c} &= -\vec{b} \, ,
\\
	\label{gs.times.h}
g\vec{s} \times \vec{h} &= -gw\vec{z} + g\vec{c} \, ,
\\
	\label{gs.times.z}
g\vec{s} \times \vec{z} &= \vec{h} \, ,
\\
\label{gs.times.e}
g\vec{s} \times \vec{e} &= 0 \,.
	\end{align}
\end{lemma}

\begin{lemma} Using symbols defined in Table.\ref{tbl.symbols}, we can calculate the following curls:
\begin{align}
	\label{curlwb.0}
	\nabla \times \vec{b} &= 3\vec{c} \, ,
\\
	\label{curlwb.k}
	\nabla \times w^k \vec{b} &= (2k+3)w^k \vec{c} - 2kw^{k-1}\vec{d}; \forall k>0 \, ,
\\
	\label{curl.c.0}
	\nabla \times \vec{c} &= 0 \, ,
\\
	\label{curlwc.k}
	\nabla \times w^k \vec{c} &= -2kgw^{k-1} \vec{b}; \forall k > 0 \, ,
\\
	\label{curlwd}
	\nabla \times w^k \vec{d} &= 2gw^k \vec{b}; \forall k \ge 0 \, ,
\\
	\label{curl.h}
	\nabla \times \vec{h} &= -2g \vec{z} \, ,
\\
	\label{curl.wh.k}
	\nabla \times w^k\vec{h} &= -2g(k+1)w^k\vec{z} + 2gkw^{k-1} \vec{c}; k>0 \, ,
\\
	\label{curl.z}
	\nabla \times \vec{z} &= 0 \, ,
\\
	\label{curl.wz.k}
	\nabla \times w^k\vec{z} &= 2kw^{k-1}\vec{h}, \forall k > 0 \, ,
\\
	\label{curl.we.k}
	\nabla \times w^k\vec{e} &= -w^k\vec{h}, \forall k \ge 0 \, .
\end{align}
\end{lemma}

\subsection{\textbf{Theorems}}
All symbols not otherwise specified are defined in Table.\ref{tbl.symbols}.
\begin{thm} For all $F \in C^{\infty}(\mathbb{R}^3,\mathbb{R}^3)$, 
	\begin{equation}
		\label{curl.exp.a.F}
	\nabla \times	\exp(-ar^2) F = \exp(-ar^2) (\nabla\times-g\vec{s}\times) F \, .
	\end{equation}
\end{thm}
\begin{proof}
We start from the left hand side of \eqref{curl.exp.a.F}, applying curl operation:
\begin{align*}
	\mathrm{LHS} &= \nabla\times \exp(-ar^2)F = \nabla\times \begin{bmatrix}
		\exp(-ar^2)F_x \\
		\exp(-ar^2)F_y \\
		\exp(-ar^2)F_z
	\end{bmatrix} \\
	&= \exp(-ar^2)\begin{bmatrix}
		 \left(\ppf{F_z}{y} - \ppf{F_y}{z} \right) - 2a(yF_z - zF_y) \\
		 \left(\ppf{F_x}{z} - \ppf{F_z}{x} \right) - 2a(zF_x - xF_z) \\
		 \left(\ppf{F_y}{x} - \ppf{F_x}{y} \right) - 2a(xF_y - yF_x)
	\end{bmatrix} \\
	&= \exp(-ar^2)\left( \nabla\times F - 2a \vec{s} \times F \right) \\
	&= \exp(-ar^2) (\nabla\times - g \vec{s}\times)  F \,.
\end{align*}	
\end{proof}

\begin{thm} For $B=\exp(-ar^2)\vec{b} $, the curls of $B$ are given by
	\begin{align}
		\label{curl.exp.a.b.2n}
		\curl{2n} B &= \exp(-ar^2)g^n \sum_{k=0}^{n} \frac{(-1)^k}{2^{n-k}}p_{2,n,k}w^k\vec{b}
\\
		\label{curl.exp.a.b.2n1}
		\curl{2n+1} B &= \exp(-ar^2) g^n \left(\sum_{k=0}^{n+1}\frac{(-1)^k}{2^{n+1-k}}u_{n,k}w^k\vec{c} + \sum_{k=0}^{n}\frac{(-1)^k}{2^{n-k}}p_{3,n,k}w^k\vec{d} \right)
	\end{align}
	\begin{equation*}
		\forall n \ge 0
	\end{equation*}
\end{thm}
\begin{proof}
We will induct on n.

\textbf{Base case ($n=0$)}: It is easy to see that for $n=0$ (\ref{curl.exp.a.b.2n}) and (\ref{curl.exp.a.b.2n1}) hold.

\textbf{Inductive Hypothesis ($n=m$)}: Assume for some $m \in \mathbb{Z}_{\ge 0}$ (\ref{curl.exp.a.b.2n}) holds.

\textbf{Inductive Step}: Let’s first prove that (\ref{curl.exp.a.b.2n1}) holds for $n=m $. 
Applying the curl operation on both side of (\ref{curl.exp.a.b.2n}) with $n=m$ and using (\ref{curl.exp.a.F}), we have
\begin{equation*}
	\begin{split}
		&\curl{2m+1} B = \nabla \times \exp(-ar^2) g^m \sum_{k=0}^{m}\frac{(-1)^k}{2^{m-k}}p_{2,m,k}w^k\vec{b} \\
		&= \exp(-ar^2) g^m \left(\sum_{k=0}^{m}\frac{(-1)^k}{2^{m-k}}p_{2,m,k} \nabla \times w^k\vec{b} - \sum_{k=0}^{m}\frac{(-1)^k}{2^{m-k}}p_{2,m,k}w^kg\vec{s} \times\vec{b} \right) 
	\end{split}
\end{equation*}
Apply (\ref{curlwb.0}), (\ref{curlwb.k}) and (\ref{gs.times.b}) and regroup the summations, we have
\begin{equation*}
	\begin{split}
		\curl{2m+1} B &= \exp(-ar^2) g^m \sum_{k=0}^{m+1}\frac{(-1)^k}{2^{m+1-k}}(2(2k+3)p_{2,m,k}+p_{2,m,k-1})w^k\vec{c} \\
		&+ \exp(-ar^2) g^m \sum_{k=0}^{m}\frac{(-1)^k}{2^{m-k}}(4(k+1)p_{2,m,k+1}+p_{2,m,k})w^k\vec{d}
	\end{split}
\end{equation*}
By (\ref{lemma.unk.0}) and (\ref{lemma.p3nk}), we arrive at (\ref{curl.exp.a.b.2n1}) with $n=m$.

Applying the curl operation on both side of (\ref{curl.exp.a.b.2n1}) with $n=m$ and using (\ref{curl.exp.a.F}), we obtain
\begin{equation*}
	\begin{split}
		&\curl{2m+2} B \\
		&= \exp(-ar^2)g^m \left(\sum_{k=0}^{m+1}\frac{(-1)^k}{2^{m+1-k}}u_{m,k}\nabla \times w^k\vec{c} + \sum_{k=0}^{m}\frac{(-1)^k}{2^{m-k}}p_{3,m,k}\nabla \times w^k\vec{d} \right) \\
		&- \exp(-ar^2) g^m \left(\sum_{k=0}^{m+1}\frac{(-1)^k}{2^{m+1-k}}u_{m,k}w^kg\vec{s} \times \vec{c}+ \sum_{k=0}^{m}\frac{(-1)^k}{2^{m-k}}p_{3,m,k}w^kg\vec{s} \times \vec{d} \right)
	\end{split}
\end{equation*}
By (\ref{curl.c.0}), (\ref{curlwc.k}), (\ref{curlwd}), (\ref{s.times.d}) and (\ref{s.times.c}), and regroup the summations, we have
\begin{equation*}
	\curl{2m+2} B = \exp(-ar^2) g^{m+1} \sum_{k=0}^{m+1}\frac{(-1)^k}{2^{m+1-k}}(4(k+1)u_{m,k+1}+4p_{3,m,k}+u_{m,k})w^k\vec{b}
\end{equation*}
By (\ref{lemma.p2np1k}) we arrive at (\ref{curl.exp.a.b.2n}) with $n=m+1 $. 

$\therefore$ By the principle of induction, (\ref{curl.exp.a.b.2n}) and (\ref{curl.exp.a.b.2n1}) hold for all $n\in \mathbb{Z}_{\ge 0}$.

\end{proof}

\begin{thm} For $B=\exp(-ar^2)\vec{z} $, the curls of $B $ are given by
	\begin{align}
		\label{curl.exp.a.z.2n1}
		\curl{2n+1} B&=g^n \exp(-ar^2) \sum_{k=0}^{n} \frac{(-1)^{k+1}}{2^{n-k}}p_{1,n,k}w^k\vec{h}
\\
		\label{curl.exp.a.z.2n2}
		\begin{split}
		\curl{2(n+1)} B &= g^{n+1} \exp(-ar^2)\sum_{k=0}^{n}\frac{(-1)^k}{2^{n-k}}p_{2,n,k}w^k\vec{e} \\
		&+ g^{n+1} \exp(-ar^2)\sum_{k=0}^{n}\frac{(-1)^k}{2^{n-k}} \left(2p_{1,n,k} -wp_{2,n,k} \right)w^k\vec{z}
		\end{split}
	\end{align}
	\begin{equation*}
		\forall n \ge 0
	\end{equation*}
\end{thm}
\begin{proof}
We will induct on n.

\textbf{Base case ($n=0,n=1$)}: It is easy to verify that for $n=0$ and $n=1$ (\ref{curl.exp.a.z.2n1}) and (\ref{curl.exp.a.z.2n2}) hold.

\textbf{Inductive Hypothesis ($n=m$)}: Assume for some $m \in \mathbb{Z}_{> 0}$ (\ref{curl.exp.a.z.2n1}) holds.

\textbf{Inductive Step}: Let’s first prove that (\ref{curl.exp.a.z.2n2}) holds for $n=m $. 
Applying the curl operation on both sides of (\ref{curl.exp.a.z.2n1}) with $n=m$ and using (\ref{curl.exp.a.F}), we have
\begin{equation*}
	\begin{split}
	\curl{2m+2} B &= g^m\exp(-ar^2) \sum_{k=0}^{m}\frac{(-1)^{k+1}}{2^{m-k}}p_{1,m,k} \nabla \times w^k \vec{h} 
	\\
	&- g^m\exp(-ar^2)\sum_{k=0}^{m} \frac{(-1)^{k+1}}{2^{m-k}} p_{1,m,k}w^k g\vec{s}\times \vec{h} 
	\end{split}
\end{equation*}
Applying (\ref{curl.h}) , (\ref{curl.wh.k}) and (\ref{gs.times.h}), and moving the factor $g^m\exp(-ar^2) $ to the left hand side for cleanness, the above equation can be written as
\begin{equation*}
	\begin{split}
		&g^{-m}\exp(ar^2) \curl{2m+2} B =\\
		&-\frac{1}{2^m}p_{1,m,0}(-2g\vec{z}) \\
		&+ \sum_{k=1}^{m} \frac{(-1)^{k+1}}{2^{m-k}}p_{1,m,k}\left(-2g(k+1)w^k\vec{z}+2gkw^{k-1}\vec{e} \right) \\
		&- \sum_{k=0}^{m} \frac{(-1)^{k+1}}{2^{m-k}}p_{1,m,k} w^k (-gw\vec{z} + g\vec{e})
	\end{split}
\end{equation*}
The above can be rearranged into
\begin{equation}
	\label{proof.curl.exp.a.z}
	\begin{split}
		&g^{-(m+1)}\exp(ar^2) \curl{2m+2} B =\\
		&\frac{2}{2^m}p_{1,m,0}\vec{z} + (-1)^{m+1}p_{1,m,m}w^{m+1}\vec{z} \\
		&+ \sum_{k=1}^{m} \frac{(-1)^{k}}{2^{m+1-k}} \left(p_{1,m,k-1}+4(k+1)p_{1,m,k}  \right)w^k\vec{z} \\
		&+ \sum_{k=0}^{m-1} \frac{(-1)^k}{2^{m-k}} \left(p_{1,m,k} +4(k+1)p_{1,m,k+1}  \right) w^k\vec{e} +(-1)^mp_{1,m,m}w^m\vec{e}
	\end{split}
\end{equation}
(\ref{lemma.p2nk}) and (\ref{ph.n.n}) give us the following identities.
\begin{equation*}
	\begin{split}
		p_{1,m,k} + 4(k+1) p_{1,m,k+1} &=p_{2,m,k} \\
		p_{1,m,k-1} + 4(k+1) p_{1,m,k} &=p_{2,m,k-1} + 4p_{1,m,k} \\
		p_{1,m,m} &= p_{2,m,m}
	\end{split}
\end{equation*}
Applying the above identities to (\ref{proof.curl.exp.a.z}), we have
\begin{equation*}
	\begin{split}
		&g^{-(m+1)}\exp(ar^2) \curl{2m+2} B \\
		&=\frac{2}{2^m}p_{1,m,0}\vec{z} + (-1)^{m+1}p_{1,m,m}w^{m+1}\vec{z} \\
		&+ \sum_{k=1}^{m} \frac{(-1)^k}{2^{m+1-k}}\left( p_{2,m,k-1} + 4p_{1,m,k}\right)w^k\vec{z} \\
		&+ \sum_{k=0}^{m-1} \frac{(-1)^k}{2^{m-k}}p_{2,m,k}w^k\vec{e} + (-1)^mp_{2,m,m}w^m\vec{e}
	\end{split}
\end{equation*}
The above equation can be rearranged into
\begin{equation*}
	\begin{split}
		&g^{-(m+1)}\exp(ar^2) \curl{2m+2} B \\
		&= \sum_{k=0}^{m+1} \frac{(-1)^k}{2^{m+1-k}}\left( p_{2,m,k-1} + 4p_{1,m,k}\right)w^k\vec{z} \\
		&+ \sum_{k=0}^{m} \frac{(-1)^k}{2^{m-k}}p_{2,m,k}w^k\vec{e} 
	\end{split}
\end{equation*}
By (\ref{lemma.gsrc.sum.1}), we have
\begin{equation*}
	\begin{split}
		g^{-(m+1)}\exp(ar^2) \curl{2m+2} B &= \sum_{k=0}^{m}\frac{(-1)^k}{2^{m-k}}p_{2,m,k} w^k \vec{e}\\
		&+ \sum_{k=0}^{m} \frac{(-1)^k}{2^{m-k}}\left(2p_{1,m,k} - wp_{2,m,k} \right)w^k\vec{z}
	\end{split}
\end{equation*} 
The above equation shows that (\ref{curl.exp.a.z.2n2}) holds for $n=m $.

Applying the curl operation on both side of (\ref{curl.exp.a.z.2n2}) with $n=m$ and using (\ref{curl.exp.a.F}), we have
\begin{equation*}
	\begin{split}
		\curl{2(m+1)+1} B &= g^{m+1}\exp(-ar^2)(\nabla-g\vec{s}) \times \\
		&\left( \sum_{k=0}^{m}\frac{(-1)^k}{2^{m-k}}(2p_{1,m,k}-wp_{2,m,k})w^k\vec{z} + \sum_{k=0}^{m}\frac{(-1)^k}{2^{m-k}}p_{2,m,k} w^k\vec{e} \right)
	\end{split}
\end{equation*}
Moving the factor $g^{m+1}\exp(-ar^2) $ to the left hand side for cleanness, the above equation can be written as
\begin{equation*}
	\begin{split}
		&g^{-(m+1)}\exp(ar^2) \curl{2(m+1)+1} B \\
		&= \sum_{k=0}^{m}\frac{(-1)^k}{2^{m-k}} (3p_{1,m,k}\nabla \times w^k\vec{z}-p_{2,m,k}\nabla \times w^{k+1}\vec{z}) \\
		&+ \sum_{k=0}^{m} \frac{(-1)^k}{2^{m-k}}p_{2,m,k}\nabla \times w^k \vec{e} \\
		&- \sum_{k=0}^{m} \frac{(-1)^k}{2^{m-k}}(2p_{1,m,k}-wp_{2,m,k})w^k g\vec{s} \times \vec{z} \\
		&- \sum_{k=0}^{m} \frac{(-1)^k}{2^{m-k}}p_{2,m,k}w^k g\vec{s} \times \vec{e}
	\end{split}
\end{equation*}  
Applying (\ref{curl.z}), (\ref{curl.wz.k}), (\ref{curl.we.k}), (\ref{gs.times.z}) and (\ref{gs.times.e}), we have
\begin{equation*}
	\begin{split}
		&g^{-(m+1)}\exp(ar^2) \curl{2(m+1)+1} B \\
		&= \sum_{k=1}^{m} \frac{k=0}{2^{m-k}} 4kp_{1,m,k}w^{k-1} \vec{h} \\
		&- \sum_{k=0}^{m}\frac{(-1)^k}{2^{m-k}}p_{2,m,k}2(k+1)w^k\vec{h} \\
		&- \sum_{k=0}^{m}\frac{(-1)^k}{2^{m-k}}p_{2,m,k} w^k\vec{h} \\
		&- \sum_{k=0}^{m} \frac{(-1)^k}{2^{m-k}}(2p_{1,m,k}-wp_{2,m,k}) w^k\vec{h}
	\end{split}
\end{equation*}  
By (\ref{lemma.gsrc.sum.1}), the above can be written as
\begin{equation*}
	\begin{split}
		&g^{-(m+1)}\exp(ar^2) \curl{2(m+1)+1} B \\
		&= \sum_{k=1}^{m} \frac{(-1)^k}{2^{m-k}} 4kp_{1,m,k}w^{k-1} \vec{h} \\
		&- \sum_{k=0}^{m}\frac{(-1)^k}{2^{m-k}}p_{2,m,k}2(k+1)w^k\vec{h} \\
		&- \sum_{k=0}^{m}\frac{(-1)^k}{2^{m-k}}p_{2,m,k} w^k\vec{h} \\
		&- \sum_{k=0}^{m+1} \frac{(-1)^k}{2^{m+1-k}}(p_{2,m,k-1}+4p_{1,m,k}) w^k\vec{h}
	\end{split}
\end{equation*}  
Rearrange the summations into
\begin{equation}
	\label{thm.curl.exp.a.z}
	g^{-(m+1)}\exp(ar^2) \curl{2(m+1)+1} B = \sum_{k=0}^{m} \frac{(-1)^{k+1}}{2^{m+1-k}}v_{m,k}w^k\vec{h}+(-1)^mw^{m+1}\vec{h}
\end{equation}
where
\begin{equation*}
	v_{m,k}=p_{2,m,k}2(2k+3)+p_{2,m,k-1}+4(4(k+1)p_{1,m,k+1}+p_{1,m,k})
\end{equation*}
By (\ref{lemma.p2nk}), we have
\begin{equation*}
	v_{m,k} = 2(2k+5)p_{2,m,k} + p_{2,m.k-1}
\end{equation*}
By (\ref{lemma.p1np1k}),
\begin{equation*}
	v_{m,k} = p_{1,m+1.k}
\end{equation*}
Thus, (\ref{thm.curl.exp.a.z}) becomes
\begin{equation*}
	g^{-(m+1)}\exp(ar^2) \curl{2(m+1)+1} B = \sum_{k=0}^{m} \frac{(-1)^{k+1}}{2^{m+1-k}}p_{1,m+1,k}w^k\vec{h}+(-1)^mw^{m+1}\vec{h}
\end{equation*}
The above can be written as
\begin{equation*}
	g^{-(m+1)}\exp(ar^2) \curl{2(m+1)+1} B = \sum_{k=0}^{m+1} \frac{(-1)^{k+1}}{2^{m+1-k}}p_{1,m+1,k}w^k\vec{h}
\end{equation*}
The above shows that (\ref{curl.exp.a.z.2n1}) holds for $n=m+1$.

$\therefore$ By the principle of induction, (\ref{curl.exp.a.z.2n1}) and (\ref{curl.exp.a.z.2n2}) hold for all $n\in \mathbb{Z}_{\ge 0}$.
\end{proof}

\section*{Acknowledgment}
I thank Dragan Redžić and Wim Vegt for their valuable feedback and constructive comments on an earlier draft of this manuscript. I also thank the anonymous reviewers for their professional insights and constructive suggestions, which substantially improved the quality and presentation of this work.

Declaration of generative AI and AI-assisted technologies in the manuscript preparation process:

During the preparation of this manuscript, the author used ChatGPT to assist with language refinement, organization, critical discussion of mathematical exposition, and preparation of the revision. After using this tool, the author carefully reviewed and edited all content and takes full responsibility for the published manuscript.

\bibliographystyle{elsarticle-num}
\bibliography{solution3D_v2}

\end{document}